\documentclass[a4paper,11pt
]{article}
\usepackage[normalem]{ulem}

\usepackage[utf8]{inputenc}
\usepackage[UKenglish]{babel}
\usepackage[T1]{fontenc}
\usepackage{amsmath,amsfonts,amssymb,amsthm}
\usepackage{graphicx}
\usepackage[top=2cm,bottom=3cm,left=2.5cm,right=2.5cm]{geometry}
\usepackage{setspace}
\usepackage{xcolor}
\usepackage{stmaryrd}
\usepackage{enumitem}
\usepackage{esint}
\usepackage{centernot}
\usepackage[format=hang,singlelinecheck=false]{caption}

\usepackage[colorlinks=true]{hyperref}

\title{\textsc{Geodesics in planar Poisson roads random metric}}

\author{Guillaume Blanc\thanks{Université Paris-Saclay \hfill \href{mailto:guillaume.blanc1@universite-paris-saclay.fr}{guillaume.blanc1@universite-paris-saclay.fr}} \and Nicolas Curien\thanks{Université Paris-Saclay \hfill \href{mailto:nicolas.curien@gmail.com}{nicolas.curien@gmail.com}} \and Jonas Kahn \thanks{University of Electronic Science and Technology of China \hfill \href{mailto:jokahn@phare.normalesup.org}{jokahn@phare.normalesup.org}}}

\begin{document}

\maketitle

\theoremstyle{plain}
\newtheorem{prop}{Proposition}
\newtheorem{thm}{Theorem}
\newtheorem{open}{Open problem}
\newtheorem{lem}{Lemma}
\newtheorem{claim}{Claim}

\theoremstyle{remark}
\newtheorem{rem}{Remark}

\begin{abstract}
We study the structure of geodesics in the fractal random metric constructed by Kendall from a self-similar Poisson process of roads (i.e, lines with speed limits) in $\mathbb{R}^2$. 
In particular, we prove a conjecture of Kendall stating that geodesics do not pause en route, i.e, use roads of arbitrary small speed except at their endpoints.
It follows that the geodesic frame of $\left(\mathbb{R}^2,T\right)$ is the set of points on roads.
We also consider geodesic stars and hubs, and give a complete description of the local structure of geodesics around points on roads. 
Notably, we prove that leaving a road by driving off-road is never geodesic.
\end{abstract}

\begin{figure}[!h]
\centering
\begin{tabular}{lr}
\includegraphics[width=0.4\linewidth]{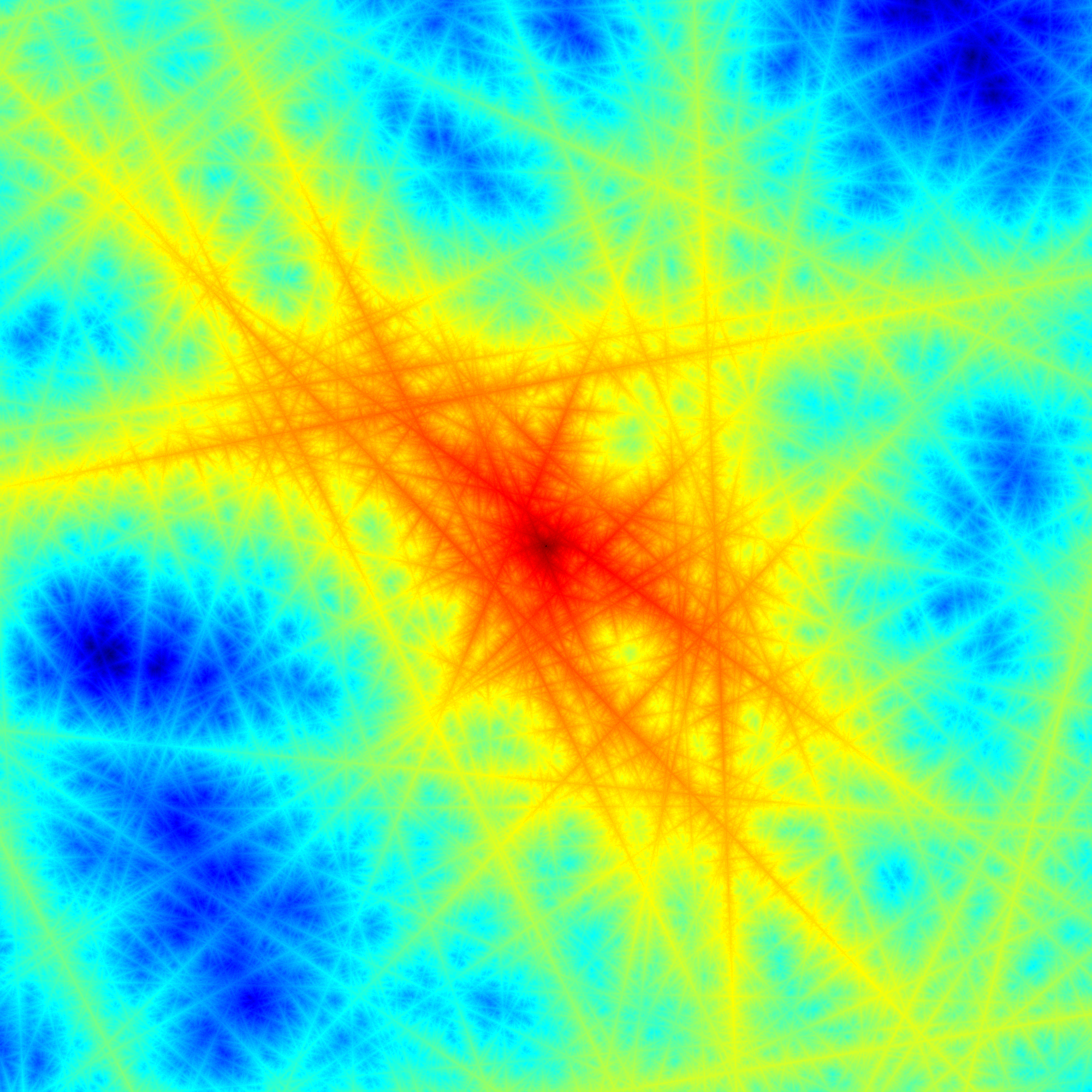}&\includegraphics[width=0.4\linewidth]{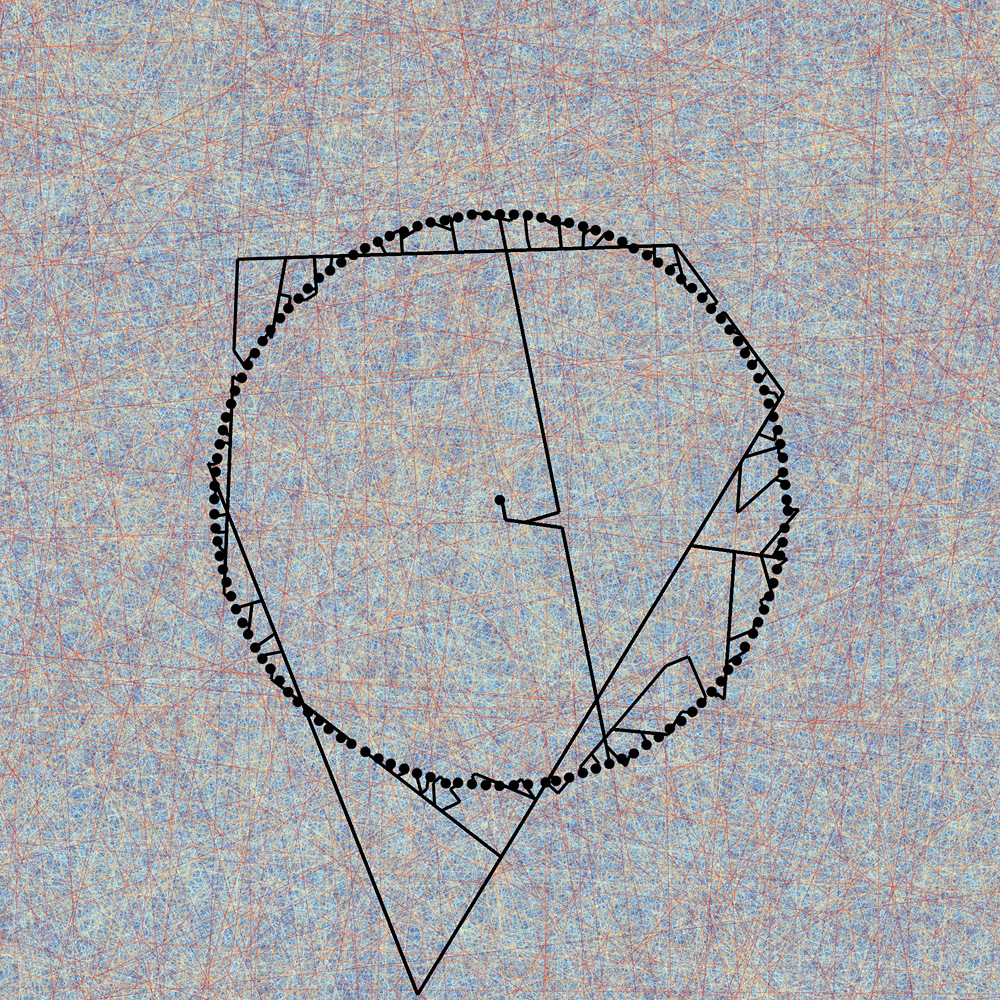}
\end{tabular}
\caption{\emph{Left:} Illustration of the fractal random metric constructed by Kendall from a self-similar Poisson process of roads (i.e, lines with speed limits) in $\mathbb{R}^2$. 
The balls of increasing radii around the origin are displayed by varying colors from blue to red.\\
\emph{Right:} The geodesic network between several points on the unit sphere and the origin.\\
\emph{Credits:} Arvind Singh.}
\end{figure}

\newpage

\section*{Introduction and main results}

\paragraph{Geodesics at the core of two-dimensional random geometry.} 
The study of geodesics has arguably become one of the most important topics in the theory of planar random metrics. 
Indeed, the pioneer  analysis of fine properties of geodesics in large random planar maps and their scaling limit the Brownian sphere was the central ingredient in the breakthrough of Le Gall \cite{legalluniqueness} and Miermont \cite{miermontscalingmap} establishing the uniqueness of the Brownian sphere. 
The strong confluence properties of geodesics also played a crucial role in the definition of the Liouville quantum gravity metric, another major success of the last decade \cite{gwynnemilleruniqueness,dingdubedatgwynneintroduction}. 
Since then, a systematic study of geodesic networks has emerged, see \cite{gwynnenetworks,dinggwynnesupercritical,gwynnemillerconfluence} in Liouville quantum gravity, \cite{angelkolesnikmiermont,millerqian,legallstars,legallgeodesics} in Brownian geometry, and \cite{bhatiaatypical,dauvergne27} in the directed landscape, a sort of ``signed directed metric'' describing the KPZ fixed point of growing interfaces \cite{dauvergneortmannviragdirected}.

\paragraph{Kendall's Poisson roads random metric.} 
In this paper, we consider the fractal random metric constructed by Kendall \cite{kendall} from a self-similar Poisson process of roads in $\mathbb{R}^2$, following a suggestion of Aldous  \cite{aldousscale} in his general investigation of scale-invariant random spatial networks (\textsc{sirsn}). 
Let us briefly present the construction:
Recall that there exists a unique, up to multiplicative constant, locally finite Borel measure $\mu$ on the space $ \mathbb{L}$ of affine lines in $ \mathbb{R}^2$ that is invariant under rotations and translations {(this measure is sometimes called the ``kinematic measure'' and is used e.g.~in classical proofs of the Cauchy--Crofton formula \cite{santalo2004integral})}.
Then, consider a Poisson process $\Pi$ with intensity proportional to $\mu\otimes v^{-\beta}\mathrm{d}v$ on $\mathbb{L}\times\mathbb{R}_+^*$, where $\beta>2$ is a parameter of the model.
Viewing each atom $(\ell,v)$ of $\Pi$ as a \textbf{road} in $\mathbb{R}^2$, with $v$ the \textbf{speed limit} on the line $\ell$, picture driving on the random road network generated by $\Pi$: this induces a random metric $T$ on $\mathbb{R}^2$, for which the distance $T(x,y)$ between points $x,y\in\mathbb{R}^2$ is given by the infimal driving time of a path that respects the speed limits from $x$ to $y$.
Since the set of points on roads, namely $\mathcal{L}=\bigcup_{(\ell,v)\in\Pi}\ell$, is almost surely of zero Lebesgue measure as a countable union of lines, this is a non-trivial result, which is due to Kendall \cite{kendall}.
In fact, the construction of Kendall holds in $\mathbb{R}^d$ for any $d\geq2$, although in dimension $d\geq3$, almost surely, the roads of $\Pi$ do not intersect.
The construction requires $\beta>d$, and produces a random metric $T$ on $\mathbb{R}^d$ with a nice self-similarity property.
Moreover, almost surely, the metric space $\left(\mathbb{R}^d,T\right)$ is homeomorphic to the usual Euclidean $\mathbb{R}^d$, and has Hausdorff dimension
\[\dim_H\left(\mathbb{R}^d,T\right)=\frac{(\beta-1)d}{\beta-d}>d,\]
as shown by the first author \cite{moa1fractal}.
See Section \ref{sec:rappelsgeo} for a more detailed presentation of the model.
Our main results in this paper concern the planar case $d=2$.

\paragraph{Geodesics do not pause en route.}
Almost surely, the metric space $\left(\mathbb{R}^d,T\right)$ is a geodesic space, in which geodesics in the usual ``metric geometry'' sense correspond to paths that respect the speed limits with minimal driving time.
Some fundamental properties of geodesics in $\left(\mathbb{R}^d,T\right)$ were already obtained by Kendall \cite{kendall}, and by the third author \cite{kahn}.
In particular, as for many other random metrics, the geodesic in $\left(\mathbb{R}^d,T\right)$ between two fixed points, for instance $0$ and ${e_1=(1,0,\ldots,0)}$, is almost surely unique (see \cite[Theorem 4.4]{kendall} for $d=2$ and \cite[Theorem 4.8]{kahn} for $d\geq3$).
Now, a specificity of the Poisson roads random metric $T$ is that it comes with an underlying ``speed limits'' function $V:\mathbb{R}^d\rightarrow\mathbb{R}_+$, which formally is defined by
\[V(x)=\sup\{v\,;\,(\ell,v)\in\Pi:\text{$\ell$ passes through $x$}\}\quad\text{for all $x\in\mathbb{R}^d$,}\]
with the convention $\sup\emptyset=0$.
Of course, since almost surely neither $0$ nor $e_{1}$ belong to $\mathcal{L}$, the geodesic from $0$ to $e_1$ must use roads with arbitrarily small speed at its endpoints, as illustrated in Figure \ref{fig:pauseintro}.
However, it was conjectured by Kendall \cite[Section 7]{kendall} that in the planar case $d=2$, almost surely, the speed of geodesics is bounded away from zero away from their endpoints, which he popularized under the aphorism ``\textit{geodesics do not pause en route}''. 
We confirm this conjecture.

\begin{thm}[Geodesics do not pause en route]\label{thm:pause}
Almost surely, for any geodesic ${\gamma:[0,\tau]\rightarrow\mathbb{R}^2}$, the following holds: for every $t_1<t_2\in{]0,\tau[}$, there exists $\varepsilon>0$ such that $V(\gamma(t))\geq\varepsilon$ for all $t\in[t_1,t_2]$.
\end{thm}

\begin{figure}[!h]
\begin{center}
\includegraphics[width=12cm]{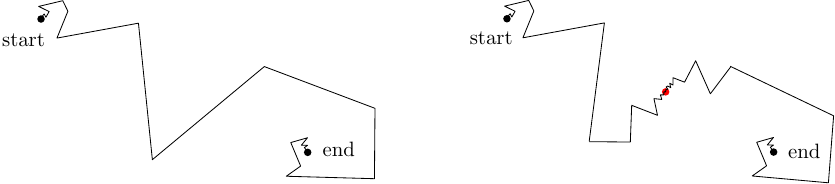}
\caption{Illustration of a geodesic that does not pause en route \emph{(left)}, and of one which pauses en route at the red point \emph{(right)}. 
Together with the results of \cite{kendall}, our Theorem \ref{thm:pause} shows that in dimension $2$, geodesics are made up of consecutive sequences of line segments, which possibly accumulate at their endpoints but not elsewhere \emph{(left)}. \label{fig:pauseintro}}
\end{center}
\end{figure}

Perhaps surprisingly, the only probabilistic inputs in our proof are confluence of geodesics properties: the rest of the argument relies on an elementary deterministic lemma.
See Section \ref{sec:pause}.
Incidentally, this provides another proof of the uniqueness of geodesics between typical points, in the planar case $d=2$ (see Remark \ref{rem:newuniqueness} below).

\paragraph{Confluence of geodesics and geodesic stars.} 
As for many other random metrics, we have some confluence of geodesics properties in $\left(\mathbb{R}^d,T\right)$, some of which were already obtained by Kendall \cite[Section 6]{kendall}, and by the third author \cite{kahn}. 
Informally, geodesics tend to use the same routes, which leads to them clumping together. 
One way to quantify this phenomenon is via the study of geodesic stars. 
Recall that in a geodesic space, a point $x$ is a geodesic star with at least $k$ arms, or \textbf{$k^+$-star} for short, if there exists $k$ disjoint geodesics emanating from $x$. 
We say that $x$ is a  \textbf{$k$-star} if it is a $k^+$-star but not a $(k+1)^+$-star. 
In particular, every point in a non-trivial geodesic space is a $1^+$-star, but being a $1$-star entails a local confluence of geodesics property.
Following up on the arguments of Kendall \cite{kendall}, it can be proved that almost surely, the point $0$ is a $1$-star in $\left(\mathbb{R}^d,T\right)$, in general dimension $d\geq2$.
Together with the uniqueness of geodesics between typical points, this entails that almost surely, for every $y_1,y_2\in\left.\mathbb{R}^d\middle\backslash\{0\}\right.$, and for any geodesics $\gamma_1:[0,T(0,y_1)]\rightarrow\mathbb{R}^d$ from $0$ to $y_1$ and $\gamma_2:[0,T(0,y_2)]\rightarrow\mathbb{R}^d$ from $0$ to $y_2$, there exists $t_0>0$ such that $\gamma_1(t)=\gamma_2(t)$ for all $t\in{[0,t_0]}$.
This was already obtained by Kendall \cite{kendall} in the planar case $d=2$.
We present this and other confluence results in Section \ref{sec:confluence}.
See Proposition \ref{prop:confluence} therein for a summary.
On the geodesic stars, we also prove that there exists a deterministic integer $m\in\mathbb{N}^*$ such that almost surely, there are no $(m+1)^{+}$-stars in $\left(\mathbb{R}^{d},T\right)$.
See Proposition \ref{prop:K+stars} below.
The maximal number of arms of a geodesic star {(i.e, the smallest integer $m$ for which there are no $(m+1)^+$-stars)} is known to be $4$ in the directed landscape \cite{dauvergneortmannviragdirected}, and remains to be arbitrated between $4$ and $5$ in Brownian geometry \cite{millerqian,legallstars}.

\paragraph{Geodesic frame and geodesic hubs.}
By definition, the \textbf{geodesic frame} of $\left(\mathbb{R}^d,T\right)$ is the union of all the geodesics minus their endpoints.
In fact, due to the confluence of geodesics, almost surely, this union can be reduced to a union over geodesics between countably many typical points (see Proposition \ref{prop:approxgeotyp} below). 
In particular, almost surely, the geodesic frame of $\left(\mathbb{R}^d,T\right)$ has Hausdorff dimension $1$, the dimension of a single geodesic. The same phenomenon appears in Brownian geometry \cite[Theorem 1.7 and Corollary 1.8]{millerqian}, and in the directed landscape \cite[Lemma 3.3]{dauvergne27}. 
A striking difference with our setting however is that in our case, geodesics are automatically Lipschitz continuous with respect to the Euclidean metric, whereas they are fractal objects in Liouville quantum gravity, as shown in \cite{fangoswamiroughness}.
Now, for the rest of this discussion, we consider the planar case $d=2$.
In our second main result (Theorem \ref{thm:frame} below), we identify the geodesic frame of $\left(\mathbb{R}^2,T\right)$ as the set of points on roads $\mathcal{L}$, and we give a complete description of the local structure of geodesics around points on roads.
Before stating this result, let us introduce the notion of \textbf{geodesic hub}, which comes up naturally in this context.
We say that a point $x$ in a geodesic space $X$ is a geodesic hub with at least $k$ arms, or \textbf{$k^+$-hub} for short, if there exists $k$ geodesics $\gamma_i:[0,\tau_i]\rightarrow X$ with $\gamma_i(0)=x$ and $\gamma_i{]0,\tau_i]}\cap\gamma_j{]0,\tau_j]}=\emptyset$ for every $i\neq j$, such that for each $i\neq j$, following $\gamma_i$ from $\gamma_i(\tau_i)$ to $x$ and then $\gamma_j$ from $x$ to $\gamma_j(\tau_j)$ yields a geodesic.
See Figure \ref{fig:hubsintro} below for an illustration.

\begin{figure}[!h]
\begin{center}
\includegraphics[width=7cm]{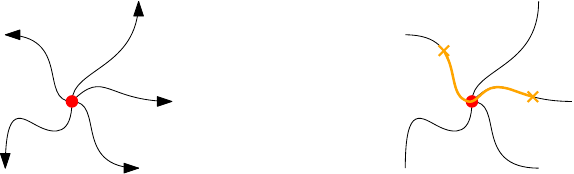}
\caption{\emph{Left:} A $5^+$-star, i.e, a point from which (at least) $5$ disjoint geodesics emanate.\\ 
\emph{Right:} This $5^+$-star is furthermore a $5^+$-hub if the concatenation of any two of these $5$ geodesics remains a geodesic.  
\label{fig:hubsintro}}
\end{center}
\end{figure}

In particular, the set of $2^+$-hubs corresponds to the geodesic frame.
Note that a $k^+$-hub is a $k^+$-star, but being a hub is much stronger than being a star.
For instance, in the Brownian sphere, it is known that the set of $3^+$-stars has Hausdorff dimension $2$ \cite{millerqian,legallstars}, while we believe that there are no $3^+$-hubs, nor in \textsc{Lqg} random surfaces.
In contrast with the case of Brownian geometry, we prove the following result, which gives a complete characterisation of geodesic hubs in $\left(\mathbb{R}^2,T\right)$.
We denote by ${\mathcal{I}=\bigcup_{(\ell_1,v_1)\neq(\ell_2,v_2)\in\Pi}\ell_1\cap\ell_2}$ the set of intersections of two roads of $\Pi$.
Almost surely, it is a countable union of points, and there is no intersection of three roads of $\Pi$.

\begin{thm}[On the geodesic frame and hubs]\label{thm:frame}
Almost surely, the geodesic frame of $\left(\mathbb{R}^2,T\right)$ is the set of points on roads $\mathcal{L}$.
Furthermore, for every $x\in\mathcal{L}$, the following holds:
\begin{itemize}
\item if $x\notin\mathcal{I}$, then $x$ is a $2^+$-hub but not a $3^+$-star,
\item if $x\in\mathcal{I}$, then $x$ is a $4^+$-hub but not a $5^+$-star.
\end{itemize}
In particular, if we denote by $\mathcal{H}_k^+$ the set of $k^+$-hubs in $\left(\mathbb{R}^2,T\right)$, then we have
\[\mathcal{H}_1^+=\mathbb{R}^d,\quad\mathcal{H}_2^+=\mathcal{L},\quad\mathcal{H}_3^+,\mathcal{H}_4^+=\mathcal{I},\quad\text{and}\quad\mathcal{H}_5^+=\emptyset.\]
\end{thm}
{The fact that a typical point on a road is a $2^+$-hub but not a $3^+$-star, and the fact that the intersection point of two roads is a $4^+$-hub but not a $5^+$-star, merely follow from the confluence of geodesics properties in Proposition \ref{prop:confluence}.
The hard part is to prove that there is no exceptional point on a road $(\ell_0,v_0)\in\Pi$ from which one can start a geodesic that immediately leaves $(\ell_0,v_0)$ without using an intersection with another road of $\Pi$.
This is a strong form of Kendall's \cite[Theorem 4.3]{kendall}, in which the author assumes that the geodesic uses $(\ell_0,v_0)$ before leaving it.}

Finally, we start the investigation of the \textbf{cut locus} of $0$ in $\left(\mathbb{R}^2,T\right)$, that is, the set $\mathcal{C}_0$ of points $x\in\mathbb{R}^2$ for which there exist {(at least)} two distinct geodesics from $x$ to $0$.
Adapting an argument of Gwynne \cite[Section 3.1]{gwynnenetworks}, we show that there are at most countably many points from which there exist {(at least)} three distinct geodesics to the origin. 
The existence of such points yields an example of a $3^+$-star which is not on the geodesic frame. 
Actually, we suspect that $\mathcal{C}_0$ has Euclidean Hausdorff dimension strictly greater than $1$ (see \cite{aizenmanburchard} and \cite{moa2coloring} for a general criterion ensuring this), but leave it open for further research:
 
\begin{open}\label{op:}
Prove that the cut locus $\mathcal{C}_0$ of $0$ is dense and path-connected, does not contain any non-trivial cycle, and that there exists $\varepsilon>0$ such that almost surely, we have 
\[\mathrm{dim}_H(\mathcal{C}_0,|\cdot|)\geq1+\varepsilon.\]
\end{open}

%

\paragraph{Acknowledgements.} The first two authors were supported by SuPerGRandMa, the ERC Consolidator Grant no $101087572$. 
We thank Jean-François Le Gall and Grégory Miermont for {insightful} discussions about the presence of $3^+$-hubs in the Brownian sphere.

\tableofcontents

\section{Introduction to Kendall's Poisson roads random metric}\label{sec:rappelsgeo}

In this section, we present the model in detail.
Throughout the paper, we let $\Pi$ be a Poisson process with intensity measure proportional to $\mu\otimes v^{-\beta}\mathrm{d}v$ on $\mathbb{L}\times\mathbb{R}_+^*$, where $\beta>d$ is a parameter of the model, and $\mu$ is the invariant measure on the space $\mathbb{L}$ of affines lines in $\mathbb{R}^d$.
Then, we consider the ``driving time'' random metric $T$ constructed by Kendall, viewing each atom $(\ell,v)$ of $\Pi$ as a road in $\mathbb{R}^d$, with $v$ the speed limit on the line $\ell$.
Let us recall some elements of the construction of these objects.

\paragraph{The invariant measure on the space of lines.}
A convenient way to describe $\mathbb{L}$ and $\mu$ reads as follows.
Choose a reference line, say $\ell_0=\mathbb{R}\times\{0\}^{d-1}$, denote by $\ell_0^\perp=\{0\}\times\mathbb{R}^{d-1}$ the orthogonal hyperplane, and consider the surjective mapping
\[\begin{matrix}
\Phi:&\ell_0^\perp\times\mathbf{SO}\left(\mathbb{R}^d\right)&\longrightarrow&\mathbb{L}\\
&(w,g)&\longmapsto&g(w+\ell_0).
\end{matrix}\]
The space $\mathbb{L}$ is endowed with the finest topology for which $\Phi$ is continuous, and with the corresponding Borel $\sigma$-algebra.
The measure $\mu$ is defined as the pushforward of the product of the $(d-1)$-dimensional Lebesgue measure on $\ell_0^\perp$ times the Haar probability measure on $\mathbf{SO}\left(\mathbb{R}^d\right)$.
The measure $\mu$ is invariant under rotations and translations, and self-similar under scaling.
More precisely, for every $x\in\mathbb{R}^d$, $r>0$ and $h\in\mathbf{SO}\left(\mathbb{R}^d\right)$, the pushforward of $\mu$ by the map $f:\ell\in\mathbb{L}\mapsto x+r\cdot h(\ell)$ is the measure $\mu\circ f^{-1}=r^{-(d-1)}\cdot\mu$.
For a compact subset $K\subset\mathbb{R}^d$, we denote by $\langle K\rangle=\{\ell\in\mathbb{L}:\ell\cap K\neq\emptyset\}$ the set of lines that hit $K$.
Writing ${\overline{B}(x,r)=\left\{y\in\mathbb{R}^d:|x-y|\leq r\right\}}$ for the closed Euclidean ball centered at $x\in\mathbb{R}^d$ with radius $r>0$, we have 
\begin{equation}\label{eq:normmugeo}
\mu\left\langle\overline{B}(x,r)\right\rangle=\upsilon_{d-1}\cdot r^{d-1},
\end{equation}
where
\[\upsilon_s=\frac{\pi^{s/2}}{\Gamma(s/2+1)}\]
is the Lebesgue measure of the unit Euclidean ball in $\mathbb{R}^s$.
In fact, the measure $\mu$ is the unique Borel measure on $\mathbb{L}$ which is invariant under rotations and translations such that \eqref{eq:normmugeo} holds (we refer to \cite{stochintgeo}).

\paragraph{The Poisson process of roads $\Pi$.}
Fix a parameter $\beta>d$, and let $\Pi$ be a Poisson process with intensity measure ${\nu=c\cdot\mu\otimes v^{-\beta}\mathrm{d}v}$ on $\mathbb{L}\times\mathbb{R}_+^*$, where the normalising constant $c=\upsilon_{d-1}^{-1}\cdot(\beta-1)$ is conveniently chosen so that for every $x\in\mathbb{R}^d$, $r>0$ and $v_0\in\mathbb{R}_+^*$, the parameter of the Poisson random variable 
\[\Pi\left\{(\ell,v)\in\mathbb{L}\times\mathbb{R}_+^*:\text{$\ell$ hits $\overline{B}(x,r)$ and $v\geq v_0$}\right\}\]
equals
\begin{equation}\label{eq:normgeo}
c\cdot\mu\left\langle\overline{B}(x,r)\right\rangle\cdot\int_{v_0}^\infty v^{-\beta}\mathrm{d}v=r^{d-1}\cdot v_0^{-(\beta-1)}.
\end{equation}
By construction, the process $\Pi$ has the following invariance property: for every $x\in\mathbb{R}^d$ and $r>0$, we have the equality in distribution
\begin{equation}\label{eq:selfsimilarPi}
\Pi\circ f_{x,r}^{-1}=\left\{f_{x,r}(\ell,v)\,;\,(\ell,v)\in\Pi\right\}\overset{\text{\scriptsize law}}{=}\Pi,
\end{equation}
where
\begin{equation}\label{eq:defscaling}
\begin{matrix}
f_{x,r}:&\mathbb{L}\times\mathbb{R}_+^*&\longrightarrow&\mathbb{L}\times\mathbb{R}_+^*\\
&(\ell,v)&\longmapsto&\left(x+r\cdot\ell,r^{(d-1)/(\beta-1)}\cdot v\right).
\end{matrix}
\end{equation}
{We could also add a rotation $h\in\mathbf{SO}\left(\mathbb{R}^d\right)$ to the picture, but we will only need to make explicit reference to the $f_{x,r}$ maps.}
Viewing each atom $(\ell,v)$ of $\Pi$ as a road in $\mathbb{R}^d$, with $v$ the speed limit on the line $\ell$, picture the process $\Pi$ as generating a random road network in $\mathbb{R}^d$.
This road network has the following properties:
\begin{itemize}
\item Almost surely, for every $x\in\mathbb{R}^d$ and $r>0$, the number of roads of $\Pi$ with speed at least $v_0$ that pass through $\overline{B}(x,r)$ is finite for each $v_0\in\mathbb{R}_+^*$, and goes to $\infty$ as $v_0\to0$,

\item Almost surely, for every $v\in\mathbb{R}_+^*$, there is at most one road of $\Pi$ with speed $v$,

\item For each $x\in\mathbb{R}^d$, almost surely, there is no road of $\Pi$ that passes through $x$,

\item In dimension $d\geq3$, almost surely, the roads of $\Pi$ do not intersect, 

\item In the planar case $d=2$, almost surely, no three roads of $\Pi$ intersect. 
\end{itemize}
The last four points are easily checked using the (multivariate) Mecke formula (see, e.g, \cite[Theorem 4.4]{lastpenrose}).

\paragraph{Construction of the ``driving time'' random metric $T$.}
The idea is to drive on the random road network generated by $\Pi$, and to consider the random metric $T$ on $\mathbb{R}^d$ for which the distance between points $x,y\in\mathbb{R}^d$ is given by the infimal driving time of a path that respects the speed limits from $x$ to $y$.
To give a proper definition to this metric, let $V:\mathbb{R}^d\rightarrow\mathbb{R}_+$ be the (random) ``speed limits'' function, defined by:
\[V(x)=\sup\{v;\,(\ell,v)\in\Pi:\text{$\ell$ passes through $x$}\}\quad\text{for all $x\in\mathbb{R}^d$,}\]
with the convention $\sup\emptyset=0$.
If no road of $\Pi$ passes through $x$, then $V(x)=0$, otherwise $x$ lies on one or two roads of $\Pi$, and $V(x)$ corresponds to the maximal speed limit between these roads.
Given a realisation of $\Pi$, we call $V$-path any continuous path $\gamma:[0,\tau]\rightarrow\mathbb{R}^d$ that respects the speed limits set by $V$, i.e, such that
\[|\gamma(t_2)-\gamma(t_1)|\leq\int_{t_1}^{t_2}V(\gamma(t))\mathrm{d}t\quad\text{for all $t_1\leq t_2\in[0,\tau]$.}\]
The seminal result of Kendall \cite[Theorem 3.6]{kendall} states that this definition is not vacuous: almost surely, for every $x,y\in\mathbb{R}^d$, there exists a $V$-path $\gamma:[0,\tau]\rightarrow\mathbb{R}^d$ from $x$ to $y$.
Moreover, the (random) function
\begin{equation}\label{eq:defT}
\begin{matrix}
T:&\mathbb{R}^d\times\mathbb{R}^d&\longrightarrow&\mathbb{R}_+\\
&(x,y)&\longmapsto&\inf\left\{\tau>0:\text{there exists a $V$-path $\gamma:[0,\tau]\rightarrow\mathbb{R}^d$ from $x$ to $y$}\right\}
\end{matrix}
\end{equation}
defines a metric on $\mathbb{R}^d$.

\subparagraph{Geodesic $V$-paths.}
In fact, the infimum in \eqref{eq:defT} is a minimum: almost surely, for every $x,y\in\mathbb{R}^d$, there exists a $V$-path $\gamma:[0,\tau]\rightarrow\mathbb{R}^d$ with driving time $\tau=T(x,y)$.
We call any such $V$-path a geodesic $V$-path.
Geodesic $V$-paths correspond to geodesics in the metric space $\left(\mathbb{R}^d,T\right)$, i.e, to paths $\gamma:[0,\tau]\rightarrow\mathbb{R}^d$ such that $T(\gamma(t_1),\gamma(t_2))=|t_1-t_2|$ for all $t_1,t_2\in[0,\tau]$.
It was shown by Kendall \cite{kendall} (in the planar case $d=2$) and by the third author \cite{kahn} (in dimension $d\geq3$) that for each $x,y\in\mathbb{R}^d$, almost surely, there exists a unique geodesic $V$-path from $x$ to $y$.
We shall refer to this fundamental property as the ``uniqueness of geodesics between typical points''.
Finally, let us state some basic properties of geodesic $V$-paths that we will use throughout the paper.
Given a $V$-path $\gamma:[0,\tau]\rightarrow\mathbb{R}^d$ and a road $(\ell,v)\in\Pi$, we say that $\gamma$ uses $(\ell,v)$ if there exists $t_1<t_2\in[0,\tau]$ such that $\gamma[t_1,t_2]\subset\ell$, with $\gamma(t_1)\neq\gamma(t_2)$.
Almost surely, for any geodesic $V$-path $\gamma:[0,\tau]\rightarrow\mathbb{R}^d$, the following holds, where $(\ell_1,v_1),(\ell_2,v_2),\ldots$ are the roads of $\Pi$ used by $\gamma$, with $v_1>v_2>\ldots$.
By the triangle inequality, the set of times $\{t\in[0,\tau]:\gamma(t)\in\ell_1\}$ is an interval, say $[a,b]\subset[0,\tau]$, and we have $\gamma(t)=\gamma(a)+t/v_1\cdot(\gamma(b)-\gamma(a))$ for all $t\in[a,b]$.
By the same argument, the sets of times $\{t\in[0,a]:\gamma(t)\in\ell_2\}$ and $\{t\in[b,\tau]:\gamma(t)\in\ell_2\}$ are intervals, etc.
Iterating the argument shows that {for each $i$, the set $\gamma[0,\tau]\cap\ell_i$ is a finite union of line segments, among which at most $2^i$ are non-trivial.}
Finally, {it can be shown that} the set of times $\{t\in[0,\tau]:\gamma(t)\in\mathcal{L}\}$ is of full Lebesgue measure, where $\mathcal{L}=\bigcup_{(\ell,v)\in\Pi}\ell$ is the set of points on roads.
In particular, we have the equalities
\[\tau=\sum_{(\ell,v)\in\Pi}\frac{\mathrm{length}(\gamma[0,\tau]\cap\ell)}{v}\quad\text{and}\quad\mathrm{Length}(\gamma)=\sum_{(\ell,v)\in\Pi}\mathrm{length}(\gamma[0,\tau]\cap\ell).\]
Here, the quantity $\mathrm{length}(\gamma[0,\tau]\cap\ell)$ is the one-dimensional Lebesgue measure of the finite union of line segments $\gamma[0,\tau]\cap\ell$, whereas $\mathrm{Length}(\gamma)$ is the length of the Lipschitz-continuous path $\gamma$.
Finally, in the planar case $d=2$, Kendall \cite[Theorem 4.3]{kendall} proved that any non-trivial line segment within a geodesic is immediately followed by another non-trivial line segment. 
However, those segments could possibly accumulate inside a geodesic (this must be the case at the endpoints of the geodesic, if they do not belong to $\mathcal{L}$), see Figure \ref{fig:pauseintro}. 
Kendall conjectured that this is not the case, and we shall establish this in Theorem \ref{thm:pause}.

\section{Confluence of geodesics}\label{sec:confluence}

In this section, we establish some confluence of geodesics properties in $\left(\mathbb{R}^d,T\right)$, which we summarise in the following proposition. 
Some of these were already known in the planar case $d=2$, and our mere input here is to present the technical details of the construction in general dimension $d\geq2$, and to derive a few useful corollaries on the structure of geodesics, following \cite{dauvergne27}.
See \cite{gwynnemillerconfluence,gwynnenetworks,millerqian,dauvergne27} for similar results in Liouville quantum gravity, in Brownian geometry or in the directed landscape.

\begin{figure}[!ht]
\begin{center}
\includegraphics[width=15cm]{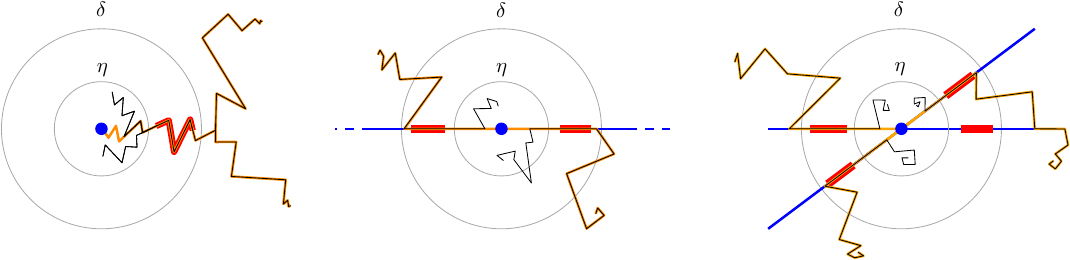}
\caption{Illustration of the confluence of geodesics properties: at and near a typical point \emph{(left)}, at and near a typical point on a road \emph{(centre)}, and at and near the intersection point of two roads \emph{(right)}.
In each case, the target point $x$ is in blue.
For each $\delta>0$, there exists (a random) $\eta\in{]0,\delta[}$ such that any geodesic from a point inside $\overline{B}(x,\eta)$ to a point outside $\overline{B}(x,\delta)$ must pass through one of the red regions.
Moreover, geodesics (in orange) from points outside $\overline{B}(x,\delta)$ to $x$ coalesce before hitting $\overline{B}(x,\eta)$.}
\end{center}
\end{figure}

\begin{prop}\label{prop:confluence}
In $\left(\mathbb{R}^d,T\right)$, we have the following confluence of geodesics properties.
\begin{enumerate}
\item\label{item:confluencetyp} Fix $x\in\mathbb{R}^d$.
Almost surely, 
\begin{enumerate}
\item For every $\delta>0$, there exists ${\eta\in{]0,\delta[}}$ and a cut point $z\in\left.\overline{B}(x,\delta)\middle\backslash\overline{B}(x,\eta)\right.$ such that any geodesic $V$-path from a point inside $\overline{B}(x,\eta)$ to a point outside $\overline{B}(x,\delta)$ must pass through $z$,

\item For every $y_1,y_2\in\left.\mathbb{R}^d\middle\backslash\{x\}\right.$, and for any geodesic $V$-paths ${\gamma_1:[0,T(x,y_1)]\rightarrow\mathbb{R}^d}$ from $x$ to $y_1$ and ${\gamma_2:[0,T(x,y_2)]\rightarrow\mathbb{R}^d}$ from $x$ to $y_2$, there exists $t_0>0$ such that ${\gamma_1(t)=\gamma_2(t)}$ for all $t\in[0,t_0]$.
\end{enumerate}

\item\label{item:confluenceroads} Almost surely, for each road $(\ell,v)\in\Pi$, the following holds for almost every point $x\in\ell$: 
\begin{enumerate}
\item For each $\delta>0$, there exists $\eta\in{]0,\delta[}$ such that any geodesic $V$-path from a point inside $\overline{B}(x,\eta)$ to a point outside $\overline{B}(x,\delta)$ must use $(\ell,v)$ within $\overline{B}(x,\delta)$,

\item For any geodesic $V$-path $\gamma:[0,\tau]\rightarrow\mathbb{R}^d$ such that $\gamma(0)=x$, there exists ${t_0\in{]0,\tau]}}$ such that $\gamma(t)\in\ell$ for all $t\in[0,t_0]$.
\end{enumerate}

\item\label{item:confluenceintersections} In the planar case $d=2$, almost surely, for each pair of roads $(\ell_1,v_1)\neq(\ell_2,v_2)\in\Pi$, the following holds, where $x$ is the intersection point of $\ell_1$ and $\ell_2$:
\begin{enumerate}
\item For every $\delta>0$, there exists $\eta\in{]0,\delta[}$ such that any geodesic $V$-path from a point inside $\overline{B}(x,\eta)$ to a point outside $\overline{B}(x,\delta)$ must use $(\ell_1,v_1)$ or $(\ell_2,v_2)$ within $\overline{B}(x,\delta)$,

\item For any geodesic $V$-path $\gamma:[0,\tau]\rightarrow\mathbb{R}^2$ such that $\gamma(0)=x$, there exists ${t_0\in{]0,\tau]}}$ such that $\gamma(t)\in\ell_1$ for all $t\in[0,t_0]$ or $\gamma(t)\in\ell_2$ for all $t\in[0,t_0]$.
\end{enumerate}
\end{enumerate}
\end{prop}

The first item is proved in Subsection \ref{subsec:confluencetyp}, and the second and third items are proved in Subsection \ref{subsec:confluenceroads}. 
It has already been noticed in planar random geometry that confluence of geodesics properties can be used to derive a very useful approximation of geodesics result by geodesics between typical points: see \cite[Lemma 3.3]{dauvergne27}, or \cite[Theorem 1.7]{millerqian}. 
For the convenience of the reader, we provide a quick proof in our setting.

\begin{lem}\label{lem:approxgeotyp}
Let $D$ be a countable and dense subset of $\mathbb{R}^d$.
Almost surely, for every geodesic $V$-path $\gamma:[0,\tau]\rightarrow\mathbb{R}^d$, the following holds: for every $t_1<t_2\in{]0,\tau[}$, there exists $y_1\neq y_2\in D$ and $t_1'<t_2'\in{]0,T(y_1,y_2)[}$ such that $\gamma[t_1,t_2]=\gamma_{y_1,y_2}[t_1',t_2']$, where $\gamma_{y_1,y_2}:[0,T(y_1,y_2)]\rightarrow\mathbb{R}^d$ denotes the unique geodesic $V$-path from $y_1$ to $y_2$.
\end{lem}
\begin{proof}
First, recall that almost surely, for each road $(\ell,v)\in\Pi$, the following property $P_{(\ell,v)}(x)$ holds for almost every $x\in\ell$: ``for each $\delta>0$, there exists $\eta\in{]0,\delta[}$ such that any geodesic $V$-path from a point inside $\overline{B}(x,\eta)$ to a point outside $\overline{B}(x,\delta)$ must use $(\ell,v)$ within $\overline{B}(x,\delta)$''.
This is item \ref{item:confluenceroads} of Proposition \ref{prop:confluence}.
Now, fix a typical realisation of $\Pi$, and let $\gamma:[0,\tau]\rightarrow\mathbb{R}^d$ be a geodesic $V$-path.
Then, fix ${t_1<t_2\in{]0,\tau[}}$, and let us prove that there exists $y_1\neq y_2\in D$ and $t_1'<t_2'\in{]0,T(y_1,y_2)[}$ such that $\gamma[t_1,t_2]=\gamma_{y_1,y_2}[t_1',t_2']$.
\begin{itemize}
\item Since the $V$-path $\gamma$ travels an Euclidean distance $|\gamma(t_1)-\gamma(0)|$ between times $0$ and $t_1$, it must use a road of $\Pi$ with speed at least $|\gamma(t_1)-\gamma(0)|/t_1$, say $(\ell_1,v_1)$: there exists $a_1\neq b_1\in\ell_1$ such that before time $t_1$, the $V$-path $\gamma$ drives from $a_1$ to $b_1$ using the road $(\ell_1,v_1)$.
Then, we can fix a point $x_1\in{]a_1,b_1[}$ for which $P_{(\ell_1,v_1)}(x_1)$ holds.

\item Similarly, since the $V$-path $\gamma$ travels an Euclidean distance $|\gamma(\tau)-\gamma(t_2)|$ between times $t_2$ and $\tau$, it must use a road of $\Pi$ with speed at least $|\gamma(\tau)-\gamma(t_2)|/(\tau-t_2)$, say $(\ell_2,v_2)$: there exists $a_2\neq b_2\in\ell_2$ such that after time $t_2$, the $V$-path $\gamma$ drives from $a_2$ to $b_2$ using the road $(\ell_2,v_2)$.
Then, we can fix a point $x_2\in{]a_2,b_2[}$ for which $P_{(\ell_2,v_2)}(x_2)$ holds.
\end{itemize}
Next, let $\delta\in{]0,|x_1-x_2|/2[}$ be small enough so that 
\[\ell_1\cap\overline{B}(x_1,\delta)\subset[a_1,b_1]\quad\text{and}\quad\ell_2\cap\overline{B}(x_2,\delta)\subset[a_2,b_2].\]
By $P_{(\ell_1,v_1)}(x_1)$ and $P_{(\ell_2,v_2)}(x_2)$, there exists $\eta\in{]0,\delta[}$ such that:
\begin{itemize}
\item Any geodesic $V$-path from a point inside $\overline{B}(x_1,\eta)$ to a point outside $\overline{B}(x_1,\delta)$ must use $(\ell_1,v_1)$ within $\overline{B}(x_1,\delta_1)$,

\item Any geodesic $V$-path from a point inside $\overline{B}(x_2,\eta)$ to a point outside $\overline{B}(x_2,\delta)$ must use $(\ell_2,v_2)$ within $\overline{B}(x_2,\delta_2)$.
\end{itemize}
Finally, fix points $y_1\in\overline{B}(x_1,\eta)\cap D$ and $y_2\in\overline{B}(x_2,\eta)\cap D$, and let $\gamma_{y_1,y_2}$ be the unique geodesic $V$-path from $y_1$ to $y_2$.
By the confluence properties above, there exist points ${z_1\in[a_1,b_1]}$ and $z_2\in[a_2,b_2]$ such that $\gamma_{y_1,y_2}$ passes through $z_1$ and $z_2$.
By the uniqueness of $\gamma_{y_1,y_2}$, the $V$-path $\gamma$ must agree with $\gamma_{y_1,y_2}$ between $z_1$ and $z_2$.
Letting $t_1'<t_2'\in{]0,T(y_1,y_2)[}$ be defined by $\gamma_{y_1,y_2}(t_1')=\gamma(t_1)$ and $\gamma_{y_1,y_2}(t_2')=\gamma(t_2)$, we obtain the result of the lemma.
\end{proof}

{It is also useful to record the following result on the structure of geodesic networks between a typical point and a given point, which in planar random geometry is a well-known consequence of the confluence of geodesics properties: see \cite[Section 1.3]{angelkolesnikmiermont} in Brownian geometry, or \cite[Lemma 1.6]{gwynnenetworks} in Liouville quantum gravity.
We omit the proof, since it uses the same arguments as above: confluence of geodesics properties (Proposition \ref{prop:confluence}) and uniqueness of geodesics between typical points.

\begin{lem}\label{lem:nobanana}
Fix $x\in\mathbb{R}^d$.
Almost surely, for every point $y\in\left.\mathbb{R}^d\middle\backslash\{x\}\right.$, and for any geodesic $V$-path ${\gamma:[0,T(x,y)]\rightarrow\mathbb{R}^d}$ from $x$ to $y$, the following holds: for each $t\in{]0,T(x,y)[}$, the restriction of $\gamma$ to $[0,t]$ is the unique geodesic $V$-path from $x$ to $\gamma(t)$.

In particular, for every point $y\in\left.\mathbb{R}^d\middle\backslash\{x\}\right.$, and for any two distinct geodesic $V$-paths 
\[\gamma_1,\gamma_2:[0,T(x,y)]\rightarrow\mathbb{R}^d\]
from $x$ to $y$, there exists $t_0\in{]0,T(x,y)[}$ such that $\gamma_1(t)=\gamma_2(t)$ for all $t\in[0,t_0]$, and $\gamma_1{]t_0,T(x,y)[}\cap\gamma_2{]t_0,T(x,y)[}=\emptyset$.
\end{lem}}

\subsection{General strategy}

The general strategy which is common to every item of Proposition \ref{prop:confluence} is to construct a ``master'' good event $G$ that has positive probability, on which the desired confluence of geodesics property holds at a fixed scale around some target point, say $x$.
See Lemma \ref{lem:confluencetypmaster} and Figure \ref{fig:confluencetyp} below for {a glimpse of} the master good event $G$ in the case of confluence of geodesics around a typical point. 
{By the invariance in distribution of $\Pi$ under translations and scaling, at each scale when zooming in around $x$, the confluence event has a positive probability of occurring.
Although the scales are not independent, by ergodicity we obtain that almost surely, there is a positive proportion of scales for which the confluence event occurs.}

Let us present the basic tools that we use to construct the master good events $G$.
The first one is essentially a mix between \cite[Theorem 3.1]{kahn} of the third author and \cite[Proposition 1.3]{moa1fractal} of the first author. It states that {away from some point $z$, the ``driving time'' random metric induced by the roads of $\Pi$ with speed less than $1$ that do not pass through some small ball around $z$ already has the same regularity as the vanilla random metric $T$.}

\begin{lem}\label{lem:confluencetool1}
Fix $z\in B(0,1)$ and $r>0$ such that $\overline{B}(z,r)\subset B(0,1)$, and let $A=\left.\overline{B}(0,1)\middle\backslash B(z,r)\right.$.
There exists $\rho\in{]0,r[}$ and a random variable $\Gamma$, measurable with respect to the restriction of $\Pi$ to the set of roads with speed less than $1$ that pass through $\overline{B}(0,2)$ but do not pass through $\overline{B}(z,\rho)$, such that
\[T(x,y)\leq\Gamma\cdot|x-y|^{(\beta-d)/(\beta-1)}\cdot\ln\left(\frac{4}{|x-y|}\right)^{1/(\beta-1)}\quad\text{for all $x\neq y\in A$.}\]
Moreover, there exists constants $C$ and $c$ such that
\[\mathbb{P}\left(\Gamma>t\right)\leq C\cdot\exp\left[-c\cdot t^{\beta-1}\right]\quad\text{for all $t\in\mathbb{R}_+^*$.}\]
\end{lem}

For the convenience of the reader, we provide a proof of this ``folklore'' result in the appendix, see Section \ref{sec:appendix}.
Our second tool is the following elementary geometric result, which we use to approximate a circular arc by a finite sequence of line segments.

\begin{lem}\label{lem:confluencetool2}
Fix $\rho\in{]0,1[}$.
There exists an integer $K=K(\rho)\in\mathbb{N}^*$ and a constant $C$ such that the following holds: for every $x,y\in\partial B(0,1)$, there exists a sequence $x=z_0,z_1,\ldots,z_k=y$, with ${z_0\neq z_1\neq\ldots\neq z_k\in\partial B(0,1)}$ and $k\leq K$, such that $d(0;[z_{i-1},z_i])>\rho$ for all ${i\in\llbracket1,k\rrbracket}$ and ${|z_0-z_1|+\ldots+|z_{k-1}-z|\leq C}$.
\end{lem}

The proof is immediate {(the constant $C$ can be taken to be $\pi$).}

\subsection{Around typical points}\label{subsec:confluencetyp}

In this subsection, we prove item \ref{item:confluencetyp} of Proposition \ref{prop:confluence}.
Following the general strategy outlined in the previous subsection, we first construct a ``master'' good event $G$ on which the desired confluence of geodesics property holds at a fixed scale around some target point $x$.
See Figure \ref{fig:confluencetyp} for an illustration.
The construction is based on an idea from Aldous \cite[Figure 6]{aldousscale}, which was already used by Kendall for this purpose in the planar case $d=2$ \cite[Figure 10]{kendall}.

\begin{lem}\label{lem:confluencetypmaster}
There exists a constant $r\in{]0,1/3[}$ and a good event $G$ that has positive probability, such that on $G$, there exists a cut point $z\in\left.\overline{B}(0,2/3)\middle\backslash\overline{B}(e_1/3,r)\right.$ such that any geodesic $V$-path from a point inside $\overline{B}(e_1/3,r)$ to a point outside $\overline{B}(0,2/3)$ must pass through $z$.

A fortiori, on $G$, there exists a cut point $z\in\left.\overline{B}(e_1/3,1)\middle\backslash\overline{B}(e_1/3,r)\right.$ such that any geodesic $V$-path from a point inside $\overline{B}(e_1/3,r)$ to a point outside $\overline{B}(e_1/3,1)$ must pass through $z$.
\end{lem}

\begin{figure}[!ht]
\centering
\includegraphics[width=0.6\linewidth]{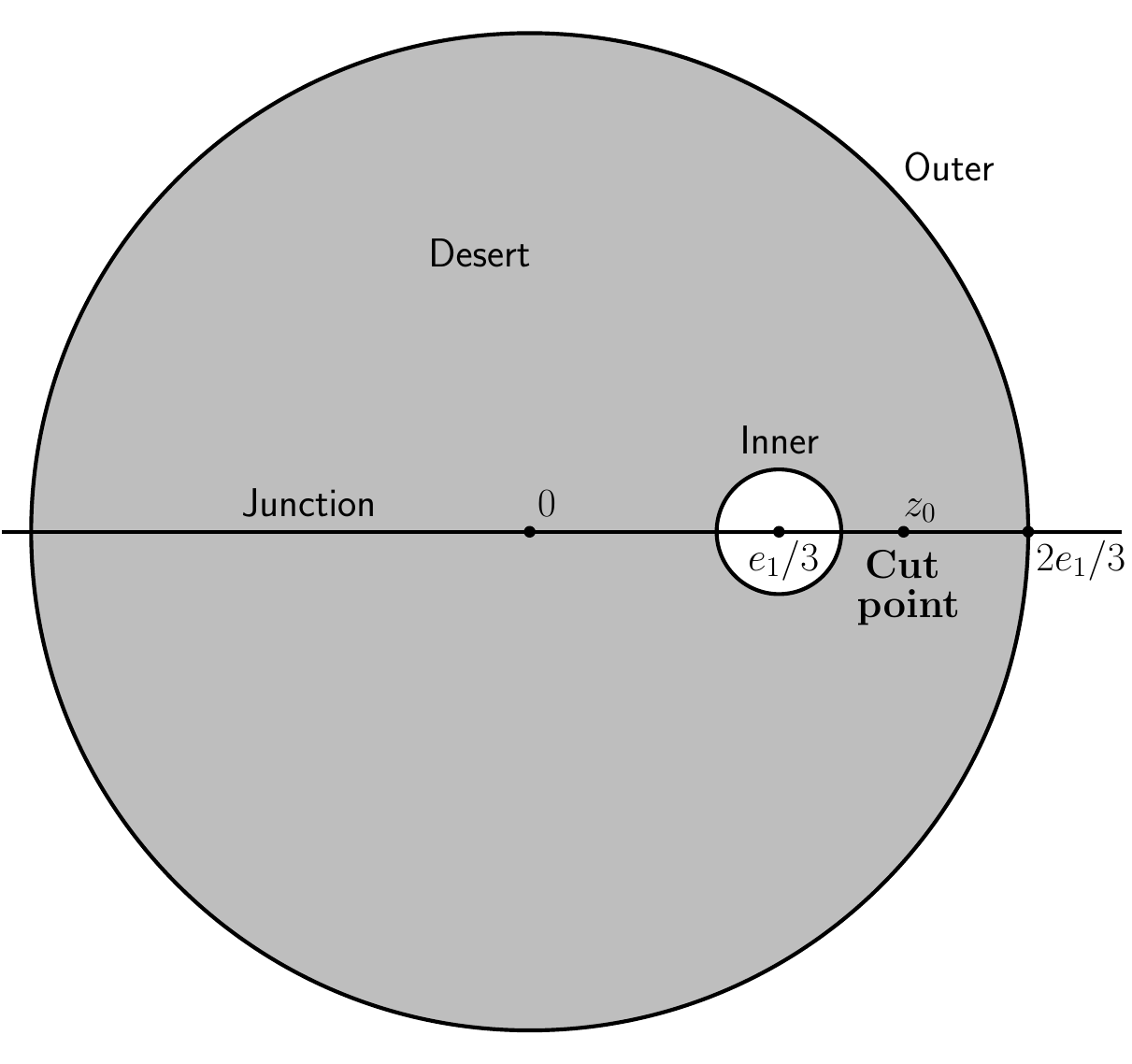}
\caption{Sketch of the good event $G$ of Lemma \ref{lem:confluencetypmaster}.
Picture $\mathsf{Desert}$ as an area with a poor road system, crossed by a single fast road labeled $\mathsf{Junction}$.
Furthermore, picture $\mathsf{Inner}$ and $\mathsf{Outer}$ as ring roads (constructed out of multiple roads of $\Pi$), whose circuit has negligible driving time in comparison with the driving time from $\mathsf{Inner}$ to $\mathsf{Outer}$ using the road $\mathsf{Junction}$.
Then, because of the dissymmetry, any geodesic from a point inside $\mathsf{Inner}$ to a point outside $\mathsf{Outer}$ must use the road $\mathsf{Junction}$ from a point near $e_1/3$ to a point near $2e_1/3$.
The precise endpoints might differ from one geodesic to another, but all geodesics must pass through a certain cut point $z$ between $e_1/3$ and $2e_1/3$.}\label{fig:confluencetyp}
\end{figure}

\begin{proof}[Proof of Lemma \ref{lem:confluencetypmaster}] 
The construction of $G$ is sketched in  Figure \ref{fig:confluencetyp}.
We now embark in the {rather lenghty but not difficult construction of this event in detail.}
By Lemma \ref{lem:confluencetool1}, there exists a random variable $\Gamma$, measurable with respect to the restriction of $\Pi$ to the set of roads with speed less than $1$ that pass through $\overline{B}(0,2)$, such that
\[T(x,y)\leq\Gamma\cdot|x-y|^{(\beta-d)/(\beta-1)}\cdot\ln\left(\frac{4}{|x-y|}\right)^{1/(\beta-1)}\quad\text{for all $x\neq y\in\overline{B}(0,1)$.}\]
Moreover, there exists constants $C$ and $c$ such that
\[\mathbb{P}(\Gamma>t)\leq C\cdot\exp\left[-c\cdot t^{\beta-1}\right]\quad\text{for all $t\in\mathbb{R}_+^*$.}\]
In particular, we can fix $t\in\mathbb{R}_+^*$ large enough so that $\mathbb{P}(\Gamma>t)\leq1/2$.
Next, fix ${r\in{]0,1/432]}}$.
By Lemma \ref{lem:confluencetool2}, there exists an integer $K=K(r)$ and a constant $C$ such that the following holds: for every ${x,y\in\partial B(0,2/3)}$, there exists a sequence $x=z_0,z_1,\ldots,z_k=y$, with 
\[z_0\neq z_1\neq\ldots\neq z_k\in\partial B\left(0,\frac{2}{3}\right)\]
and $k\leq K$, such that $d(0;[z_{i-1},z_i])>2/3-r$ for all $i\in\llbracket1,k\rrbracket$ and
\[|z_0-z_1|+\ldots+|z_{k-1}-z_k|\leq C.\]
Then, fix $\lambda\geq25$ large enough so that $C/\lambda\leq1/216$, and fix $\delta\in{]0,r]}$ small enough so that
\begin{equation}\label{eq:delta}
2(K+2)\cdot t\cdot\omega(\delta)+2K\cdot\frac{\delta}{\lambda}\leq\frac{1}{216}
\end{equation}
where $\omega(\delta)=\sup_{s\in{]0,2\delta]}}s^{(\beta-d)/(\beta-1)}\cdot\ln(4/s)^{1/(\beta-1)}$.
We proceed as follows.
\begin{itemize}
\item Fix a finite covering $(B(w,\delta)\,;\,w\in W_\mathsf{in})$ of $\partial B(e_1/3,r)$ by balls of radius $\delta$, with centres ${w\in\partial B(e_1/3,r)}$.
Then, set $\mathsf{Inner}$ to be the event: ``for every ${w_1\neq w_2\in W_\mathsf{in}}$, there exists a road $(\ell,v)\in\Pi$ with speed $v\in[1,2[$ that passes through $\overline{B}(w_1,\delta)$ and $\overline{B}(w_2,\delta)$''.

\item Fix a finite covering $(B(w,\delta)\,;\,w\in W_\mathsf{out})$ of $\partial B(0,2/3)$ by balls of radius $\delta$, with centres $w\in\partial B(0,2/3)$.
Then, set $\mathsf{Outer}$ to be the event: ``for every ${w_1\neq w_2\in W_\mathsf{out}}$ such that ${\mu\left(\left\langle\overline{B}(w_1,\delta)\,;\,\overline{B}(w_2,\delta)\right\rangle\middle\backslash\left\langle\overline{B}(0,2/3-r)\right\rangle\right)>0}$, there exists a road of $\Pi$ with speed at least $\lambda$ that passes through $\overline{B}(w_1,\delta)$ and $\overline{B}(w_2,\delta)$ but does not pass through ${\overline{B}(0,2/3-r)}$''.
\end{itemize}
Finally, we introduce the event $\mathsf{Junction}$: ``there exists a road $(\ell,v)\in\Pi$ with speed ${v\in[24,25[}$ that passes through $\overline{B}(-2e_1/3,\delta)$ and $\overline{B}(2e_1/3,\delta)$'', and the event $\mathsf{Desert}$: ``there is no more than one road of $\Pi$ with speed at least $2$ that passes through $\overline{B}(0,2/3-r)$''.
Then, we let
\[G=(\Gamma\leq t)\cap\mathsf{Inner}\cap\mathsf{Outer}\cap\mathsf{Junction}\cap\mathsf{Desert}.\]
We claim that on $G$, the following holds.
\begin{enumerate}
\item First, there exists a unique road $(\ell_0,v_0)\in\Pi$ with speed $v_0\geq2$ that passes through $\overline{B}(0,2/3-r)$.
Moreover, this road has speed $v_0\in[24,25[$, and passes through $\overline{B}(-2e_1/3,\delta)$ and $\overline{B}(2e_1/3,\delta)$.

\item For every $x\in\partial B(e_1/3,r)$ and $y\in\partial B(0,2/3)$, we have $T(x,y)\leq1/36$.

\item For any geodesic $V$-path $\gamma$ from a point inside $\overline{B}(e_1/3,r)$ to a point outside $\overline{B}(0,2/3)$, there exists $x_0\in\ell_0\cap\overline{B}(e_1/3,1/9)$ and $y_0\in\ell_0\cap\overline{B}(2e_1/3,1/9)$ such that the $V$-path $\gamma$ drives from $x_0$ to $y_0$ using the road $(\ell_0,v_0)$.
In particular, any such path must pass through the cut point $z$ defined as the orthogonal projection of $e_1/2$ on $\ell_0$.
\end{enumerate}
This yields the result of the lemma, provided that $G$ has positive probability.
First, let us check these assertions.
\begin{enumerate}
\item The first assertion is immediate from the construction of $G$.

\item Fix $x\in\partial B(e_1/3,r)$ and $y\in\partial B(0,2/3)$, and let us check that ${T(x,y)\leq1/36}$.
By the triangle inequality, we have
\[T(x,y)\leq T\left(x,\frac{e_1}{3}+re_1\right)+T\left(\frac{e_1}{3}+re_1,\frac{2e_1}{3}\right)+T\left(\frac{2e_1}{3},y\right).\]
Now, let us control each term separately, starting with the first one.
Fix ${w_1,w_2\in W_\mathsf{in}}$ such that $x\in B(w_1,\delta)$ and $e_1/3+re_1\in B(w_2,\delta)$.
If $w_1=w_2$, then we have 
\[\left|x-\left(\frac{e_1}{3}+re_1\right)\right|\leq2\delta,\]
hence
\[T\left(x,\frac{e_1}{3}+re_1\right)\leq t\cdot\omega(\delta).\]
Otherwise, there exists a road $(\ell,v)\in\Pi$ with speed $v\in[1,2[$ that passes through $\overline{B}(w_1,\delta)$ and $\overline{B}(w_2,\delta)$.
Let us denote by $w_1'$ (resp. $w_2'$) the orthogonal projection of $w_1$ (resp. $w_2$) on $\ell$.
By convexity, we have $w_1'\in\overline{B}(w_1,\delta)$ and $w_2'\in\overline{B}(w_2,\delta)$, hence $|x-w_1'|\leq2\delta$ and $|w_2'-(e_1/3+re_1)|\leq2\delta$.
By the triangle inequality, it follows that
\[\begin{split}
T\left(x,\frac{e_1}{3}+re_1\right)&\leq T\left(x,w_1'\right)+T\left(w_1',w_2'\right)+T\left(w_2',\frac{e_1}{3}+re_1\right)\\
&\leq t\cdot\omega(\delta)+\frac{\left|w_1'-w_2'\right|}{v}+t\cdot\omega(\delta)\\
&\leq t\cdot\omega(\delta)+\frac{|w_1-w_2|}{1}+t\cdot\omega(\delta)\\
&\leq t\cdot\omega(\delta)+\frac{2r}{1}+t\cdot\omega(\delta).
\end{split}\]
(We got $|w_1'-w_2'|\leq|w_1-w_2|$ from the $1$-Lipschitz continuity of the orthogonal projection on $\ell$, and $|w_1-w_2|\leq2r$ since $w_1,w_2\in\overline{B}(e_1/3,r)$.)
In any case, we have
\begin{equation}\label{eq:inner}
T\left(x,\frac{e_1}{3}+re_1\right)\leq2r+2\cdot t\cdot\omega(\delta).
\end{equation}
Now, let us bound the second term, namely $T(e_1/3+re_1,2e_1/3)$.
By convexity, the road $(\ell_0,v_0)$ passes through ${\overline{B}(e_1/3+re_1,\delta)}$ and ${\overline{B}(2e_1/3,\delta)}$.
Let us denote by $x_0$ (resp. $y_0$) the orthogonal projection of $e_1/3+re_1$ (resp. $2e_1/3$) on $\ell_0$.
By the triangle inequality, we have
\[\begin{split}
T\left(\frac{e_1}{3}+re_1,\frac{2e_1}{3}\right)&\leq T\left(\frac{e_1}{3}+re_1,x_0\right)+T(x_0,y_0)+T\left(y_0,\frac{2e_1}{3}\right)\\
&\leq t\cdot\omega(\delta)+\frac{|x_0-y_0|}{v_0}+t\cdot\omega(\delta)\\
&\leq t\cdot\omega(\delta)+\frac{|(e_1/3+re_1)-2e_1/3|}{24}+t\cdot\omega(\delta)\\
&\leq t\cdot\omega(\delta)+\frac{1/3}{24}+t\cdot\omega(\delta).
\end{split}\]
Thus, we obtain
\begin{equation}\label{eq:junction}
T\left(\frac{e_1}{3}+re_1,\frac{2e_1}{3}\right)\leq\frac{1}{72}+2\cdot t\cdot\omega(\delta).
\end{equation}
Finally, let us bound the last term, namely $T(2e_1/3,y)$.
By construction, there exists a sequence ${2e_1/3=z_0,z_1,\ldots,z_k=y}$, with ${z_0\neq z_1\neq\ldots\neq z_k\in\partial B(0,2/3)}$ and $k\leq K$, such that $d(0;[z_{i-1},z_i])>2/3-r$ for all $i\in\llbracket1,k\rrbracket$ and ${|z_0-z_1|+\ldots+|z_{k-1}-z_k|\leq C}$.
For each $i\in\llbracket0,k\rrbracket$, fix ${w_i\in W_\mathsf{out}}$ such that $z_i\in B(w_i,\delta)$.
Now, fix $i\in\llbracket1,k\rrbracket$, and let us bound $T(z_{i-1},z_i)$.
If $w_{i-1}=w_i$, then we have $|z_{i-1}-z_i|\leq2\delta$, hence ${T(z_{i-1},z_i)\leq t\cdot\omega(\delta)}$.
Otherwise, the line that passes through $z_{i-1}$ and $z_i$ passes through $\overline{B}(w_{i-1},\delta)$ and $\overline{B}(w_i,\delta)$ but does not pass through ${\overline{B}(0,2/3-r)}$, and the same holds for any small perturbation of that line.
This shows that 
\[\mu\left(\left\langle\overline{B}(w_{i-1},\delta)\,;\,\overline{B}(w_i,\delta)\right\rangle\middle\backslash\left\langle\overline{B}(0,2/3-r)\right\rangle\right)>0,\]
hence there exists a road $(\ell_i,v_i)\in\Pi$ with speed at least $\lambda$ that passes through $\overline{B}(w_{i-1},\delta)$ and $\overline{B}(w_i,\delta)$ but does not pass through ${\overline{B}(0,2/3-r)}$.
Let us denote by $w_{i-1}'$ (resp. $w_i'$) the orthogonal projection of $w_{i-1}$ (resp. $w_i$) on $\ell_i$.
By convexity, we have $w_{i-1}'\in\overline{B}(w_i,\delta)$ and ${w_i'\in\overline{B}(w_i,\delta)}$, hence $\left|z_{i-1}-w_{i-1}'\right|\leq2\delta$ and $|w_i'-z_i|\leq2\delta$.
By the triangle inequality, it follows that
\[\begin{split}
T(z_{i-1},z_i)&\leq T\left(z_{i-1},w_{i-1}'\right)+T\left(w_{i-1}',w_i'\right)+T\left(w_i',z_i\right)\\
&\leq t\cdot\omega(\delta)+\frac{\left|w_{i-1}'-w_i'\right|}{v_i}+t\cdot\omega(\delta)\\
&\leq t\cdot\omega(\delta)+\frac{|w_{i-1}-w_i|}{\lambda}+t\cdot\omega(\delta)\\
&\leq t\cdot\omega(\delta)+\frac{\delta+|z_{i-1}-z_i|+\delta}{\lambda}+t\cdot\omega(\delta)\\
&=\frac{|z_{i-1}-z_i|}{\lambda}+2\cdot\left(t\cdot\omega(\delta)+\frac{\delta}{\lambda}\right).
\end{split}\]
Summing over $i\in\llbracket1,k\rrbracket$, we deduce that
\begin{equation}\label{eq:outer}
T\left(\frac{2e_1}{3},y\right)\leq\frac{C}{\lambda}+k\cdot2\cdot\left(t\cdot\omega(\delta)+\frac{\delta}{\lambda}\right)\leq\frac{C}{\lambda}+2K\cdot\left(t\cdot\omega(\delta)+\frac{\delta}{\lambda}\right).
\end{equation}
Putting together \eqref{eq:inner}, \eqref{eq:junction} and \eqref{eq:outer}, we obtain
\[T(x,y)\leq2r+\frac{1}{72}+\frac{C}{\lambda}+2(K+2)\cdot t\cdot\omega(\delta)+2K\cdot\frac{\delta}{\lambda}.\]
Recalling \eqref{eq:delta}, we end up with
\[T(x,y)\leq\frac{2}{432}+\frac{1}{72}+\frac{1}{216}+\frac{1}{216}=\frac{1}{36},\]
as we wanted to show.

\item Fix a geodesic $V$-path $\gamma:[0,T]\rightarrow\mathbb{R}^d$ from a point inside $\overline{B}(e_1/3,r)$ to a point outside $\overline{B}(0,2/3)$, and let us check that there exists 
\[x_0\in\ell_0\cap\overline{B}\left(\frac{e_1}{3},\frac{1}{9}\right)\quad\text{and}\quad y_0\in\ell_0\cap\overline{B}\left(\frac{2e_1}{3},\frac{1}{9}\right)\]
such that $\gamma$ drives from $x_0$ to $y_0$ using the road $(\ell_0,v_0)$.
Let ${x\in\partial B(e_1/3,r)}$ be the first exit point of $\gamma$ from $\overline{B}(e_1/3,r)$, and let $y\in\partial B(0,2/3-r)$ be the first exit point of $\gamma$ from $\overline{B}(0,2/3-r)$.
We claim that $T(x,y)\leq1/36$.
Indeed, let $z\in\partial B(0,2/3)$ be the first exist point of $\gamma$ from $\overline{B}(0,2/3)$.
We have $T(x,y)\leq T(x,z)$, and by the previous point, we have ${T(x,z)\leq1/36}$.
On the other hand, by the triangle inequality, we have
\[|x-y|\geq|y|-|x|\geq\frac{2}{3}-r-\left(\frac{1}{3}+r\right)=\frac{1}{3}-2r\geq\frac{1}{3}-\frac{2}{432}=\frac{71}{216}.\]
It follows that
\[\frac{|x-y|}{T(x,y)}\geq\frac{71/216}{1/36}=\frac{71}{6}\geq2,\]
hence the $V$-path $\gamma$ must use a road of $\Pi$ with speed at least $2$ to drive from $x$ to $y$.
Since $(\ell_0,v_0)$ is the only road of $\Pi$ with speed at least $2$ that passes through $\overline{B}(0,2/3-r)$, this road cannot be any other than $(\ell_0,v_0)$.
Moreover, since $(\ell_0,v_0)$ is the fastest road of $\Pi$ within $\overline{B}(0,2/3-r)$, the $V$-path $\gamma$ uses $(\ell_0,v_0)$ over a time interval, say $[a,b]\subset[0,T]$.
We set $x_0=\gamma(a)$ and $y_0=\gamma(b)$.
Now, we want to show that $x_0\in\overline{B}(e_1/3,1/9)$ and $y_0\in\overline{B}(2e_1/3,1/9)$.
To this end, note that we have
\[T(x,y)\geq\frac{|x-x_0|}{2}+\frac{|x_0-y_0|}{v_0}+\frac{|y_0-y|}{2},\]
hence
\[\frac{|x-x_0|}{2}+\frac{|x_0-y_0|}{25}+\frac{|y_0-y|}{2}\leq\frac{1}{36}.\]
In particular, we have $|x-x_0|\leq1/18$ and $|y_0-y|\leq1/18$, and $|x_0-y_0|\leq25/36$.
By the triangle inequality, it follows that
\[\left|x_0-\frac{e_1}{3}\right|\leq|x_0-x|+\left|x-\frac{e_1}{3}\right|\leq\frac{1}{18}+r\leq\frac{1}{18}+\frac{1}{432}=\frac{25}{432}\leq\frac{1}{9}.\]
Next, we use the bound $|y_0-y|\leq1/18$ to show that 
\[y_0\in\overline{B}\left(-\frac{2e_1}{3},\frac{1}{9}\right)\cup\overline{B}\left(\frac{2e_1}{3},\frac{1}{9}\right).\]
First, since $\ell_0$ passes through $\overline{B}(-2e_1/3,\delta)$ and $\overline{B}(2e_1/3,\delta)$, by convexity we must have ${y_0\in\overline{B}(w,\delta)}$ for some $w\in[-2e_1/3,2e_1/3]$.
Then, by the triangle inequality, we have
\[|w|\geq|y_0|-\delta\geq|y|-|y-y_0|-\delta\geq\frac{2}{3}-r-\frac{1}{18}-\delta\geq\frac{2}{3}-\left(\frac{1}{18}+2r\right).\]
Therefore, we must have 
\[w\in\overline{B}\left(-\frac{2e_1}{3},\frac{1}{18}+2r\right)\cup\overline{B}\left(\frac{2e_1}{3},\frac{1}{18}+2r\right).\]
Finally, since $|y_0-w|\leq\delta$, and since
\[\frac{1}{18}+2r+\delta\leq\frac{1}{18}+3r\leq\frac{1}{18}+\frac{3}{432}=\frac{27}{432}\leq\frac{1}{9},\]
we conclude that ${y_0\in\overline{B}(-2e_1/3,1/9)\cup\overline{B}(2e_1/3,1/9)}$.
It remains to rule out the possibility that $y_0\in\overline{B}(-2e_1/3,1/9)$.
To this end, we use the bound $|x_0-y_0|\leq25/36$ to argue that $y_0$ is too close to $e_1/3$ to belong to $\overline{B}(-2e_1/3,1/9)$.
Indeed, by the triangle inequality, we have
\[\left|y_0-\frac{e_1}{3}\right|\leq|y_0-x_0|+\left|x_0-\frac{e_1}{3}\right|\leq\frac{25}{36}+\frac{1}{9}=\frac{29}{36}<\frac{8}{9}=d\left(\frac{e_1}{3};\overline{B}\left(-\frac{2e_1}{3},\frac{1}{9}\right)\right).\]
\end{enumerate}
To conclude the proof of the lemma, it remains to check that the event $G$ has positive probability.
To this end, we rewrite it as the intersection
\[G=(\Gamma\leq t)\cap\mathsf{Inner}\cap\mathsf{Outer}\cap\mathsf{Junction}'\cap\mathsf{Desert}',\]
where $\mathsf{Junction}'$ is the good event: ``there is exactly one road $(\ell,v)\in\Pi$ with speed ${v\in[24,25[}$ that passes through $\overline{B}(-2e_1/3,\delta)$ and $\overline{B}(2e_1/3,\delta)$'', and $\mathsf{Desert}'$ is the intersection of the following events:
\begin{itemize}
\item $D$: ``there is no road $(\ell,v)\in\Pi$ with speed $v\in{[2,24[}\cup{[25,\infty[}$ that passes through $\overline{B}(-2e_1/3,\delta)$ and $\overline{B}(2e_1/3,\delta)$''.
\item $E$: ``there is no road of $\Pi$ with speed at least $2$ that passes through $\overline{B}(0,2/3-r)$ but does not belong to $\left\langle\overline{B}(-2e_1/3,\delta)\,;\,\overline{B}(2e_1/3,\delta)\right\rangle$''.
\end{itemize}
In this decomposition, the events $(\Gamma\leq t)$, $\mathsf{Inner}$, $\mathsf{Outer}$, $\mathsf{Junction}'$, $D$ and $E$ are clearly independent, and each of them has positive probability: it follows that $G$ has positive probability, which completes the proof.
\end{proof}

Now, it remains to reap the rewards of the above work.

\begin{proof}[Proof of item \ref{item:confluencetyp} of Proposition \ref{prop:confluence}]
Without loss of generality, we assume that $x=0$.
\begin{enumerate}[label=(\alph*)]
\item Set $r_k=r^k$ for all $k\in\mathbb{N}$, where $r\in{]0,1[}$ is the constant of Lemma \ref{lem:confluencetypmaster}.
It suffices to prove that almost surely, there are infinitely many integers $k\in\mathbb{N}$ for which there exists a cut point $z\in\left.\overline{B}(0,r_k)\middle\backslash\overline{B}(0,r_{k+1})\right.$ such that any geodesic $V$-path from a point inside $\overline{B}(0,r_{k+1})$ to a point outside $\overline{B}(0,r_k)$ must pass through $z$.
For each $k\in\mathbb{N}$, let $G_k$ be the good event $G$ of Lemma \ref{lem:confluencetypmaster}, but for the pushforward $\Pi\circ f_{e_1/3,r_k^{-1}}^{-1}$ in place of $\Pi$, where $f_{x,r}$ is the map of \eqref{eq:defscaling}.
By \eqref{eq:selfsimilarPi}, we have $\mathbb{P}(G_k)=\mathbb{P}(G)$.
Moreover, note that for every ${(\ell,v)\in\mathbb{L}\times\mathbb{R}_+^*}$, the road $(\ell,v)$ passes through $\overline{B}(0,r_k)$ \big(resp. $\overline{B}(0,r_{k+1})$\big) if and only if $f_{e_1/3,r_k^{-1}}(\ell,v)$ passes through $\overline{B}(e_1/3,1)$ \big(resp. $\overline{B}(e_1/3,r)$\big).
In particular, on $G_k$, there exists a cut point $z\in\left.\overline{B}(0,r_k)\middle\backslash\overline{B}(0,r_{k+1})\right.$ such that any geodesic $V$-path from a point inside $\overline{B}(0,r_{k+1})$ to a point outside $\overline{B}(0,r_k)$ must pass through $z$.
To conclude, it suffices to {notice that by ergodicity (see Lemma \ref{lem:mixingzoom}), we have}
\[\frac{\#\{k\in{\llbracket0,n\llbracket}:\text{$G_k$ is realised}\}}{n}\underset{n\to\infty}{\longrightarrow}\mathbb{P}(G)\quad\text{almost surely.}\]

\item This is a straightforward consequence of the previous point and the uniqueness of geodesics between typical points.
Recall that almost surely, there exists a unique geodesic $V$-path from $0$ to $e_1$ (see \cite[Theorem 4.4]{kendall} for $d=2$ and \cite[Theorem 4.8]{kahn} for $d\geq3$).
Now, fix a typical realisation of $\Pi$, and fix $y_1,y_2\in\left.\mathbb{R}^d\middle\backslash\{0\}\right.$.
Let ${\gamma_1:[0,T(0,y_1)]\rightarrow\mathbb{R}^d}$ be a geodesic $V$-path from $0$ to $y_1$, and ${\gamma_2:[0,T(0,y_2)]\rightarrow\mathbb{R}^d}$ from $0$ to $y_2$.
Moreover, let $\gamma_0:[0,T(0,e_1)]\rightarrow\mathbb{R}^d$ be the unique geodesic $V$-path from $0$ to $e_1$.
Then, fix $\delta>0$ small enough so that $y_1,y_2,e_1\notin\overline{B}(0,\delta)$.
By the previous point, there exists $\eta\in{]0,\delta[}$ and a cut point $z\in\left.\overline{B}(0,\delta)\middle\backslash\overline{B}(0,\eta)\right.$ such that any geodesic $V$-path from a point inside $\overline{B}(0,\eta)$ to a point outside $\overline{B}(0,\delta)$ must pass through $z$.
In particular, there exists ${t_1\in{]0,T(0,y_1)[}}$, $t_2\in{]0,T(0,y_2)[}$ and $t_0\in{]0,T(0,e_1)[}$ such that ${\gamma_1(t_1)=\gamma_2(t_2)=\gamma_0(t_0)=z}$.
Since $\gamma_1$, $\gamma_2$ and $\gamma_0$ are geodesic, we must have ${t_1=t_2=t_0=T(0,z)}$.
Finally, since $\gamma_0$ is the unique geodesic from $0$ to $e_1$, we must have $\gamma_1(t)=\gamma_2(t)=\gamma_0(t)$ for all $t\in[0,T(0,z)]$.
\end{enumerate}
\end{proof}

\subsection{Around points on roads}\label{subsec:confluenceroads}

In this subsection, we prove item \ref{item:confluenceroads} and \ref{item:confluenceintersections} of Proposition \ref{prop:confluence}.
Following the same general strategy as in the previous subsection, we first construct a ``master'' good event as follows.

\begin{lem}\label{lem:confluenceroadsmaster}
There exists a constant $r\in{]0,1/3[}$, and a good event $G$ that has positive probability, such that on $G$, the following holds:
\begin{itemize}
\item For every $x\in\partial B(0,r)$ and $y\in\partial B(0,2/3)$, we have $T(x,y)\geq1/6$,

\item For every $x,x'\in\partial B(0,r)$ and $y,y'\in\partial B(0,2/3)$, we have 
\[T\left(x,x'\right)+T\left(y,y'\right)\leq\frac{1}{12}.\]
\end{itemize}
\end{lem}

Later on, we will use this statement together with a Mecke formula argument: the point $0$ will then play the role of a typical point on a road $(\ell,v)\in\Pi$ (in view of item \ref{item:confluenceroads} of Proposition \ref{prop:confluence}), or of the intersection point of two roads $(\ell_1,v_1)\neq(\ell_2,v_2)\in\Pi$ in the planar case $d=2$ (in view of item \ref{item:confluenceintersections} of Proposition \ref{prop:confluence}).

\begin{proof}
Since the construction is very similar to that of Lemma \ref{lem:confluencetypmaster}, we only sketch it in  Figure \ref{fig:near-road}, and leave the details to the reader.

\begin{figure}[!h]
\begin{center}
\includegraphics[width=10cm]{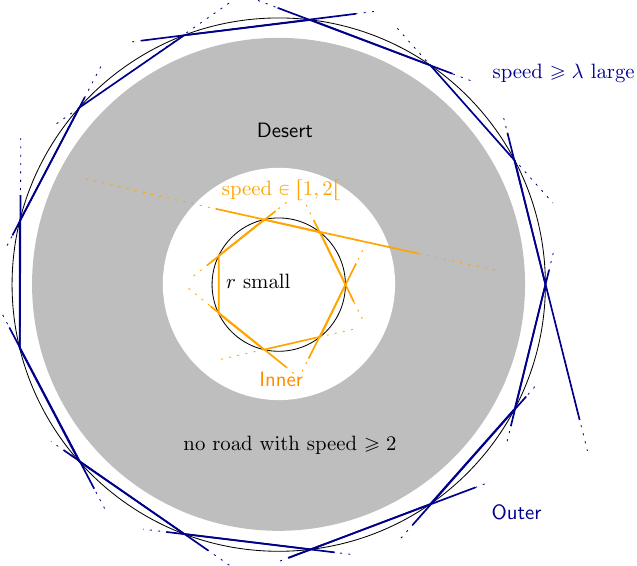}
\caption{
{Sketch of the good event $G$ of Lemma \ref{lem:confluenceroadsmaster}.
Picture $\mathsf{Desert}$ as an area with a poor road system, in which no road has speed at least $2$.
Furthermore, picture ${\partial B(0,r)=\mathsf{Inner}}$ and ${\partial B(0,2/3)=\mathsf{Outer}}$ as ring roads (constructed out of multiple roads of $\Pi$), with speed $\in[1,2[$ and $\geq\lambda$ respectively, where $r$ is small and $\lambda$ is large.
Then, taking $r$ small enough and $\lambda$ large enough, we see that the  driving time between $\mathsf{Inner}$ and $\mathsf{Outer}$ (which is at least $1/2\cdot(2/3-r)$) is much longer than the driving time diameter of $\mathsf{Inner}$ (which is roughly at most $\pi\cdot r\cdot1/1$) plus the driving time diameter of $\mathsf{Outer}$ (which is roughly at most $\pi\cdot2/3\cdot1/\lambda$).} \label{fig:near-road}}
\end{center}
\end{figure}
\end{proof}

Now, as before, it remains to reap the rewards of all the above work.

\begin{proof}[Proof of items \ref{item:confluenceroads} and \ref{item:confluenceintersections} of Proposition \ref{prop:confluence}]
\begin{enumerate}
\item[\ref{item:confluenceroads}.] \begin{enumerate}
\item In view of the Mecke formula (see, e.g, \cite[Theorem 4.1]{lastpenrose}), let us first fix a road $(\ell_0,v_0)\in\mathbb{L}\times\mathbb{R}_+^*$, and show that almost surely, the following property $P_{(\ell_0,v_0)}(x)$ holds for almost every $x\in\ell_0$: ``for each $\delta>0$, there exists $\eta\in{]0,\delta[}$ such that any geodesic $V'$-path from a point inside $\overline{B}(x,\eta)$ to a point outside $\overline{B}(x,\delta)$ must use $(\ell_0,v_0)$ within $\overline{B}(x,\delta)$'', where $V'$ denotes the ``speed limits'' function induced by $\Pi'=\Pi\oplus(\ell_0,v_0)$.
Without loss of generality {by the invariance in distribution of $\Pi$ under rotations and translations}, we assume that $\ell_0=\mathbb{R}\times\{0\}^{d-1}$.
First, let us show that almost surely, the property $P_{(\ell_0,v_0)}(x)$ holds for $x=0$.
To this end, as in the proof of the first item of Proposition \ref{prop:confluence}, set $r_k=r^k$ for all $k\in\mathbb{N}$, where $r\in{]0,1[}$ is the constant of Lemma \ref{lem:confluenceroadsmaster}, and let $G_k$ be the good event $G$ of Lemma \ref{lem:confluenceroadsmaster} but for the pushforward $\Pi\circ f_{0,r_k^{-1}}^{-1}$ in place of $\Pi$, where $f_{x,r}$ is the map of \eqref{eq:defscaling}.
By \eqref{eq:selfsimilarPi}, we have $\mathbb{P}(G_k)=\mathbb{P}(G)$.
Moreover, on $G_k$, the following holds:
\begin{itemize}
\item For every $x\in\partial B(0,r\cdot r_k)$ and $y\in\partial B(0,2r_k/3)$, we have 
\[T(x,y)\geq\frac{1}{6}\cdot r_k^{(\beta-d)/(\beta-1)},\]

\item For every $x,x'\in\partial B(0,r\cdot r_k)$ and $y,y'\in\partial B(0,2r_k/3)$, we have 
\[T\left(x,x'\right)+T\left(y,y'\right)\leq\frac{1}{12}\cdot r_k^{(\beta-d)/(\beta-1)}.\]
\end{itemize}
This readily implies that for all sufficiently large $k$, on $G_k$, any geodesic $V'$-path from a point inside $\overline{B}(0,r_{k+1})$ to a point outside $\overline{B}(0,r_k)$ must use $(\ell_0,v_0)$ within $\overline{B}(0,r_k)$.
Indeed, let us denote by $T'$ the ``driving time'' random metric induced by ${\Pi'=\Pi\oplus(\ell_0,v_0)}$. 
Assume that $k$ is large enough so that $2r_k/v_0<1/12\cdot r_k^{(\beta-d)/(\beta-1)}$, fix a $V'$-path $\gamma$ from a point inside $\overline{B}(0,r_{k+1})$ to a point outside $\overline{B}(0,r_k)$, and let us show that on $G_k$, the $V'$-path $\gamma$ must use $(\ell_0,v_0)$ within $\overline{B}(0,r_k)$.
Let ${x\in\partial B(0,r\cdot r_k)}$ be the first exit point of $\gamma$ from $\overline{B}(0,r\cdot r_k)$, and let $y\in\partial B(0,2r_k/3)$ be the first exit point of $\gamma$ from $\overline{B}(0,2r_k/3)$.
Moreover, fix points $x_0\in\ell_0\cap\partial B(0,r\cdot r_k)$ and ${y_0\in\ell_0\cap\partial B(0,2r_k/3)}$.
By the triangle inequality, we have
\[\begin{split}
T'(x,y)&\leq T(x,x_0)+\frac{|x_0-y_0|}{v_0}+T(y_0,y)\\
&\leq T(x,x_0)+\frac{2r_k}{v_0}+T(y_0,y)\leq\frac{1}{12}\cdot r_k^{(\beta-d)/(\beta-1)}+\frac{2r_k}{v_0},
\end{split}\]
therefore,
\[T'(x,y)<\frac{1}{12}\cdot r_k^{(\beta-d)/(\beta-1)}+\frac{1}{12}\cdot r_k^{(\beta-d)/(\beta-1)}=\frac{1}{6}\cdot r_k^{(\beta-d)/(\beta-1)}\leq T(x,y).\]
It follows that $\gamma$ must use $(\ell_0,v_0)$ to drive from $x$ to $y$, hence within $\overline{B}(0,r_k)$.
By ergodicity as before (using Lemma \ref{lem:mixingzoom}), we obtain that almost surely, the event $G_k$ is realised for infinitely many $k\in\mathbb{N}$.
Thus, we obtain that almost surely, the property $P_{(\ell_0,v_0)}(x)$ holds for $x=0$.
By the invariance in distribution of $\Pi$ under translations and by Fubini's theorem, it follows that almost surely, the property $P_{(\ell_0,v_0)}(x)$ holds for almost every $x\in\ell_0$.
We deduce by the Mecke formula that almost surely, for each road $(\ell,v)\in\Pi$, the property $P_{(\ell,v)}(x)$ holds for almost every $x\in\ell$.

\item Fix a typical realisation of $\Pi$, fix a road $(\ell,v)\in\Pi$, and fix a typical point $x\in\ell$.
By the previous item, the following holds: for each $\delta>0$, there exists $\eta\in{]0,\delta[}$ such that any geodesic $V$-path from a point inside $\overline{B}(x,\eta)$ to a point outside $\overline{B}(x,\delta)$ must use $(\ell,v)$ within $\overline{B}(x,\delta)$.
Now, fix a $V$-path $\gamma:[0,\tau]\rightarrow\mathbb{R}^2$ such that $\gamma(0)=x$, and let us prove that there must exist $t_0\in{]0,\tau]}$ such that $\gamma(t)\in\ell$ for all $t\in[0,t_0]$.
Let $\delta\in{]0,|x-\gamma(\tau)|[}$ be small enough so that $(\ell,v)$ is the fastest road of $\Pi$ that passes through $\overline{B}(x,\delta)$.
Since the $V$-path $\gamma$ must use $(\ell,v)$ within $\overline{B}(x,\delta)$, there must exist $t_0\in{]0,\tau]}$ such that $\gamma(t_0)\in\ell\cap\overline{B}(x,\delta)$.
Moreover, the restriction of $\gamma$ to $[0,t_0]$ must remain within $\overline{B}(x,\delta)$.
Indeed, since $(\ell,v)$ is the fastest road of $\Pi$ within $\overline{B}(x,\delta)$, the driving time of any $V$-path from $x$ to a point outside $\overline{B}(x,\delta)$ must be strictly greater than $\delta/v$, which is an obvious upper bound for $t_0$.
In turn, since $(\ell,v)$ is the fastest road of $\Pi$ within $\overline{B}(x,\delta)$, the triangle inequality proves that the $V$-path $\gamma$ must drive from $x$ to $\gamma(t_0)$ using $(\ell,v)$.
\end{enumerate}

\item[\ref{item:confluenceintersections}.] \begin{enumerate}[label=(\alph*)]
\item In view of the (multivariate) Mecke formula (see, e.g, \cite[Theorem 4.4]{lastpenrose}), let us first fix a pair of roads ${(\ell_1,v_1)\neq(\ell_2,v_2)\in\mathbb{L}\times\mathbb{R}_+^*}$.
Arguing exatcly as before, we obtain that almost surely, the intersection point $x$ of $\ell_1$ and $\ell_2$ has the following property: for each $\delta>0$, there exists $\eta\in{]0,\delta[}$ such that any geodesic $V'$-path from a point inside $\overline{B}(x,\eta)$ to a point outside $\overline{B}(x,\delta)$ must use $(\ell_1,v_1)$ or $(\ell_2,v_2)$ within $\overline{B}(x,\delta)$, where $V'$ denotes the ``speed limits'' function induced by $\Pi'=\Pi\oplus(\ell_1,v_1)\oplus(\ell_2,v_2)$.
As above, we conclude using the (multivariate) Mecke formula.

\item Fix a typical realisation of $\Pi$, and fix a pair of roads $(\ell_1,v_1)\neq(\ell_2,v_2)\in\Pi$, with $v_1>v_2$.
By the previous item, the intersection point $x$ of $\ell_1$ and $\ell_2$ has the following property: for each $\delta>0$, there exists $\eta\in{]0,\delta[}$ such that any geodesic $V$-path from a point inside $\overline{B}(x,\eta)$ to a point outside $\overline{B}(x,\delta)$ must use $(\ell_1,v_1)$ or $(\ell_2,v_2)$ within $\overline{B}(x,\delta)$.
Now, fix a $V$-path $\gamma:[0,\tau]\rightarrow\mathbb{R}^2$ such that $\gamma(0)=x$, and let us prove that there must exist $t_0\in{]0,\tau]}$ such that $\gamma(t)\in\ell_1$ for all $t\in[0,t_0]$ or $\gamma(t)\in\ell_2$ for all $t\in[0,t_0]$.
Let $\delta\in{]0,|x-\gamma(\tau)|[}$ be small enough so that the third fastest road of $\Pi$ within $\overline{B}(x,\delta)$ has speed less than $v_2$.
If the $V$-path $\gamma$ uses $(\ell_1,v_1)$ within $\overline{B}(x,\delta)$, then we argue exactly as before: since $(\ell_1,v_1)$ is the fastest road of $\Pi$ within $\overline{B}(x,\delta)$, there must exist $t_0\in{]0,\tau]}$ such that $\gamma(t)\in\ell_1$ for all $t\in[0,t_0]$.
Otherwise, the $V$-path $\gamma$ must use $(\ell_2,v_2)$ within $\overline{B}(x,\delta)$, hence there must exist $t_0\in{]0,\tau]}$ such that $\gamma(t_0)\in\ell_2\cap\overline{B}(x,\delta)$.
Then, once more we argue exactly as before.
First, since $(\ell_2,v_2)$ is the fastest road of $\Pi$ used by $\gamma$ within $\overline{B}(x,\delta)$, the restriction of $\gamma$ to $[0,t_0]$ must remain within $\overline{B}(x,\delta)$ (for otherwise its driving time $t_0$ would be strictly greater than $\delta/v_2$).
In turn, since $(\ell_2,v_2)$ is the fastest road of $\Pi$ used by $\gamma$ within $\overline{B}(x,\delta)$, the triangle inequality proves that the $V$-path $\gamma$ must drive from $x$ to $\gamma(t_0)$ using $(\ell_2,v_2)$.
\end{enumerate}
\end{enumerate}
\end{proof}

\section{On the number of arms of geodesic stars}

In this section, we prove that there is a deterministic bound on the number of arms of geodesic stars.
More precisely, we prove the following result, which is yet another confluence of geodesics property.

\begin{prop}\label{prop:K+stars}
There exists an integer $m\in\mathbb{N}^*$ such that almost surely, for every $x\in\mathbb{R}^d$, the following property $P(x)$ holds: ``for each $\delta>0$, there exists $\eta\in{]0,\delta[}$ and a cut set 
\[\mathcal{Z}\subset\left.\overline{B}(x,\delta)\middle\backslash\overline{B}(x,\eta)\right.\]
with cardinality $\#\mathcal{Z}\leq m$, such that any geodesic $V$-path from a point inside $\overline{B}(x,\eta)$ to a point outside $\overline{B}(x,\delta)$ must pass through $z$ for some $z\in\mathcal{Z}$''.

In particular, there are no $(m+1)^+$-stars in $\left(\mathbb{R}^d,T\right)$.
\end{prop}
See \cite[Proposition 2.6]{gwynnenetworks} for a related result in Liouville quantum gravity.
Our approach is similar to that of the previous section, except that it must be quantitative: as opposed to the confluence of geodesics properties around \textit{typical} points in Proposition \ref{prop:confluence}, we need here to handle \textit{all} points simultaneously.
The main input in the proof is the following lemma.

\begin{lem}\label{lem:K+stars}
There exists an integer $m\in\mathbb{N}^*$ and a constant $\rho\in{]0,1[}$ such that the following holds, where $r_k=\rho^k$ for all $k\in\mathbb{N}$. There exists good events $(G_k)_{k\in\mathbb{N}}$ such that:
\begin{itemize}
\item On $G_k$, there exists a cut set $\mathcal{Z}\subset\left.\overline{B}(0,r_k)\middle\backslash\overline{B}(0,r_{k+1})\right.$ with cardinality $\#\mathcal{Z}\leq m$, such that any geodesic $V$-path from a point inside $\overline{B}(0,r_{k+1})$ to a point outside $\overline{B}(0,r_k)$ must pass through $z$ for some $z\in\mathcal{Z}$,

\item We have
\begin{equation}\label{eq:scalesK+stars}
\mathbb{P}\left(\#\{k\in{\llbracket0,n\llbracket}:\text{$G_k$ is realised}\}\leq\frac{n}{3}\right)\leq r_n^{2d}\quad\text{for all sufficiently large $n$.}
\end{equation}
\end{itemize}
\end{lem}

First, let us show how the lemma yields the result.

\begin{proof}[Proof of Proposition \ref{prop:K+stars} assuming Lemma \ref{lem:K+stars}]
Fix $\delta>0$, and let $\mathcal{X}_\delta$ be the set of points $x\in\mathbb{R}^d$ for which the following property $P_\delta(x)$ does not hold: ``there exists $\eta\in{]0,\delta[}$ and a cut set $\mathcal{Z}\subset\left.\overline{B}(x,\delta)\middle\backslash\overline{B}(x,\eta)\right.$ with cardinality $\#\mathcal{Z}\leq m$, such that any geodesic $V$-path from a point inside $\overline{B}(x,\eta)$ to a point outside $\overline{B}(x,\delta)$ must pass through $z$ for some $z\in\mathcal{Z}$''.
Let us show that for each $R>0$, almost surely, the set ${\mathcal{X}_\delta\cap\overline{B}(0,R)}$ is empty.
For every $n\in\mathbb{N}$, let $(B(y,r_n)\,;\,y\in Y_n)$ be a covering of $\overline{B}(0,R)$ by balls of radius $r_n$, with centres $y\in\overline{B}(0,R)$ at least $r_n$ apart from each other so that the balls $(B(y,r_n/2)\,;\,y\in Y_n)$ are disjoint.
In particular, a measure argument entails that $\#Y_n\leq C\cdot r_n^{-d}$ for some constant $C$.
For each $y\in Y_n$, and for every $k\in\mathbb{N}$, let $G^y_k$ be the good event $G_k$ of Lemma \ref{lem:K+stars}, but for the pushforward $\Pi\circ f_{-y,1}^{-1}$ in place of $\Pi$, where $f_{x,r}$ is the map of \eqref{eq:defscaling}.
By \eqref{eq:selfsimilarPi}, we have
\[\#\left\{k\in{\llbracket0,n\llbracket}:\text{$G_k^y$ is realised}\right\}\overset{\text{\scriptsize law}}{=}\#\left\{k\in{\llbracket0,n\llbracket}:\text{$G_k$ is realised}\right\}.\]
Now, we claim that for all sufficiently large $n$, we have
\begin{equation}\label{eq:K+starsdiscr}
\mathcal{X}_\delta\cap\overline{B}(0,R)\subset\bigcup_{y\in\mathcal{Y}_n}B(y,r_n),
\end{equation}
where $\mathcal{Y}_n$ denotes the set of points $y\in Y_n$ such that $\#\left\{k\in{\llbracket0,n\llbracket}:\text{$G^y_k$ is realised}\right\}\leq n/3$.
Indeed, set $n_0$ to be the smallest integer $n\in\mathbb{N}$ such that $2r_n\leq\delta$, and let $x\in\mathcal{X}_\delta\cap\overline{B}(0,R)$.
Then, fix ${n\geq3n_0}$, and let $y\in Y_n$ be such that $x\in B(y,r_n)$.
By the triangle inequality, we have $\overline{B}(y,r_{n_0})\subset\overline{B}(x,\delta)$.
Thus, since $P_\delta(x)$ does not hold, none of the good events $\left(G^y_k\,;\,k\in{\llbracket n_0,n\llbracket}\right)$ can be realised.
In particular, we have 
\[\#\left\{k\in{\llbracket0,n\llbracket}:\text{$G^y_k$ is realised}\right\}\leq n_0\leq\frac{n}{3},\]
which proves \eqref{eq:K+starsdiscr}.
Now, by the union bound, we have
\begin{eqnarray*}
\lefteqn{\mathbb{P}\left(\exists y\in Y_n:\#\left\{k\in{\llbracket0,n\llbracket}:\text{$G^y_k$ is realised}\right\}\leq\frac{n}{3}\right)}\\
&\leq&\sum_{y\in Y_n}\mathbb{P}\left(\#\left\{k\in{\llbracket0,n\llbracket}:\text{$G^y_k$ is realised}\right\}\leq\frac{n}{3}\right)\\
&=&\#Y_n\cdot\mathbb{P}\left(\#\left\{k\in{\llbracket0,n\llbracket}:\text{$G_k$ is realised}\right\}\leq\frac{n}{3}\right)\\
&\leq&C\cdot r_n^{-d}\cdot r_n^{2d}=C\cdot r_n^d,
\end{eqnarray*}
where the last inequality holds for all sufficiently large $n$.
By the Borel--Cantelli lemma, we deduce that almost surely, for all sufficiently large $n$, there is no point $y\in Y_n$ for which $\#\left\{k\in{\llbracket0,n\llbracket}:\text{$G^y_k$ is realised}\right\}\leq n/3$, i.e, the set $\mathcal{Y}_n$ is empty.
It follows that almost surely, the set $\mathcal{X}_\delta\cap\overline{B}(0,R)$ is empty, as we wanted to prove.

The fact that almost surely, there are no $(m+1)^+$-stars in $\left(\mathbb{R}^d,T\right)$ is a straightforward consequence of the fact that almost surely, the property $P(x)$ holds for all $x\in\mathbb{R}^d$.
Indeed, by contraposition, if there exists an $(m+1)^+$-star in $\left(\mathbb{R}^d,T\right)$, say $x$, then there exists $(m+1)$ geodesics $\gamma_i:[0,\tau_i]\rightarrow\mathbb{R}^d$ emanating from $x$ such that $\gamma_i{]0,\tau_i]}\cap\gamma_j{]0,\tau_j]}=\emptyset$ for all $i\neq j$.
Then, for $\delta>0$ small enough so that $|x-\gamma(\tau_i)|>\delta$ for all $i$, the property $P_\delta(x)$ cannot hold, hence $P(x)$ does not hold.
\end{proof}

Now, we come to the proof of the lemma.

\begin{proof}[Proof of Lemma \ref{lem:K+stars}]
\emph{\underline{Step 1:} Construction of a ``master'' good event.}
As for the confluence of geodesics results in the previous subsection, we start by devising a master good event $G$ on which the desired property holds, and from which all the good events $G_k$ will be obtained by simple transformations.
This time, the construction relies on an idea of \cite[Theorem 6.1]{kahn}.
Let $A=\left.\overline{B}(0,1)\middle\backslash B(0,1/2)\right.$ be the annulus between radii $1/2$ and $1$ around $0$.
By Lemma \ref{lem:confluencetool1}, there exists $\rho_0\in{]0,1/2[}$ and a random variable $\Gamma$, measurable with respect to the restriction of $\Pi$ to the set of roads that pass through $\overline{B}(0,2)$ but do not pass through $\overline{B}(0,\rho_0)$, such that 
\[T(x,y)\leq\Gamma\quad\text{for all $x\neq y\in A$}.\]
We use here that the function ${s\in{]0,2]}\mapsto s^{(\beta-d)/(\beta-1)}\cdot\ln(4/s)^{1/(\beta-1)}}$ is bounded.
Moreover, there exists constants $C$ and $c$ such that
\[\mathbb{P}(\Gamma>t)\leq C\cdot\exp\left[-c\cdot t^{\beta-1}\right]\quad\text{for all $t\in\mathbb{R}_+^*$.}\]
In particular, we have $\mathbb{P}(\Gamma>t)\rightarrow0$ as $t\to\infty$.
Now, fix parameters $t\in\mathbb{R}_+^*$ and $m\in\mathbb{N}^*$ to be adjusted, and set $G^{t,m}=D^t\cap E^{t,m}$, where $D^t=(\Gamma\leq t)$, and $E^{t,m}$ is the event: ``there are no more than $m$ roads of $\Pi$ with speed at least $1/(4t)$ that pass through $\overline{B}(0,2)$''.
We claim that on $G^{t,m}$, there exists a cut set $\mathcal{Z}\subset A$ with cardinality $\#\mathcal{Z}\leq m\cdot\left(16m\cdot2^m+1\right)$\footnote{The exact dependence on $m$ is not important.}, such that any geodesic $V$-path from a point inside $\overline{B}(0,1/2)$ to a point outside $\overline{B}(0,1)$ must pass through $z$ for some $z\in\mathcal{Z}$.

Indeed, fix a realisation of $G^{t,m}$, and let $\gamma:[0,\tau]\rightarrow\mathbb{R}^d$ be a geodesic $V$-path from a point inside $\overline{B}(0,1/2)$ to a point outside $\overline{B}(0,1)$.
Let $\sigma_1<\sigma_2\in{[0,\tau]}$ be such that $\gamma(\sigma_1)\in\partial B(0,1/2)$ and ${\gamma(\sigma_2)\in\partial B(0,1)}$, and $\gamma(s)\in A$ for all $s\in[\sigma_1,\sigma_2]$\footnote{For instance, let $\sigma_2$ be the first exit time of $\gamma$ from $\overline{B}(0,1)$, and set $\sigma_1=\sup\{t\in{[0,\sigma_2[}:|\gamma(t)|<1/2\}$.}.
Observe that
\begin{equation}\label{eq:intersectionlength}
\sum_{\substack{(\ell,v)\in\Pi\\v\geq1/(4t)}}\mathrm{length}(\gamma[\sigma_1,\sigma_2]\cap\ell)\geq\frac{1}{4}.
\end{equation}
Indeed, on the one hand, we have $T(\gamma(\sigma_1),\gamma(\sigma_2))\leq\Gamma\leq t$.
On the other hand, we have
\[T(\gamma(\sigma_1),\gamma(\sigma_2))=\sum_{(\ell,v)\in\Pi}\frac{\mathrm{length}(\gamma[\sigma_1,\sigma_2]\cap\ell)}{v}\geq\sum_{\substack{(\ell,v)\in\Pi\\v<1/(4t)}}\frac{\mathrm{length}(\gamma[\sigma_1,\sigma_2]\cap\ell)}{1/(4t)},\]
hence
\[\sum_{\substack{(\ell,v)\in\Pi\\v<1/(4t)}}\mathrm{length}(\gamma[\sigma_1,\sigma_2]\cap\ell)\leq\frac{t}{4t}=\frac{1}{4}.\]
But we also have
\[\sum_{(\ell,v)\in\Pi}\mathrm{length}(\gamma[\sigma_1,\sigma_2]\cap\ell)\geq|\gamma(\sigma_2)-\gamma(\sigma_1)|\geq|\gamma(\sigma_2)|-|\gamma(\sigma_1)|=\frac{1}{2},\]
hence
\[\begin{split}
\sum_{\substack{(\ell,v)\in\Pi\\v\geq1/(4t)}}\mathrm{length}(\gamma[\sigma_1,\sigma_2]\cap\ell)&=\sum_{(\ell,v)\in\Pi}\mathrm{length}(\gamma[\sigma_1,\sigma_2]\cap\ell)-\sum_{\substack{(\ell,v)\in\Pi\\v<1/(4t)}}\mathrm{length}(\gamma[\sigma_1,\sigma_2]\cap\ell)\\
&\geq\frac{1}{2}-\frac{1}{4}=\frac{1}{4}.
\end{split}\]
Next, since there are no more than $m$ roads of $\Pi$ with speed at least $1/(4t)$ that pass through $\overline{B}(0,1)$, we deduce from \eqref{eq:intersectionlength} that there must exist a road $(\ell_0,v_0)\in\Pi$ with speed $v_0\geq1/(4t)$ such that 
\[\mathrm{length}(\gamma[\sigma_1,\sigma_2]\cap\ell_0)\geq\frac{1}{4m}.\]
Moreover, by the structure of geodesic $V$-paths recalled in Section \ref{sec:rappelsgeo}, the set $\gamma[\sigma_1,\sigma_2]\cap\ell_0$ is a finite union of intervals, among which at most $2^m$ are non-trivial.
Therefore, there must exist points $x\neq y\in\ell_0$ with $|x-y|\geq1/(4m)\cdot2^{-m}=:\delta$ such that the $V$-path $\gamma$ drives from $x$ to $y$ using the road $(\ell,v)$, over the time interval $[\sigma_1,\sigma_2]$.
In particular, note that $[x,y]\in A$.
To complete the proof of the claim, it suffices to argue that there exists a cut set $\mathcal{Z}\subset A$ with cardinality $\#\mathcal{Z}\leq m\cdot\left(16m\cdot2^m+1\right)$, such that for each road $(\ell,v)\in\Pi$ with speed $v\geq1/(4t)$, the following holds: for every point $w\in\ell\cap A$, there exists $z\in\mathcal{Z}$ such that $|w-z|\leq\delta/2$.
Applying this to $(\ell,v)=(\ell_0,v_0)$ and $w=(x+y)/2$ from above, we deduce that there exists $z\in\mathcal{Z}$ such that $z\in[x,y]$, hence $\gamma$ passes through $z$.
The existence of such a set $\mathcal{Z}$ is guaranteed by the usual covering argument: for each road $(\ell,v)\in\Pi$ with speed $v\geq1/(4t)$ that passes through $\overline{B}(0,1)$, there exists a covering $\left(B(z,\delta/2)\,;\,z\in\mathcal{Z}_{(\ell,v)}\right)$ of $\ell\cap A$ by balls of radius $\delta/2$, with centres $z\in\ell\cap A$ at least $\delta/2$ apart from each other.
Then, since the balls $\left(B(z,\delta/4)\,;\,\mathcal{Z}_{(\ell,v)}\right)$ are disjoint and included in $\ell\cap\overline{B}(0,1+\delta/4)$, a measure argument entails that $\#\mathcal{Z}_{(\ell,v)}\leq4/\delta+1=16m\cdot2^m+1$. 
Finally, setting $\mathcal{Z}$ to be the union of the $\mathcal{Z}_{(\ell,v)}$, over the roads $(\ell,v)\in\Pi$ with speed $v\geq1/(4t)$ that pass through $\overline{B}(0,1)$, we obtain what we wanted, with $\#\mathcal{Z}\leq m\cdot\left(16m\cdot2^m+1\right)$.

To conclude this first step of the proof, let us record that:
\begin{itemize}
\item On $G^{t,m}$, there exists a cut set $\mathcal{Z}\subset A$ with cardinality $\#\mathcal{Z}\leq m\cdot\left(16m\cdot2^m+1\right)$, such that any geodesic $V$-path from a point inside $\overline{B}(0,1/2)$ to a point outside $\overline{B}(0,1)$ must pass through $z$ for some $z\in\mathcal{Z}$.

\item The good event $G^{t,m}$ can be decomposed as the intersection $G^{t,m}=D^t\cap E^{t,m}$, where:
\begin{itemize}
\item the event $D^t=(\Gamma\leq t)$ is measurable with respect to the restriction of $\Pi$ to set of roads that pass through $\overline{B}(0,2)$ but do not pass through $\overline{B}(0,\rho_0)$, and has probability ${\mathbb{P}\left(D^t\right)=\mathbb{P}(\Gamma\leq t)}$ going to $1$ as $t\to\infty$,

\item $E^{t,m}$ is the event: ``there are no more than $m$ roads of $\Pi$ with speed at least $1/(4t)$ that pass through $\overline{B}(0,2)$''.
\end{itemize}
\end{itemize}

\smallskip

\emph{\underline{Step 2:} Multiscale argument.}
Set $r_k=(\rho_0/2)^k$ for all $k\in\mathbb{N}$, where $\rho_0\in{]0,1/2[}$ is the constant of step 1.
For each $k\in\mathbb{N}$, let $G_k^{t,m}$ be the good event $G^{t,m}$, but for the pushforward $\Pi\circ f_{0,2/r_k}^{-1}$ in place of $\Pi$, where $f_{x,r}$ is the map of \eqref{eq:defscaling}.
Moreover, set 
\[v_k=\frac{1}{4t}\cdot\left(\frac{r_k}{2}\right)^{(d-1)/(\beta-1)}\quad\text{for all $k\in\mathbb{N}$.}\]
Note that for every $(\ell,v)\in\mathbb{L}\times\mathbb{R}_+^*$, the road $(\ell,v)$ passes through $\overline{B}(0,r_k)$ \big(resp. $\overline{B}(0,r_{k+1})$\big) if and only if $f_{0,2/r_k}(\ell,v)$ passes through $\overline{B}(0,2)$ \big(resp. $\overline{B}(0,\rho_0)$\big), and has speed $v\geq v_k$ if and only if $f_{0,2/r_k}(\ell,v)$ has speed at least $1/(4t)$.
In particular, the following holds:
\begin{itemize}
\item On $G_k^{t,m}$, there exists a cut set $\mathcal{Z}\subset\left.\overline{B}(0,r_k)\middle\backslash\overline{B}(0,r_{k+1})\right.$ with cardinality 
\[\#\mathcal{Z}\leq m\cdot\left(16m\cdot2^m+1\right),\]
such that any geodesic $V$-path from a point inside $\overline{B}(0,r_{k+1})$ to a point outside $\overline{B}(0,r_k)$ must pass through $z$ for some $z\in\mathcal{Z}$.

\item The good event $G_k^{t,m}$ can be decomposed as the intersection $G_k^{t,m}=D_k^t\cap E_k^{t,m}$, where:
\begin{itemize}
\item the event $D_k^t$ is measurable with respect to the restriction of $\Pi$ to set of roads that pass through $\overline{B}(0,r_k)$ but do not pass through $\overline{B}(0,r_{k+1})$, and has probability ${\mathbb{P}\left(D_k^t\right)=\mathbb{P}(D^t)}$ going to $1$ as $t\to\infty$,

\item $E^{t,m}$ is the event: ``there are no more than $m$ roads of $\Pi$ with speed at least $v_k$ that pass through $\overline{B}(0,r_k)$''.
\end{itemize}
\end{itemize}
Finally, let us adjust $t$ and $m$ in order to obtain \eqref{eq:scalesK+stars}.
To begin with, observe that
\begin{eqnarray*}
\lefteqn{\mathbb{P}\left(\#\left\{k\in{\llbracket0,n\llbracket}:\text{$G_k^{t,m}$ is realised}\right\}\leq\frac{n}{3}\right)}\\
&\leq&\mathbb{P}\left(\#\left\{k\in{\llbracket0,n\llbracket}:\text{$D_k^t$ is not realised}\right\}\geq\frac{n}{3}\right)\\
&&+\mathbb{P}\left(\#\left\{k\in{\llbracket0,n\llbracket}:\text{$E_k^{t,m}$ is not realised}\right\}\geq\frac{n}{3}\right).
\end{eqnarray*}
We deal with the first term first, which will lead us to adjusting $t$.
Since the events $\left(D_k^t\right)_{k\in\mathbb{N}}$ are independent, the random variable $\#\left\{k\in{\llbracket0,n\llbracket}:\text{$D_k^t$ is not realised}\right\}$ is stochastically dominated by a binomial random variable with $n$ trials and success probability ${1-\mathbb{P}\left(D^t\right)=:q(t)}$.
By the Chernoff bound for binomial random variables, it follows that
\begin{equation}\label{eq:HkK+stars}
\mathbb{P}\left(\#\left\{k\in{\llbracket0,n\llbracket}:\text{$D_k^t$ is not realised}\right\}\geq\frac{n}{3}\right)\leq\left(\left(\frac{q(t)}{1/3}\right)^{1/3}\cdot\left(\frac{1-q(t)}{2/3}\right)^{2/3}\right)^n.
\end{equation}
Then, since we have $q(t)\rightarrow0$ as $t\to\infty$, we can fix $t$ large enough so that 
\[\left(\frac{q(t)}{1/3}\right)^{1/3}\cdot\left(\frac{1-q(t)}{2/3}\right)^{2/3}\leq\left(\frac{\rho_0}{2}\right)^{3d}.\]
Now that $t$ has been fixed, let us consider the other term.
To begin with, we write
\[\#\left\{k\in{\llbracket0,n\llbracket}:\text{$E_k^{t,m}$ is not realised}\right\}\leq\sum_{k=0}^{n-1}\frac{\Pi\left(\left\langle\overline{B}(0,r_k)\right\rangle\times[v_k,\infty[\right)}{m}.\]
Then, using the shorthand $v_{-1}=\infty$, we write
\[\begin{split}
\sum_{k=0}^{n-1}\Pi\left(\left\langle\overline{B}(0,r_k)\right\rangle\times[v_k,\infty[\right)&=\sum_{k=0}^{n-1}\sum_{j=0}^k\Pi\left(\left\langle\overline{B}(0,r_k)\right\rangle\times[v_j,v_{j-1}[\right)\\
&\leq\sum_{j=0}^{n-1}\sum_{k\geq j}\Pi\left(\left\langle\overline{B}(0,r_k)\right\rangle\times[v_j,v_{j-1}[\right)\\
&=\sum_{j=0}^{n-1}\sum_{l\geq0}\Pi\left(\left\langle\overline{B}\left(0,r_{j+l}\right)\right\rangle\times[v_j,v_{j-1}[\right)=:\sum_{j=0}^{n-1}X_j.
\end{split}\]
Thus, we obtain
\[\mathbb{P}\left(\#\left\{k\in{\llbracket0,n\llbracket}:\text{$E_k^{t,m}$ is not realised}\right\}\geq\frac{n}{3}\right)\leq\mathbb{P}\left(\sum_{j=0}^{n-1}X_j\geq m\cdot\frac{n}{3}\right).\]
Now, observe that the random variables $(X_j\,;\,j\in\llbracket0,n\llbracket)$ are independent, and that each one of them is stochastically dominated by the random variable
\[X=\sum_{k\geq0}\Pi\left(\left\langle\overline{B}(0,r_k)\right\rangle\times[v_0,\infty[\right).\]
At this point, we claim that
\begin{equation}\label{eq:momentestimate}
\mathbb{E}\left[e^{sX}\right]\leq\exp\left[\frac{\zeta^{-(\beta-1)}\cdot e^s}{1-(\rho_0/2)^{d-1}\cdot e^s}\right]\quad\text{for all $s<(d-1)\ln(2/\rho_0)$,}
\end{equation}
where $\zeta=1/(4t)\cdot2^{-(d-1)/(\beta-1)}$.
Provided that this moment estimate holds, let us conclude the argument.
Setting $s=(d-1)\ln(2/\rho_0)-\ln2$, we obtain
\[\mathbb{E}\left[e^{sX}\right]\leq\exp\left[\zeta^{-(\beta-1)}\cdot(\rho_0/2)^{-(d-1)}\right]=:C.\]
Then, by the exponential Markov inequality, we deduce that
\begin{equation}\label{eq:IkK+stars}
\mathbb{P}\left(\sum_{j=0}^{n-1}X_j\geq m\cdot\frac{n}{3}\right)\leq\mathbb{E}\left[e^{sX}\right]^n\cdot e^{-s\cdot m\cdot n/3}\leq C^n\cdot e^{-smn/3}=\left(C\cdot e^{-sm/3}\right)^n.
\end{equation}
To conclude, we choose $m$ large enough so that $C\cdot e^{-sm/3}\leq(\rho_0/2)^{3d}$.
Combining \eqref{eq:HkK+stars} and \eqref{eq:IkK+stars}, we finally obtain
\[\mathbb{P}\left(\#\left\{k\in{\llbracket0,n\llbracket}:\text{$G_k^{t,m}$ is realised}\right\}\leq\frac{n}{3}\right)\leq\left(\left(\frac{\rho_0}{2}\right)^{3d}\right)^n+\left(\left(\frac{\rho_0}{2}\right)^{3d}\right)^n=2\cdot r_n^{3d},\]
which readily yields \eqref{eq:scalesK+stars}.

To complete the proof of the lemma, it remains to check the moment estimate \eqref{eq:momentestimate}.
To this end, we write
\[\begin{split}
X&=\sum_{k\geq0}\sum_{i\geq k}\Pi\left(\left(\left\langle\overline{B}(0,r_i)\right\rangle\middle\backslash\left\langle\overline{B}(0,r_{i+1})\right\rangle\right)\times[v_0,\infty[\right)\\
&=\sum_{i\geq0}(i+1)\cdot\Pi\left(\left(\left\langle\overline{B}(0,r_i)\right\rangle\middle\backslash\left\langle\overline{B}(0,r_{i+1})\right\rangle\right)\times[v_0,\infty[\right)=:\sum_{i\geq0}(i+1)\cdot Y_i.
\end{split}\]
Now, the random variables $(Y_i)_{i\in\mathbb{N}}$ are independent Poisson random variables, and thus we obtain
\[\begin{split}
\mathbb{E}\left[e^{sX}\right]&=\prod_{i\geq0}\mathbb{E}\left[e^{s\cdot(i+1)\cdot Y_i}\right]\\
&=\prod_{i\geq0}\exp\left[\left(r_i^{d-1}-r_{i+1}^{d-1}\right)\cdot v_0^{-(\beta-1)}\cdot\left(e^{s\cdot(i+1)}-1\right)\right]\\
&\leq\prod_{i\geq0}\exp\left[r_i^{d-1}\cdot v_0^{-(\beta-1)}\cdot e^{s\cdot(i+1)}\right]\\
&=\prod_{i\geq0}\exp\left[\zeta^{-(\beta-1)}\cdot e^s\cdot\left((\rho_0/2)^{d-1}\cdot e^s\right)^i\right]=\exp\left[\frac{\zeta^{-(\beta-1)}\cdot e^s}{1-(\rho_0/2)^{d-1}\cdot e^s}\right],
\end{split}\]
for every $s<(d-1)\ln(2/\rho_0)$.
\end{proof}

\begin{rem}
Similarly as above, the arguments of the previous section can be adapted to work out within a multiscale argument, showing that:
\begin{enumerate}
\item There exists $\varepsilon>0$ such that almost surely, the set of points $x\in\mathbb{R}^d$ for which the following property $P(x)$ does not hold has Hausdorff dimension at most $d-\varepsilon$ for the Euclidean metric: ``for every $\delta>0$, there exists ${\eta\in{]0,\delta[}}$ and a cut point $z\in\left.\overline{B}(x,\delta)\middle\backslash\overline{B}(x,\eta)\right.$ such that any geodesic $V$-path from a point inside $\overline{B}(x,\eta)$ to a point outside $\overline{B}(x,\delta)$ must pass through $z$''.

\item There exists $\varepsilon>0$ such that almost surely, for each road $(\ell,v)\in\Pi$, the set of points $x\in\ell$ for which the following property $P_{(\ell,v)}(x)$ does not hold has Hausdorff dimension at most $1-\varepsilon$ for the Euclidean metric: ``for each $\delta>0$, there exists $\eta\in{]0,\delta[}$ such that any geodesic $V$-path from a point inside $\overline{B}(x,\eta)$ to a point outside $\overline{B}(x,\delta)$ must use $(\ell,v)$ within $\overline{B}(x,\delta)$''.
\end{enumerate}
\end{rem}

\section{Geodesics do not pause en route}\label{sec:pause}

In this section, we prove Theorem \ref{thm:pause}. 
An important ingredient in the proof is the notion of $(\varepsilon\vee V)$-path, which we recall now.
Given a realisation of $\Pi$, for each $\varepsilon>0$, we call $(\varepsilon\vee V)$-path any continuous path $\gamma:[0,\tau]\rightarrow\mathbb{R}^d$ that respects the speed limits set by $\varepsilon\vee V$, i.e, such that
\[|\gamma(t_2)-\gamma(t_1)|\leq\int_{t_1}^{t_2}\varepsilon\vee V(\gamma(t))\mathrm{d}t\quad\text{for all $t_1\leq t_2\in[0,\tau]$.}\]
{In words, any $(\varepsilon\vee V)$-path can use the roads of $\Pi$, respecting the speed limits, and drive everywhere at speed $\varepsilon$.}
Now, consider the (random) function
\[\begin{matrix}
T_\varepsilon:&\mathbb{R}^d\times\mathbb{R}^d&\longrightarrow&\mathbb{R}_+\\
&(x,y)&\longmapsto&\inf\left\{\tau>0:\text{there exists an $(\varepsilon\vee V)$-path $\gamma:[0,\tau]\rightarrow\mathbb{R}^d$ from $x$ to $y$}\right\}.
\end{matrix}\]
Contrary to the  case of $T$, it is straightforward that $T_\varepsilon$ is a well-defined distance function on $\mathbb{R}^d$.
We call geodesic $(\varepsilon\vee V)$-path any $(\varepsilon\vee V)$-path $\gamma:[0,\tau]\rightarrow\mathbb{R}^d$ with driving time ${\tau=T_\varepsilon(\gamma(0),\gamma(\tau))}$.
It is not difficult to check that almost surely, any geodesic $(\varepsilon\vee V)$-path ${\gamma:[0,\tau]\rightarrow\mathbb{R}^d}$ has the following structure: there exists a subdivision 
  \begin{eqnarray} \label{eq:subdivision}0=b_0\leq a_1<b_1\leq\ldots\leq a_k<b_k\leq a_{k+1}=\tau  \end{eqnarray}
of $[0,\tau]$, and a collection $(\ell_1,v_1),\ldots,(\ell_k,v_k)\in\Pi$ of roads with speed greater than $\varepsilon$, such that:
\begin{itemize}
\item For each $i\in\llbracket1,k\rrbracket$, the $(\varepsilon\vee V)$-path $\gamma$ drives from $\gamma(a_i)$ to $\gamma(b_i)$ using the road $(\ell_i,v_i)$,
\item For each $i\in\llbracket0,k\rrbracket$, the $(\varepsilon\vee V)$-path $\gamma$ drives from $\gamma(b_i)$ to $\gamma(a_{i+1})$ linearly at speed $\varepsilon$. 
\end{itemize}
Throughout this section, we call simple $V$-path any $V$-path ${\gamma:[0,\tau]\rightarrow\mathbb{R}^2}$ for which there exists a subdivision $0=t_0< t_1<\ldots<t_k=\tau$ of $[0,\tau]$, and a collection $(\ell_1,v_1),\ldots,(\ell_k,v_k)$ of roads of $\Pi$ such that for each $i\in\llbracket1,k\rrbracket$, the $V$-path $\gamma$ drives from $\gamma(t_{i-1})$ to $\gamma(t_i)$ using the road $(\ell_i,v_i)$.  
{The following result, which stems from an elementary deterministic lemma, sharpens the description of geodesic $(\varepsilon\vee V)$-paths: in words, their ability to drive at speed $\varepsilon$ is only used at their endpoints, and not within.}

\begin{prop}\label{prop:epsilongeo2}
Fix $\varepsilon>0$.
Almost surely, any geodesic $(\varepsilon\vee V)$-path $\gamma:[0,\tau]\rightarrow\mathbb{R}^2$ has the following structure: there exists a subdivision ${0\leq t_0<\ldots<t_k\leq \tau}$ of ${[0,\tau]}$, and a collection $(\ell_1,v_1),\ldots,(\ell_k,v_k)\in\Pi$ of roads with speed greater than $\varepsilon$, such that:
\begin{itemize}
\item The $(\varepsilon\vee V)$-path $\gamma$ drives from $\gamma(0)$ to $\gamma(t_0)$ linearly at speed $\varepsilon$,
\item For each $i\in\llbracket1,k\rrbracket$, the $(\varepsilon\vee V)$-path $\gamma$ drives from $\gamma(t_{i-1})$ to $\gamma(t_i)$ using the road $(\ell_i,v_i)$,
\item The $(\varepsilon\vee V)$-path $\gamma$ drives from $\gamma(t_k)$ to $\gamma(\tau)$ linearly at speed $\varepsilon$. 
\end{itemize}
In particular, the restriction of $\gamma$ to $[t_0,t_k]$ is a simple $V$-path.
\end{prop}
\begin{proof}
Fix a typical realisation of $\Pi$, and let $\gamma:[0,\tau]\rightarrow\mathbb{R}^2$ be a geodesic $(\varepsilon\vee V)$-path. Consider a subdivision $0=b_0\leq a_1<b_1\leq\ldots\leq a_k<b_k\leq a_{k+1}=\tau$ as in \eqref{eq:subdivision}.
We claim that in fact, for each $i\in\llbracket1,k\llbracket$, we have $\gamma(b_i)=\gamma(a_{i+1})$.
Provided that this holds, the result of the proposition readily follows, with $t_0=a_1$ and $t_i=b_i$ for all $i\in\llbracket1,k\rrbracket$.
The fact that the restriction of $\gamma$ to $[t_0,t_k]$ is a simple $V$-path is obvious, and it must be geodesic since $\gamma$ is geodesic as an $(\varepsilon\vee V)$-path.
The proof of the claim reads as follows: if there exists $i\in\llbracket1,k\llbracket$ such that $\gamma(a_i)\neq\gamma(b_{i+1})$, then by Lemma \ref{lem:elementarydet} below, we must have
\[\left|\left\langle\frac{\gamma(b_{i+1})-\gamma(a_i)}{|\gamma(b_{i+1})-\gamma(a_i)|},\ell_i\right\rangle\right|=\frac{\varepsilon}{v_i}\quad\text{and}\quad\left|\left\langle\frac{\gamma(b_{i+1})-\gamma(a_i)}{|\gamma(b_{i+1})-\gamma(a_i)|},\ell_{i+1}\right\rangle\right|=\frac{\varepsilon}{v_{i+1}}.\]
However, by Lemma \ref{lem:epsilongeo2lem2} below, almost surely, there is no pair of roads $(\ell_1,v_1)\neq(\ell_2,v_2)\in\Pi$ for which there exists a unit vector $e\in\mathbb{R}^2$ such that $|\langle e,\ell_1\rangle|=\varepsilon/v_1$ and $|\langle e,\ell_2\rangle|=\varepsilon/v_2$. See Figure \ref{fig:pause-route-fig}.

\begin{figure}[!h]
\begin{center}
\includegraphics[width=13cm]{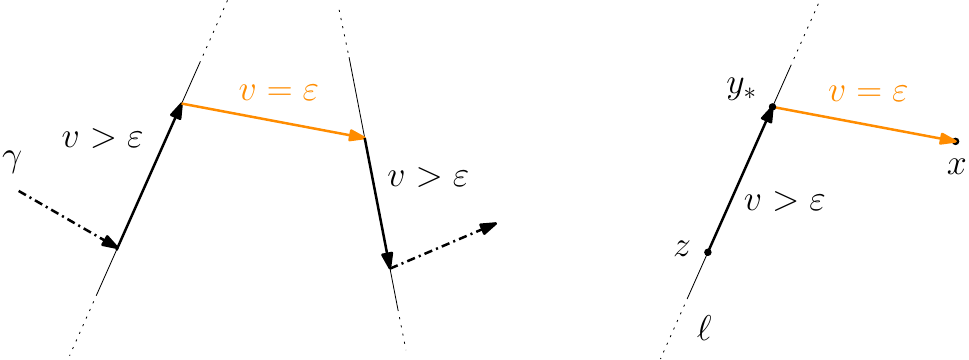}
\caption{\emph{Left:} Illustration of the proof of Proposition \ref{prop:epsilongeo2}: the combination of Lemma \ref{lem:elementarydet} and \ref{lem:epsilongeo2lem2} shows that almost surely, there is no configuration in $\Pi$ such that going from a road with speed $>\varepsilon$ to another road with speed $>\varepsilon$ by driving linearly at speed $\varepsilon$ is locally geodesic.\\
\emph{Right:} Setup of Lemma \ref{lem:elementarydet}. \label{fig:pause-route-fig}}
\end{center}
\end{figure}

To complete the proof of the lemma, we consider Lemma \ref{lem:elementarydet} and Lemma \ref{lem:epsilongeo2lem2}.

\begin{lem}\label{lem:elementarydet}
Fix $\varepsilon>0$, and fix a road $(\ell,v)\in\mathbb{L}\times\mathbb{R}_+^*$.
Then, fix a point $x\in\left.\mathbb{R}^d\middle\backslash\ell\right.$, and fix a point $z\in\ell$.
For every $y\in\ell\backslash\{z\}$, let 
\[\phi(y)=\frac{|y-x|}{\varepsilon}+\frac{|z-y|}{v}\]
be the time taken to drive from $x$ to $y$ linearly at speed $\varepsilon$, and then from $y$ to $z$ using the road $(\ell,v)$.
If $\phi$ is minimal at $y_*\in\ell\backslash\{z\}$, then we must have
\[\left|\left\langle\frac{y_*-x}{|y_*-x|},\ell\right\rangle\right|=\frac{\varepsilon}{v}.\]
\end{lem}
\begin{proof}
Let $e$ be a unit direction vector of $\ell$, and consider the function 
\[\psi:u\in\mathbb{R}^*\longmapsto\phi(z+ue)=\frac{|z+ue-x|}{\varepsilon}+\frac{|u|}{v}.\]
If $\phi$ is minimal at $y_*=z+u_*e$, then $\psi$ is minimal at $u_*$, hence $\psi'(u_*)=0$.
But the derivative of $\psi$ is given by
\[\psi'(u)=\frac{\langle z+ue-x,e\rangle/|z+ue-x|}{\varepsilon}+\frac{\sigma(u)}{v}\quad\text{for all $u\in\mathbb{R}^*$,}\]
where 
\[\sigma(u)=\begin{cases}
1&\text{if $u>0$}\\
-1&\text{otherwise.}
\end{cases}\]
We deduce that
\[\left|\frac{\langle z+u_*e-x,e\rangle}{|z+u_*e-x|}\right|=\frac{\varepsilon}{v},\]
which gives the result.
\end{proof}

\begin{lem}\label{lem:epsilongeo2lem2}
Fix $\varepsilon>0$.
Almost surely, there is no pair of roads $(\ell_1,v_1)\neq(\ell_2,v_2)\in\Pi$ for which there exists $e\in\mathbb{S}_1$ such that $|\langle e,\ell_1\rangle|=\varepsilon/v_1$ and $|\langle e,\ell_2\rangle|=\varepsilon/v_2$.
\end{lem}
\begin{proof}
Let $\chi:\left(\mathbb{L}\times\mathbb{R}_+^*\right)^2\rightarrow\{0,1\}$ be the indicator function defined by
\[\chi(\ell_1,v_1;\ell_2,v_2)=\begin{cases}
1&\text{if there exists $e\in\mathbb{S}_1$ such that $|\langle e,\ell_1\rangle|=\varepsilon/v_1$ and $|\langle e,\ell_2\rangle|=\varepsilon/v_2$}\\
0&\text{otherwise}
\end{cases}\]
for all $(\ell_1,v_1),(\ell_2,v_2)\in\mathbb{L}\times\mathbb{R}_+^*$.
Since the function 
\[\psi:(\ell_1,v_1;\ell_2,v_2;e)\in\left(\mathbb{L}\times\mathbb{R}_+^*\right)^2\times\mathbb{S}_1\longmapsto\left(|\langle e,\ell_1\rangle|-\frac{\varepsilon}{v_1},|\langle e,\ell_2\rangle|-\frac{\varepsilon}{v_2}\right)\]
is continuous, we have
\begin{eqnarray*}
\lefteqn{\left\{(\ell_1,v_1;\ell_2,v_2)\in\left(\mathbb{L}\times\mathbb{R}_+^*\right)^2:\text{$|\langle e,\ell_1\rangle|=\frac{\varepsilon}{v_1}$ and $|\langle e,\ell_2\rangle|=\frac{\varepsilon}{v_2}$}\right\}}\\
&=&\left\{(\ell_1,v_1;\ell_2,v_2)\in\left(\mathbb{L}\times\mathbb{R}_+^*\right)^2:\inf_{e\in D}|\psi(\ell_1,v_1;\ell_2,v_2;e)|=0\right\}
\end{eqnarray*}
for any countable dense subset $D\subset\mathbb{S}_1$, hence $\chi^{-1}\{1\}$ is a Borel subset of $\left(\mathbb{L}\times\mathbb{R}_+^*\right)^2$ and thus $\chi$ is measurable.
Now, we claim that for each road ${(\ell_1,v_1)\in\mathbb{L}\times\mathbb{R}_+^*}$, we have
\begin{equation}\label{eq:epsilongeo2lem2}
\int_\mathbb{L\times\mathbb{R}_+^*}\chi(\ell_1,v_1;\ell_2,v_2)\mathrm{d}\nu(\ell_2,v_2)=0.
\end{equation}
The result of the lemma is a direct consequence, since
\[\mathbb{E}\left[\sum_{(\ell_1,v_1)\neq(\ell_2,v_2)\in\Pi}\chi(\ell_1,v_1;\ell_2,v_2)\right]=\int_{\mathbb{L}\times\mathbb{R}_+^*}\int_{\mathbb{L}\times\mathbb{R}_+^*}\chi(\ell_1,v_1;\ell_2,v_2)\mathrm{d}\nu(\ell_2,v_2)\mathrm{d}\nu(\ell_1,v_1).\]
To prove \eqref{eq:epsilongeo2lem2}, it suffices to show that for each $w\in\mathbb{R}$, we have
\begin{equation}\label{eq:epsilongeo2lem2'}
\int_{\mathbf{SO}\left(\mathbb{R}^2\right)}\int_0^\infty\chi\left(\ell_1,v_1;g(w+\ell_0),v_2\right)\cdot v_2^{-\beta}\mathrm{d}v_2\mathrm{d}g=0.
\end{equation}
To this end, fix a unit direction vector $f_1$ of $\ell_1$.
For every $g\in\mathbf{SO}\left(\mathbb{R}^2\right)$, we have 
\[\chi(\ell_1,v_1;g(w+\ell_0),v_2)=1\]
if and only if there exists $e\in\mathbb{S}_1$ such that $|\langle e,f_1\rangle|=\varepsilon/v_1$ and $|\langle e,g(e_1)\rangle|=\varepsilon/v_2$, if and only if there exists $e=(x_1,x_2)\in\mathbb{S}_1$ such that
\[\begin{bmatrix}
\langle e_1,f_1\rangle&\langle e_2,f_1\rangle\\
\langle e_1,g(e_1)\rangle&\langle e_2,g(e_1)\rangle
\end{bmatrix}\begin{bmatrix}
x_1\\
x_2
\end{bmatrix}=\begin{bmatrix}
\pm\varepsilon/v_1\\
\pm\varepsilon/v_2\end{bmatrix}.\]
Now, fix signs $\sigma_1,\sigma_2\in\{-1,1\}$, and let us check that for almost every $g\in\mathbf{SO}\left(\mathbb{R}^2\right)$, and for almost every $v_2\in\mathbb{R}_+^*$, there is no unit vector $e=(x_1,x_2)$ such that
\begin{equation}\label{eq:epsilongeo2lem2''}
\begin{bmatrix}
\langle e_1,f_1\rangle&\langle e_2,f_1\rangle\\
\langle e_1,g(e_1)\rangle&\langle e_2,g(e_1)\rangle
\end{bmatrix}\begin{bmatrix}
x_1\\
x_2
\end{bmatrix}=\begin{bmatrix}
\sigma_1\varepsilon/v_1\\
\sigma_2\varepsilon/v_2\end{bmatrix}.
\end{equation}
First, recall that the pushforward of the Haar measure by the map $g\in\mathbf{SO}\left(\mathbb{R}^2\right)\mapsto g(e_1)$ is the uniform probability measure on $\mathbb{S}_1$, i.e, the law of $X/|X|$, where $X$ is a standard Gaussian random vector in $\mathbb{R}^2$.
In particular, for almost every $g\in\mathbf{SO}\left(\mathbb{R}^2\right)$, the pair $(f_1,g(e_1))$ is a basis of $\mathbb{R}^2$, hence the change-of-basis matrix 
\[A(g)=\begin{bmatrix}
\langle e_1,f_1\rangle&\langle e_2,f_1\rangle\\
\langle e_1,g(e_1)\rangle&\langle e_2,g(e_1)\rangle
\end{bmatrix}\]
is invertible.
Then, fix $g\in\mathbf{SO}\left(\mathbb{R}^2\right)$ such that $A(g)$ is invertible.
For each $v_2\in\mathbb{R}_+^*$, the only solution $(x_1,x_2)$ to \eqref{eq:epsilongeo2lem2''} is given by
\[\begin{bmatrix}
x_1(v_2)\\
x_2(v_2)
\end{bmatrix}=A(g)^{-1}\cdot\begin{bmatrix}
\sigma_1\varepsilon/v_1\\
\sigma_2\varepsilon/v_2\end{bmatrix}.\]
But from the right hand side, there is at most two values of $v_2$ for which 
\[(x_1(v_2),x_2(v_2))\in\mathbb{S}_1.\]
This yields \eqref{eq:epsilongeo2lem2'}, which completes the proof.
\end{proof}
\end{proof}

Now, we are ready to prove Theorem \ref{thm:pause}.

\begin{proof}[Proof of Theorem \ref{thm:pause}]
\emph{\underline{Step 1:} Geodesics between typical points do not pause en route.}
First, fix $x\neq y\in\mathbb{R}^2$, and let us prove that almost surely, the unique geodesic $V$-path $\gamma$ from $x$ to $y$ does not pause en route.
It suffices to prove that almost surely, for each $\delta\in{]0,|x-y|/2[}$, the $V$-path $\gamma$ is a simple $V$-path between its last exit point from $\overline{B}(x,\delta)$ and its first entry point in $\overline{B}(y,\delta)$.
Recall item \ref{item:confluencetyp} of Proposition \ref{prop:confluence}, and note that almost surely, the properties $P(x)$ and $P(y)$ hold simultaneously, where $P(w)$ is the property: ``for every $\delta>0$, there exists $\eta\in{]0,\delta[}$ and a cut point $z\in\left.\overline{B}(w,\delta)\middle\backslash\overline{B}(w,\eta)\right.$ such that any geodesic $V$-path from a point inside $\overline{B}(w,\eta)$ to a point outside $\overline{B}(w,\delta)$ must pass through $z$''.
Moreover, set ${\varepsilon_n=2^{-n}}$ for all $n\in\mathbb{N}$.
Recall the structure result of Proposition \ref{prop:epsilongeo2}, and note that almost surely, it applies to geodesic $(\varepsilon_n\vee V)$-paths simultaneously for all $n\in\mathbb{N}$.
Now, fix a typical realisation of $\Pi$, and fix ${\delta\in{]0,|x-y|/2[}}$.
By $P(x)$ and $P(y)$, there exists $\eta\in{]0,\delta[}$ and cut points $x'\in\left.\overline{B}(x,\delta)\middle\backslash\overline{B}(x,\eta)\right.$ and $y'\in\left.\overline{B}(y,\delta)\middle\backslash\overline{B}(y,\eta)\right.$ such that:
\begin{itemize}
\item Any geodesic $V$-path from a point inside $\overline{B}(x,\eta)$ to a point outside $\overline{B}(x,\delta)$ must pass through $x'$,

\item Any geodesic $V$-path from a point inside $\overline{B}(y,\eta)$ to a point outside $\overline{B}(y,\delta)$ must pass through $y'$.
\end{itemize}
Then, fix $n\in\mathbb{N}$ large enough so that $T(x,y)\cdot\varepsilon_n\leq\eta$, and let $\gamma_n:[0,T_{\varepsilon_n}(x,y)]\rightarrow\mathbb{R}^2$ be a geodesic $(\varepsilon_n\vee V)$-path from $x$ to $y$.
By the structure of geodesic $(\varepsilon\vee V)$-paths, there exists a subdivision ${0\leq t_0<\ldots<t_k\leq T_{\varepsilon_n}(x,y)}$ of ${[0,T_{\varepsilon_n}(x,y)]}$, and a collection $(\ell_1,v_1),\ldots,(\ell_k,v_k)\in\Pi$ of roads with speed greater than $\varepsilon_n$, such that:
\begin{itemize}
\item The $(\varepsilon_n\vee V)$-path $\gamma_n$ drives from $x$ to $\gamma_n(t_0)$ linearly at speed $\varepsilon_n$,
\item For each $i\in\llbracket1,k\rrbracket$, the $(\varepsilon_n\vee V)$-path $\gamma_n$ drives from $\gamma_n(t_{i-1})$ to $\gamma_n(t_i)$ using the road $(\ell_i,v_i)$,
\item The $(\varepsilon_n\vee V)$-path $\gamma_n$ drives from $\gamma_n(t_k)$ to $y$ linearly at speed $\varepsilon_n$. 
\end{itemize}
Note that we have
\[\frac{|x-\gamma_n(t_0)|}{\varepsilon_n}+\frac{|\gamma_n(t_k)-y|}{\varepsilon_n}\leq T_{\varepsilon_n}(x,y)\leq T(x,y),\]
hence 
\[|x-\gamma_n(t_0)|\leq T(x,y)\cdot\varepsilon_n\leq\eta\quad\text{and}\quad|\gamma_n(t_k)-y|\leq T(x,y)\cdot\varepsilon_n\leq\eta.\]
In particular, the restriction of $\gamma_n$ to $[t_0,t_k]$ is a geodesic $V$-path from a point inside $\overline{B}(x,\eta)$ and outside $\overline{B}(y,\delta)$ to a point inside $\overline{B}(y,\eta)$ and outside $\overline{B}(x,\delta)$.
Therefore, it must pass through the cut points $x'$ and $y'$.
Since the unique geodesic $V$-path $\gamma$ from $x$ to $y$ must also pass through $x'$ and $y'$, we deduce by uniqueness that $\gamma$ must agree with the restriction of $\gamma_n$ to $[t_0,t_k]$ between points $x'$ and $y'$.
In particular, the $V$-path $\gamma$ must be a simple $V$-path between these points, which concludes this first step of the proof.

\smallskip

\emph{\underline{Step 2:} Geodesics between arbitrary points do not pause en route.}
Let $D$ be a countable dense subset of $\mathbb{R}^2$.
By the previous step, almost surely, the unique geodesic $V$-path from $x$ to $y$ does not pause en route, simultaneously for all $x\neq y\in D$.
Now, fix a typical realisation of $\Pi$, let $\gamma:[0,\tau]\rightarrow\mathbb{R}^2$ be a geodesic $V$-path, and fix $t_1<t_2\in{]0,\tau[}$.
By Lemma \ref{lem:approxgeotyp}, there exists $y_1\neq y_2\in D$ and $t_1'<t_2'$ such that $\gamma[t_1,t_2]=\gamma_{y_1,y_2}[t_1',t_2']$, where $\gamma_{y_1,y_2}$ denotes the unique geodesic $V$-path from $y_1$ to $y_2$.
Since $\gamma_{y_1,y_2}$ does not pause en route, there exists $\varepsilon>0$ such that $V(\gamma_{y_1,y_2}(t))\geq\varepsilon$ for all $t\in[t_1',t_2']$.
It follows that $V(\gamma(t))\geq\varepsilon$ for all $t\in[t_1,t_2]$, which concludes the proof.
\end{proof}

\begin{rem}\label{rem:newuniqueness}
Here, we argued using the uniqueness of geodesic $V$-paths between typical points, which was proved by Kendall \cite[Theorem 4.1]{kendall} regardless of the fact that geodesics do not pause en route.
Note that there is another way to conduct the argument, which reads as follows.
First, deduce from the structure results of Proposition \ref{prop:epsilongeo2} the uniqueness of geodesic $(\varepsilon\vee V)$-paths between typical points.
Then, reproduce step 1 above, replacing the uniqueness of geodesic $V$-paths between typical points arguments with uniqueness of geodesic $(\varepsilon\vee V)$-paths between typical points arguments.
Conclude that geodesic $V$-paths between typical points do not pause en route, and are unique.
Then, step 2 carries through without change.
\end{rem}

\section{On the geodesic frame and hubs}

In this section, we study the geodesic frame of $\left(\mathbb{R}^d,T\right)$, and the geodesic hubs in the planar case $d=2$.
Most importantly, we prove Theorem \ref{thm:frame}. 
{The difficulty lies in showing that there is no exceptional point on a road $(\ell_0,v_0)\in\Pi$ from which one can start a geodesic that immediately leaves $(\ell_0,v_0)$ without using an intersection with another road of $\Pi$.
This is a strong form of Kendall's \cite[Theorem 4.3]{kendall}, in which the author assumes that the geodesic uses $(\ell_0,v_0)$ before leaving it.}

\subsection{The geodesic frame is the set of points on roads}

First, let us state the following straightforward consequence of the approximation of geodesics result by geodesics between typical points (Lemma \ref{lem:approxgeotyp}), which holds in general dimension $d\geq2$.

\begin{prop}\label{prop:approxgeotyp}
Let $D$ be a countable dense subset of $\mathbb{R}^d$.
For every $x\neq y\in D$, denote by ${\gamma_{x,y}:[0,T(x,y)]\rightarrow\mathbb{R}^d}$ the unique geodesic $V$-path from $x$ to $y$.
Almost surely, the geodesic frame of $\left(\mathbb{R}^d,T\right)$ equals $\bigcup_{x\neq y\in D}\gamma_{x,y}{]0,T(x,y)[}$.
In particular, it has Hausdorff dimension $1$, both with respect to the metric $T$ and to the Euclidean metric.
\end{prop}

The fact that the geodesic frame has Hausdorff dimension $1$ with respect to the random metric (here $T$), the dimension of a single geodesic, also holds in Brownian geometry or in Liouville quantum gravity, see \cite{millerqian}. 
A striking difference however is that in Liouville quantum gravity, geodesics are rough objects \cite{fangoswamiroughness}: the geodesic frame has Hausdorff dimension strictly greater than $1$ with respect to the Euclidean metric.

\begin{proof}
The converse inclusion holds automatically, and the direct inclusion readily follows from Lemma \ref{lem:approxgeotyp}.
The statement about the Hausdorff dimension follows from the countable stability of that functional, and the fact that $V$-paths are Lipschitz continuous with respect to the Euclidean metric.
\end{proof}

In the planar case $d=2$, the fact that geodesics do not pause en route allows to identify the geodesic frame of $\left(\mathbb{R}^2,T\right)$ as the set of points on roads $\mathcal{L}=\bigcup_{(\ell,v)\in\Pi}\ell$.
Moreover, recall the notion of a geodesic hub from the introduction.
We prove the following result:

\begin{prop}\label{prop:frameroads}
Almost surely, the geodesic frame of $\left(\mathbb{R}^2,T\right)$ is the set of points on roads $\mathcal{L}$.
Moreover, for every $x\in\mathcal{L}$, we have:
\begin{itemize}
\item if $x\notin\mathcal{I}$, then $x$ is a $2^+$-hub,

\item if $x\in\mathcal{I}$, then $x$ is a $4^+$-hub but not a $5^{+}$-star.
\end{itemize}
\end{prop}
\begin{proof}
The fact that the geodesic frame of $\left(\mathbb{R}^2,T\right)$ is contained in $\mathcal{L}$ readily follows from the fact that geodesics do not pause en route (Theorem \ref{thm:pause}).
The converse inclusion is implied by the second part of the proposition, which we turn to proving now.
The fact that points $x\in\mathcal{I}$ are not $5^+$-stars follows from item \ref{item:confluenceintersections} of Proposition \ref{prop:confluence}.
Although the remaining statements are rather easy to conceive, proving them in detail is a bit long.
We set ${\varepsilon_n=2^{-n}}$ for all $n\in\mathbb{N}$.
Recall Proposition \ref{prop:epsilongeo2}, and note that almost surely, this structure result applies to geodesic $(\varepsilon_n\vee V)$-paths simultaneously for all $n\in\mathbb{N}$.
Now, fix a typical realisation of $\Pi$, and let $x\in\mathcal{L}$.
\begin{itemize}
\item If $x\notin\mathcal{I}$, then there is exactly one road $(\ell,v)\in\Pi$ that passes through $x$.
Now, fix a parameter $\delta>0$ to be adjusted, and let $y\neq z\in\ell\cap\partial B(x,\delta)$ be the endpoints of the line segment $\ell\cap\overline{B}(x,\delta)$.
We claim that for $\delta$ small enough, the $V$-path that drives from $y$ to $z$ using $(\ell,v)$ is geodesic.
This implies that $x$ is a $2^+$-hub.
See Figure \ref{fig:hubs} for an illustration.
Indeed, assume that $\delta$ is small enough so that $(\ell,v)$ is the fastest road of $\Pi$ that passes through $\overline{B}(x,3\delta)$.
Then, any geodesic $V$-path from $y$ to $z$ must be contained in ${\overline{B}(x,3\delta)}$.
Indeed, since $(\ell,v)$ is the fastest road of $\Pi$ within ${\overline{B}(x,3\delta)}$, the driving time of any $V$-path from $y$ to a point outside $\overline{B}(x,3\delta)$ is strictly greater than $2\delta/v$, which is an obvious upper bound for $T(y,z)$.
In turn, since $(\ell,v)$ is the fastest road of $\Pi$ within ${\overline{B}(x,3\delta)}$, the triangle inequality proves that the $V$-path that drives from $y$ to $z$ using $(\ell,v)$ is geodesic.

\begin{figure}[!ht]
\centering
\includegraphics[width=0.74\linewidth]{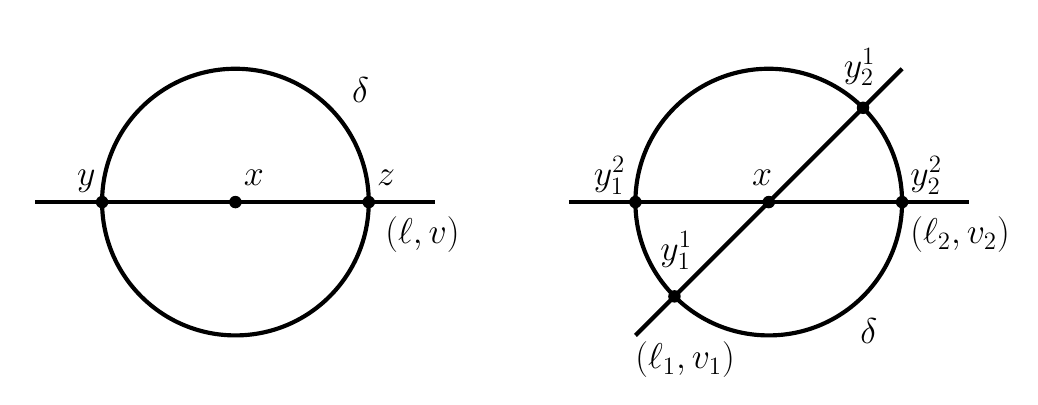}
\caption{Geodesic hubs in the set $\mathcal{L}$ of points on roads, in the planar case $d=2$.\\
\emph{Left:} If $x\notin\mathcal{I}$, then there is exactly one road $(\ell,v)\in\Pi$ that passes through $x$.
Fix $\delta>0$, and let $y\neq z\in\ell\cap\partial B(x,\delta)$ be the endpoints of the line segment $\ell\cap\overline{B}(x,\delta)$.
For $\delta$ small enough, the $V$-path that drives from $y$ to $z$ using $(\ell,v)$ is geodesic.\\
\emph{Right:} If $x\in\mathcal{I}$, then there is exactly two roads of $\Pi$ that pass through $x$, say $(\ell_1,v_1)$ and $(\ell_2,v_2)$, with $v_1>v_2$.
Fix $\delta>0$, and let $y^i_1\neq y^i_2\in\ell_i\cap\partial B(x,\delta)$ be the endpoints of the line segment $\ell_i\cap\overline{B}(x,\delta)$, for each $i\in\{1,2\}$.
For $\delta$ small enough, the following holds: for any indices $(i_1,j_1)\neq(i_2,j_2)\in\{1,2\}^2$, the $V$-path that drives from $y^{i_1}_{j_1}$ to $x$ using $\left(\ell_{i_1},v_{i_1}\right)$, and then from $x$ to $y^{i_2}_{j_2}$ using $\left(\ell_{i_2},v_{i_2}\right)$, is geodesic.}\label{fig:hubs}
\end{figure}

\item If $x\in\mathcal{I}$, then there is exactly two roads of $\Pi$ that pass through $x$, say $(\ell_1,v_1)$ and $(\ell_2,v_2)$, with $v_1>v_2$.
Now, fix a parameter $\delta>0$ to be adjusted, and let $y^i_1\neq y^i_2\in\ell_i\cap\partial B(x,\delta)$ be the endpoints of the line segment $\ell_i\cap\overline{B}(x,\delta)$, for each $i\in\{1,2\}$.
We claim that for $\delta$ small enough, the following holds: for any indices $(i_1,j_1)\neq(i_2,j_2)\in\{1,2\}^2$, the $V$-path that drives from $y^{i_1}_{j_1}$ to $x$ using $\left(\ell_{i_1},v_{i_1}\right)$, and then from $x$ to $y^{i_2}_{j_2}$ using $\left(\ell_{i_2},v_{i_2}\right)$, is geodesic.
This implies that $x$ is a $4^+$-hub.
See Figure \ref{fig:hubs} for an illustration.
Actually, the points $y^1_1$ and $y^1_2$ play the same role, as well as $y^2_1$ and $y^2_2$, so the values of $j_1$ and $j_2$ are not important: the relevant indices are $i_1$ and $i_2$.

The easiest case is when $i_1,i_2=1$, and can be handled in the exact same way as above.
To deal with the other cases, let us make the following observation: there exists $n\in\mathbb{N}$ large enough so that for any $z_1\in\ell_1\backslash\{x\}$ and $z_2\in\ell_2\backslash\{x\}$, driving from $z_1$ to $x$ using $(\ell_1,v_1)$ and then from $x$ to $z_2$ using $(\ell_2,v_2)$ is strictly faster than driving from $z_1$ to $z_2$ linearly at speed $\varepsilon_n$, i.e, we have
\begin{equation}\label{eq:noshortcut}
\frac{|a_1|}{v_1}+\frac{|a_2|}{v_2}<\frac{|a_1e_1-a_2e_2|}{\varepsilon_n}\quad\text{for all $a_1,a_2\in\mathbb{R}^*$,}
\end{equation}
where $e_1$ and $e_2$ are unit direction vectors of $\ell_1$ and $\ell_2$, respectively.
We will refer to this as the ``no shortcut'' condition.
Indeed, \eqref{eq:noshortcut} is equivalent to
\[\varepsilon_n<\frac{\sqrt{a_1^2-2\langle e_1,e_2\rangle\cdot a_1\cdot a_2+a_2^2}}{|a_1|/v_1+|a_2|/v_2}\quad\text{for all $a_1,a_2\in\mathbb{R}^*$,}\]
which holds if and only if
\[\varepsilon_n<\frac{\sqrt{1-2|\langle e_1,e_2\rangle|\cdot\rho+\rho^2}}{1/v_1+\rho/v_2}=:\phi(\rho)\quad\text{for all $\rho\in\mathbb{R}_+^*$.}\]
Now, observe that the function $\phi:\rho\in\mathbb{R}_+^*\mapsto\phi(\rho)$ is strictly positive and continuous, and that we have $\phi(\rho)\rightarrow v_1$ as $\rho\to0$ and $\phi(\rho)\rightarrow v_2$ as $\rho\to\infty$.
It follows that $\inf_{\rho\in\mathbb{R}_+^*}\phi(\rho)>0$, and thus we can fix $n\in\mathbb{N}$ large enough so that \eqref{eq:noshortcut} holds.
Note that we must have $\varepsilon_n<v_1$.
Now, assume that $\delta$ is small enough so that the third fastest road of $\Pi$ within ${\overline{B}(x,(1+2v_1/v_2)\delta)}$ has speed at most $\varepsilon_n$.
Then, for any indices ${i_1,i_2\in\{1,2\}}$, any geodesic $(\varepsilon_n\vee V)$-path from $y^{i_1}_1$ to $y^{i_2}_1$ must be contained in ${\overline{B}(x,(1+2v_1/v_2)\delta)}$.
Indeed, since $(\ell_1,v_1)$ is the fastest road of $\Pi$ within ${\overline{B}(x,(1+2v_1/v_2)\delta)}$, the driving time of any $(\varepsilon_n\vee V)$-path from $y^{i_1}_1$ to a point outside $\overline{B}(x,(1+2v_1/v_2)\delta)$ is strictly greater than $(2v_1/v_2\cdot\delta)/v_1=2\delta/v_2$, which is an obvious upper bound for $T_{\varepsilon_n}\left(y^{i_1}_1,y^{i_2}_2\right)$.

Now, let us consider the case $i_1=1$ and $i_2=2$.
By the structure results of Proposition \ref{prop:epsilongeo2}, there are three possibilities for a geodesic $(\varepsilon_n\vee V)$-path from $y^1_1$ to $y^2_1$:
\begin{itemize}
\item to drive from $y^1_1$ to $x$ using $(\ell_1,v_1)$, and then from $x$ to $y^2_1$ using $(\ell_2,v_2)$,

\item to drive from $y^1_1$ to a point $z\in\ell_1\backslash\{x\}$ using $(\ell_1,v_1)$, and then from $z$ to $y^2_1$ linearly at speed $\varepsilon_n$,

\item to drive from $y^1_1$ to a point $z\in\ell_2\backslash\{x\}$ linearly at speed $\varepsilon_n$, and then from $z$ to $y^2_1$ using $(\ell_2,v_2)$.
\end{itemize}
The second and third options are immediately ruled out by the ``no shortcut'' condition.
Thus, the $V$-path that drives from $y^1_1$ to $x$ using $(\ell_1,v_1)$, and then from $x$ to $y^2_1$ using $(\ell_2,v_2)$, is geodesic as an $(\varepsilon_n\vee V)$-path, hence as a $V$-path.

The case $i_1,i_2=2$ is not more difficult.
By the structure results of Proposition \ref{prop:epsilongeo2}, there are two possibilities for a geodesic $(\varepsilon_n\vee V)$-path from $y^2_1$ to $y^2_2$:
\begin{itemize}
\item to drive from $y^2_1$ to $y^2_2$ using $(\ell_2,v_2)$,

\item to drive from $y^2_1$ to a point $z_1\in\ell_1\backslash\{x\}$ linearly at speed $\varepsilon_n$, then from $z_1$ to a point $z_2\in\ell_1\backslash\{x\}$ using $(\ell_1,v_1)$, and then from $z_2$ to $y^2_2$ linearly at speed $\varepsilon_n$.
\end{itemize}
The second option is again ruled out by the ``no shortcut'' condition.
Indeed, 
\begin{itemize}
\item driving from $y^2_1$ to $x$ using $(\ell_2,v_2)$ is strictly faster than driving from $y^2_1$ to $x$ using $(\ell_2,v_2)$ and then from $x$ to $z_1$ using $(\ell_1,v_1)$, which by the ``no shortcut'' condition is strictly faster than driving from $y^2_1$ to $z_1$ linearly at speed $\varepsilon_n$,

\item driving from $x$ to $y^2_2$ using $(\ell_2,v_2)$ is strictly faster than driving from $z_2$ to $x$ using $(\ell_1,v_1)$ and then from $x$ to $y^2_2$ using $(\ell_2,v_2)$, which by the ``no shortcut'' condition is strictly faster than driving from $z_2$ to $y^2_2$ linearly at speed $\varepsilon_n$.
\end{itemize}
Therefore, the $V$-path that drives from $y^2_1$ to $y^2_2$ using $(\ell_2,v_2)$ is geodesic as an $(\varepsilon_n\vee V)$-path, hence as a $V$-path.
The proof of the proposition is complete.
\end{itemize}
\end{proof}

\subsection{Leaving a road by driving off-road is never geodesic}

Given Proposition \ref{prop:frameroads}, the only remaining point in Theorem \ref{thm:frame} is to prove that points $x \in \mathcal{L} \backslash \mathcal{I}$ are not $3^{+}$-stars, i.e, that  ``leaving a road by driving off-road is never geodesic''.
In view of the Mecke formula, it suffices to fix $(\ell_0,v_0)\in\mathbb{L}\times\mathbb{R}_+^*$, and to prove the following lemma, where $\Pi'=\Pi\oplus(\ell_0,v_0)$, and $V'$ denotes the ``speed limits'' function induced by $\Pi'$.

\begin{lem}\label{lem:offroad}
Almost surely, for each $x\in\ell_0$, if there exists a geodesic $V'$-path ${\gamma:[0,\tau]\rightarrow\mathbb{R}^2}$ with $\gamma(0)=x$ such that $\gamma(t)\notin\ell_0$ for all $t\in{]0,\tau]}$, then we must have $x\in\ell$ for some $(\ell,v)\in\Pi$.
\end{lem}

The rest of this subsection is devoted to proving this result.
Let us first present the overall idea of the proof to help the reader follow the arguments. 
The key idea is to consider the ``vertical speed'' $\phi(\ell,v)=\left|\left\langle\ell,\ell_0^\perp\right\rangle\right|\cdot v$ of roads $(\ell,v)\in\Pi$ around a point $x\in\ell_0$.
In fact, we can sort the roads $((\ell_i,v_i)\,;\,i \geq 1)$ intersecting say the ball of radius $1$ around $x$ by decreasing vertical speed. 
Let us assume now that the first few of these roads $((\ell_i,v_i)\,;\,i\in\llbracket1,k\rrbracket)$ have created a circuit $ \mathrm{C}_{ \varepsilon}$ surrounding the point $x$ at scale $ \varepsilon$, see Figure \ref{fig:cage}. 
Then, consider a geodesic $V'$-path $\gamma$ targeting $x$ and suppose that it does not use $(\ell_0,v_0)$.
Let $y$ be the last intersection point of $\gamma$ with the circuit $\mathrm{C}_{ \varepsilon}$, denote by $h$ its distance to $\ell_0$, and set $\gamma'$ to be the alternative path that drives to $\ell_0$ by using the roads of the circuit $ \mathrm{C}_{ \varepsilon}$, and then straight to $x$ using $(\ell_{0},v_{0})$.
By definition, the restriction $\gamma_{x,y}$ of $\gamma$ to the segment between $x$ and $y$ must be quicker than the alternative path $\gamma'$, and we shall reach a contradiction arguing differently according to $h$:

\begin{figure}[!h]
\begin{center}
\includegraphics[width=12cm]{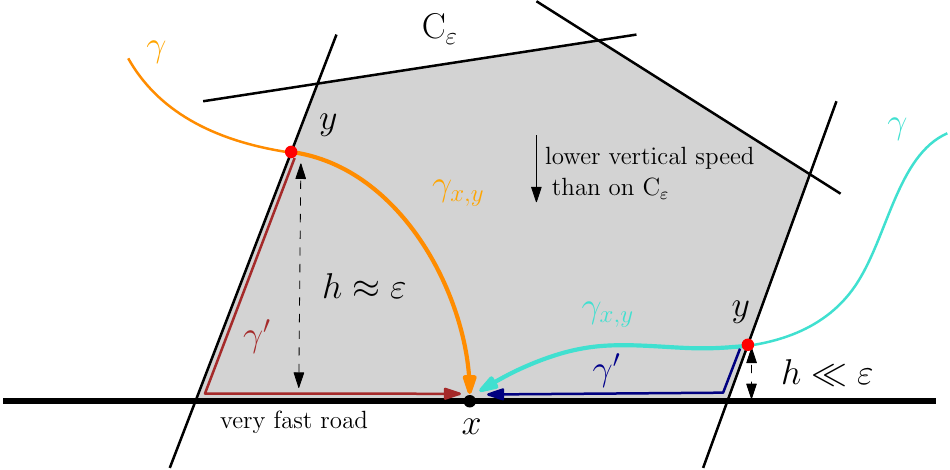}
\caption{Illustration of the idea of the proof of Lemma \ref{lem:offroad}.
All the roads passing through the polygon delimited by $\mathrm{C}_{\varepsilon}$ have smaller vertical speed than the roads constituting the circuit.\label{fig:cage}}
\end{center}
\end{figure}

\begin{itemize}
\item Either $h$ is comparable to $\varepsilon$ (see the orange curve on Figure \ref{fig:cage}). 
Then, since $\gamma'$ is longer than $\gamma_{x,y}$, the latter must use a road within the polygon delimited by $\mathrm{C}_\varepsilon$ whose vertical speed is very close to the smallest vertical speed of roads constituting the circuit: this is because $\gamma'$ uses roads with greater vertical speed than $\gamma_{x,y}$ to drive to $\ell_0$, and the driving time of the segment on $\ell_{0}$ is negligible {(typically, the largest vertical speed of a road within a ball of radius $\varepsilon$ is of order $\varepsilon^{1/(\beta-1)}$, whereas $v_0$ is a constant)}.
We then argue that it is not possible to have such a small gap between the vertical speed of roads in the vicinity of $x$.

\item Or $h$ is small compared to $\varepsilon$ (see the blue curve on Figure \ref{fig:cage}). 
Then, since $\gamma'$ is longer than $\gamma_{x,y}$, the latter must use a very fast road within the polygon delimited by $\mathrm{C}_\varepsilon$, because $\gamma'$ is short as it mainly uses the very fast road $(\ell_{0},v_{0})$. To rule out that possibility, we shall see in the coming Lemma \ref{lem:anormalroutes} that it is not possible to find fast roads at every scale when zomming in around $x \in \ell_{0}$, unless $x$ is at an intersection with a road of $\Pi$.
\end{itemize}

Let us first make the last statement precise.
Throughout this subsection, we use the following notation: for $x\in\mathbb{R}^2$ and $r>0$, we denote by $V^x_r=\sup_{\overline{B}(x,r)}V$ the maximal speed of a road that passes through $\overline{B}(x,r)$.
Note that $(\ell_{0},v_{0})$ does not count in $V^x_r$.
Typically, the quantity $V^x_r$ is of order $r^{1/(\beta-1)}$. 
The atypical behavior of these quantities have studied in depth in \cite[Section 4]{moa1fractal}. 
Although the most natural way to have $V^x_r$ atypically large for all sufficiently small $r$ is to choose $x\in \mathcal{L}$, Proposition 4.2 in \cite{moa1fractal} shows that for each $\alpha\in{]0,1/(\beta-1)[}$, almost surely, there are exceptional points $x\in\left.\mathbb{R}^2\middle\backslash\mathcal{L}\right.$ such that $V^x_r \geq r^\alpha$ for all sufficiently small $r$. 
We show however that those pathological points cannot be found if we restrict attention to the line $\ell_{0}$.

\begin{lem}\label{lem:anormalroutes}
Fix $\rho\in{]0,1[}$, and set $r_k=\rho^k$ for all $k\in\mathbb{N}$.
For each ${\alpha\in{]0,1/(\beta-1)[}}$, almost surely, we have
\[\left\{x\in\ell_0:\text{$V^x_{r_n}\geq r_n^\alpha$ for all sufficiently large $n$}\right\}=\bigcup_{(\ell,v)\in\Pi}\ell_0\cap\ell.\]
\end{lem}
\begin{proof}
The converse inclusion is clear.
In the other direction, fix ${\alpha\in{]0,1/(\beta-1)[}}$, and let us show that for each $R>0$, almost surely, we have
\[\left\{x\in\ell_0\cap\overline{B}(0,R):\text{$V^x_{r_n}\geq r_n^\alpha$ for all sufficiently large $n$}\right\}\subset\bigcup_{(\ell,v)\in\Pi}\ell_0\cap\ell.\]
Without loss of generality by the invariance in distribution of $\Pi$ under rotations and translations, we assume that $\ell_0=\mathbb{R}\times\{0\}$.
Fix $n_0\in\mathbb{N}$, and for every $n\geq n_0$, let
\[\mathcal{X}_n=\left\{x\in\ell_0\cap\overline{B}(0,R):\text{$V^x_{r_k}\geq r_k^\alpha$ for all $k\in\llbracket n_0,n\rrbracket$}\right\}.\]
We want to prove that almost surely, we have 
\begin{equation}\label{eq:goal1lemanormalroutes1}
\bigcap_{n\geq n_0}\mathcal{X}_n\subset\bigcup_{(\ell,v)\in\Pi}\ell_0\cap\ell.
\end{equation}
Let $\chi:\left(\mathbb{L}\times\mathbb{R}_+^*\right)\times\mathbb{M}\rightarrow\{0,1\}$ be the indicator function such that ${\chi(\ell,v;\Pi)=1}$ if and only if there exists $n\geq n_0$ such that $v\geq r_n^\alpha$ and $[\ell]_{r_n}$ meets $\mathcal{X}_n$ (i.e, such that $(\ell,v)$ takes part in $\mathcal{X}_n$).
Note that $\chi(\ell,v;\Pi)=1$ is equivalent to $[\ell]_{r_m}$ meets $\mathcal{X}_m$, where $m$ is the smallest integer $n\geq n_0$ such that $v\geq r_n^\alpha$.
We will refer to this as the ``key property'' of $\chi$ below.
To prove \eqref{eq:goal1lemanormalroutes1}, it suffices to show that
\begin{equation}\label{eq:finite}
\sum_{(\ell,v)\in\Pi}\chi(\ell,v;\Pi)<\infty\quad\text{almost surely.}
\end{equation}
Indeed, if \eqref{eq:finite} holds, then let $\mathcal{R}$ be the set of roads $(\ell,v)\in\Pi$ such that ${\chi(\ell,v;\Pi)=1}$.
By definition, we have $\mathcal{X}_n\subset\bigcup_{(\ell,v)\in\mathcal{R}}\ell_0\cap[\ell]_{r_n}$ for all $n\geq n_0$, hence
\[\bigcap_{n\geq n_0}\mathcal{X}_n\subset\bigcap_{n\geq n_0}\left(\bigcup_{(\ell,v)\in\mathcal{R}}\ell_0\cap[\ell]_{r_n}\right)=\bigcup_{(\ell,v)\in\mathcal{R}}\ell_0\cap\ell,\]
where the last equality holds because $\mathcal{R}$ is finite.

Now, let us dive into the proof of \eqref{eq:finite}.
For every $n\geq n_0$, let $\Pi_n$ be the restriction of $\Pi$ to the set of roads with speed at least $r_n^\alpha$, and let $\Pi_n'$ be the restriction of $\Pi$ to the set of roads with speed less than $r_n^\alpha$ and at least $r_{n+1}^\alpha$.
Moreover, let $\mathcal{F}_n$ be the $\sigma$-algebra generated by $\Pi_n$.
Note that $\Pi_n'$ is independent of $\mathcal{F}_n$.
We introduce
\[\Phi_n=\sum_{(\ell,v)\in\Pi_n}\chi(\ell,v;\Pi)\cdot\psi(\ell),\]
where $\psi:\mathbb{L}\rightarrow[1,\infty]$ is the function such that for all $\ell\in\mathbb{L}$, the half length of the line segment $\ell_0\cap[\ell]_r$ equals $\psi(\ell)\cdot r$.
See Figure \ref{fig:psi} for an illustration.
By the key property of $\chi$, we have
\begin{equation}\label{eq:keychi}
\chi(\ell,v;\Pi)=\chi(\ell,v;\Pi_n)\quad\text{for all $(\ell,v)\in\Pi_n$.}
\end{equation}
In particular, it follows that $\Phi_n$ is $\mathcal{F}_n$-measurable.
Since $\psi(\ell)\geq1$ for all $\ell\in\mathbb{L}$, we have
\[\sum_{(\ell,v)\in\Pi_n}\chi(\ell,v;\Pi)\leq\Phi_n\quad\text{for all $n\geq n_0$.}\]
Therefore, to obtain \eqref{eq:finite}, it suffices to prove that $\sup_{n\geq n_0}\Phi_n<\infty$ almost surely.
To this end, we will show that almost surely, we have
\begin{equation}\label{eq:goallemanormalroutes2}
\Phi_{n+1}-\Phi_n\leq r_n^{\varepsilon/3}\cdot\Phi_n\quad\text{for all sufficiently large $n$,}
\end{equation}
where $\varepsilon=1-\alpha(\beta-1)>0$.

\begin{figure}
\centering
\includegraphics[width=0.7\linewidth]{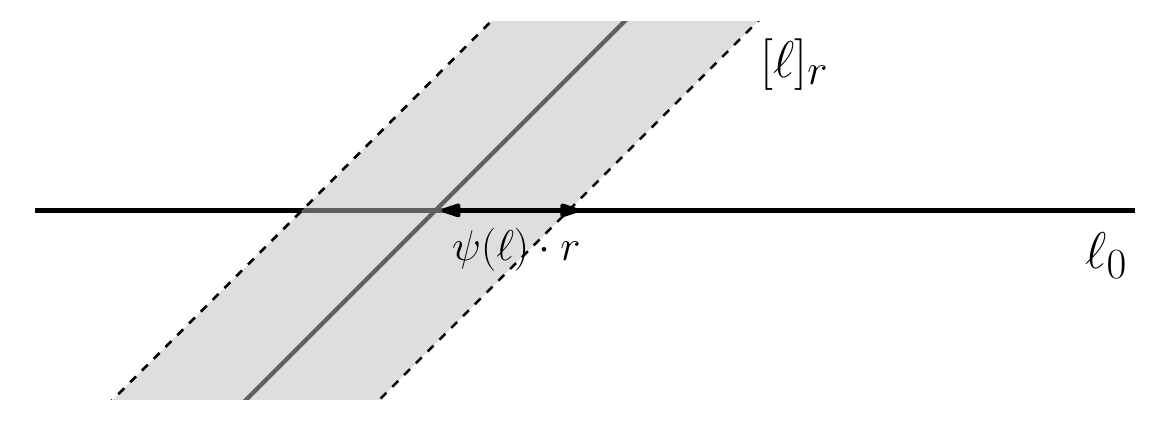}
\caption{In gray, we have represented the $r$-neighbourhood $[\ell]_r=\left\{x\in\mathbb{R}^2:d(x,\ell)\leq r\right\}$ of $\ell$. Its trace on $\ell_0$ is a line segment whose half length equals $\psi(\ell)\cdot r$, by definition.}\label{fig:psi}
\end{figure}

Fix $n\geq n_0$.
We have
\begin{equation}\label{eq:DeltaPhin}
\Phi_{n+1}-\Phi_n=\sum_{(\ell',v')\in\Pi_n'}\chi(\ell',v';\Pi).
\end{equation}
Furthermore, for each $(\ell',v')\in\Pi_n'$, we have
\begin{equation}\label{eq:chi}
\chi(\ell',v';\Pi)\leq\sum_{(\ell,v)\in\Pi_n}\chi(\ell,v;\Pi)\cdot\mathbf{1}\left(\text{$[\ell']_{r_{n+1}}$ meets $\ell_0\cap[\ell]_{r_n}$}\right).
\end{equation}
Indeed, if $\chi(\ell',v';\Pi)=1$, then by the key property of $\chi$, there must exist $x\in\mathcal{X}_{n+1}$ such that $d(x,\ell')\leq r_{n+1}$.
Next, since $x\in\mathcal{X}_n$, there must exist $(\ell,v)\in\Pi_n$ such that $d(x,\ell)\leq r_n$.
In particular, we have $\chi(\ell,v;\Pi)=1$.
Moreover, since $x$ belongs both to $[\ell']_{r_{n+1}}$ and to $\ell_0\cap[\ell]_{r_n}$, we conclude that $[\ell']_{r_{n+1}}$ meets $\ell_0\cap[\ell]_{r_n}$.
This justifies \eqref{eq:chi}.
Plugging \eqref{eq:chi} into \eqref{eq:DeltaPhin}, we deduce that
\[\Phi_{n+1}-\Phi_n\leq\sum_{(\ell',v')\in\Pi_n'}\sum_{(\ell,v)\in\Pi_n}\chi(\ell,v;\Pi)\cdot\mathbf{1}\left(\text{$[\ell']_{r_{n+1}}$ meets $\ell_0\cap[\ell]_{r_n}$}\right)\cdot\psi(\ell')=:\Psi_n.\]
Now, we want to compute $\mathbb{E}[\Psi_n\,|\,\mathcal{F}_n]$.
By \eqref{eq:keychi}, we have
\[\Psi_n=\sum_{(\ell',v')\in\Pi_n'}\sum_{(\ell,v)\in\Pi_n}\chi(\ell,v;\Pi_n)\cdot\mathbf{1}\left(\text{$[\ell']_{r_{n+1}}$ meets $\ell_0\cap[\ell]_{r_n}$}\right)\cdot\psi(\ell'),\]
hence
\begin{eqnarray*}
\lefteqn{\mathbb{E}[\Psi_n\,|\,\mathcal{F}_n]}\\
&=&c\int_\mathbb{L}\int_{r_{n+1}^\alpha}^{r_n^\alpha}\sum_{(\ell,v)\in\Pi_n}\chi(\ell,v;\Pi_n)\cdot\mathbf{1}\left(\text{$[\ell']_{r_{n+1}}$ meets $\ell_0\cap[\ell]_{r_n}$}\right)\cdot\psi(\ell')\cdot v^{-\beta}\mathrm{d}v\mathrm{d}\mu(\ell'),
\end{eqnarray*}
where $c=2^{-1}\cdot(\beta-1)$ is the multiplicative constant in the intensity measure of $\Pi$.
Unfortunately, for every $\ell\in\mathbb{L}$, we have
\[\int_\mathbb{L}\psi(\ell')\cdot\mathbf{1}\left(\text{$[\ell']_{r_{n+1}}$ meets $\ell_0\cap[\ell]_{r_n}$}\right)\mathrm{d}\mu(\ell')=\infty.\]
To prevent this degeneracy, we add the following cutoff.
Consider the bad event $B_n$: ``there exists a road $(\ell',v')\in\Pi_n'$ with ${\psi(\ell')\geq1/r_n}$ such that $[\ell']_{r_{n+1}}$ meets $\ell_0\cap\overline{B}(0,R)$''.
We claim that the function $\psi$ has the following ``key property'' with respect to this cutoff: there exists a constant $C$ such that
\begin{equation}\label{eq:bad}
\mathbb{P}(B_n)\leq C\cdot r_n^{1+\varepsilon},
\end{equation}
and such that for every $\ell\in\mathbb{L}$, we have
\begin{equation}\label{eq:psi}
\int_{\psi(\ell')\leq1/r_n}\psi(\ell')\cdot\mathbf{1}(\text{$[\ell']_{r_{n+1}}$ meets $\ell_0\cap[\ell]_{r_n}$})\mathrm{d}\mu(\ell')\leq C\cdot\psi(\ell)\cdot r_n\cdot\ln(1/r_n).
\end{equation}
Provided that this holds, let us conclude the argument.
On the complement of the bad event $B_n$, we have
\[\Psi_n=\sum_{\substack{(\ell',v')\in\Pi_n'\\\psi(\ell')\leq1/r_n}}\sum_{(\ell,v)\in\Pi_n}\chi(\ell,v;\Pi_n)\cdot\mathbf{1}\left(\text{$[\ell']_{r_{n+1}}$ meets $\ell_0\cap[\ell]_{r_n}$}\right)\cdot\psi(\ell')=:\Psi_n',\]
where
\begin{eqnarray*}
\lefteqn{\mathbb{E}[\Psi_n'\,|\,\mathcal{F}_n]}\\
&=&c\int_{\psi(\ell')\leq1/r_n}\int_{r_{n+1}^\alpha}^{r_n^\alpha}\sum_{(\ell,v)\in\Pi_n}\chi(\ell,v;\Pi_n)\cdot\mathbf{1}\left(\text{$[\ell']_{r_{n+1}}$ meets $\ell_0\cap[\ell]_{r_n}$}\right)\cdot\psi(\ell')\cdot v^{-\beta}\mathrm{d}v\mathrm{d}\mu(\ell').
\end{eqnarray*}
By \eqref{eq:psi}, we obtain that
\[\mathbb{E}[\Psi_n'\,|\,\mathcal{F}_n]\leq C'\cdot r_n\cdot\ln(1/r_n)\cdot r_n^{-\alpha(\beta-1)}\cdot\sum_{(\ell,v)\in\Pi_n}\chi(\ell,v;\Pi_n)\cdot\psi(\ell),\]
for some constant $C'$.
Now, by the definition of $\varepsilon$, we have 
\[r_n\cdot\ln(1/r_n)\cdot r_n^{-\alpha(\beta-1)}=o\left(r_n^{2\varepsilon/3}\right)\quad\text{as $n\to\infty$.}\]
In particular, there exists a constant $C''$ such that ${r_n\cdot\ln(1/r_n)\cdot r_n^{-\alpha(\beta-1)}\leq C''\cdot r_n^{2\varepsilon/3}}$ for all $n\geq n_0$, and we obtain
\begin{equation}\label{eq:Psin'}
\mathbb{E}[\Psi_n'\,|\,\mathcal{F}_n]\leq C''\cdot r_n^{2\varepsilon/3}\cdot\Phi_n.
\end{equation}
Finally, we write
\[\mathbb{P}\left(\Phi_{n+1}-\Phi_n>r_n^{\varepsilon/3}\cdot\Phi_n\right)\leq\mathbb{P}(B_n)+\mathbb{P}\left(\Psi_n'>r_n^{\varepsilon/3}\cdot\Phi_n\right).\]
The first term is controlled by \eqref{eq:bad}, and for the second term, by \eqref{eq:Psin'} and by Lemma \ref{lem:markovcond} below, we have
\[\mathbb{P}\left(\Psi_n'>r_n^{\varepsilon/3}\cdot\Phi_n\right)\leq C''\cdot r_n^{\varepsilon/3}.\]
We deduce that
\[\sum_{n\geq n_0}\mathbb{P}\left(\Phi_{n+1}-\Phi_n>r_n^{\varepsilon/3}\cdot\Phi_n\right)<\infty,\]
which proves \eqref{eq:goallemanormalroutes2} by the Borel--Cantelli lemma.

\begin{lem}\label{lem:markovcond}
Let $X$ and $Y$ be nonnegative random variables.
Suppose that $X$ is measurable with respect to some sub-$\sigma$-algebra $\mathcal{F}$, and that there exists a constant $C$ such that
\[\mathbb{E}[Y\,|\,\mathcal{F}]\leq C\cdot X\quad\text{almost surely.}\]
Then, we have
\[\mathbb{P}(Y>\lambda\cdot X)\leq\frac{C}{\lambda}\quad\text{for all $\lambda\in\mathbb{R}_+^*$.}\]
\end{lem}
\begin{proof}[Proof of the lemma]
Let $\lambda\in\mathbb{R}_+^*$.
As $\mathbb{P}(Y>\lambda\cdot X)=\lim_{\eta\downarrow0}\mathbb{P}(Y\geq\eta+\lambda\cdot X)$, it suffices to prove that for each $\eta>0$, we have $\mathbb{P}(Y\geq\eta+\lambda\cdot X)\leq C/\lambda$.
In fact, for each $\eta>0$, we have $\mathbb{P}(Y\geq\eta+\lambda\cdot X)=\mathbb{E}[\mathbb{P}(Y\geq\eta+\lambda\cdot X\,|\,\mathcal{F})]$, with
\[\mathbb{P}(Y\geq\eta+\lambda\cdot X\,|\,\mathcal{F})\leq\frac{\mathbb{E}[Y\,|\,\mathcal{F}]}{\eta+\lambda\cdot X}\leq\frac{C\cdot X}{\eta+\lambda\cdot X}\leq\frac{C}{\lambda}\quad\text{almost surely,}\]
by the conditional Markov inequality.
\end{proof}

Finally, it remains to check the key property of $\psi$, namely that \eqref{eq:bad} and \eqref{eq:psi} hold.
Let us start with \eqref{eq:bad}.
We have
\[\mathbb{P}(B_n)\leq c\cdot\mu\left\{\ell'\in\mathbb{L}:\text{$\psi(\ell')\geq1/r_n$ and $[\ell']_{r_{n+1}}$ meets $\ell_0\cap\overline{B}(0,R)$}\right\}\cdot\int_{r_{n+1}^\alpha}^{r_n^\alpha}v^{-\beta}\mathrm{d}v,\]
where $c=2^{-1}\cdot(\beta-1)$ is the multiplicative constant in the intensity measure of $\Pi$.
To compute the right hand side, recall the definition of the measure $\mu$ from Section \ref{sec:rappelsgeo}.
In the planar case $d=2$, the group $\mathbf{SO}\left(\mathbb{R}^2\right)$ consists of the rotations 
\[\left(g_\theta:z\in\mathbb{R}^2=\mathbb{C}\mapsto e^{i\theta}\cdot z,\,\theta\in[-\pi,\pi]\right).\]
Moreover, we have
\[\int_{\mathbf{SO}\left(\mathbb{R}^2\right)}\varphi(g)\mathrm{d}g=\frac{1}{2\pi}\int_{-\pi}^\pi\varphi(g_\theta)\mathrm{d}\theta\quad\text{for all Borel functions $\varphi:\mathbf{SO}\left(\mathbb{R}^2\right)\rightarrow\mathbb{R}_+$.}\]
Note that when $\ell'=g_\theta(w+\ell_0)$, we have
\[\ell_0\cap[\ell']_r=\left\{\left(-\frac{y}{\sin\theta},0\right)\,;\,y\in[w-r,w+r]\right\}\quad\text{and}\quad\psi(\ell')=\frac{1}{|\sin\theta|}.\]
It follows that
\begin{eqnarray*}
\lefteqn{\mu\left\{\ell'\in\mathbb{L}:\text{$\psi(\ell')\geq1/r_n$ and $[\ell']_{r_{n+1}}$ meets $\ell_0\cap\overline{B}(0,R)$}\right\}}\\
&=&\frac{1}{2\pi}\int_{1/|\sin\theta|\geq1/r_n}\int_{-\infty}^\infty\mathbf{1}\left(\exists y\in[w-r_{n+1},w+r_{n+1}]:\left|\frac{y}{\sin\theta}\right|\leq R\right)\mathrm{d}w\mathrm{d}\theta\\
&\leq&\frac{1}{2\pi}\int_{|\sin\theta|\leq r_n}\int_{-\infty}^\infty\mathbf{1}\left(|w|\leq R\cdot|\sin\theta|+r_{n+1}\right)\mathrm{d}w\mathrm{d}\theta\\
&=&\frac{1}{2\pi}\int_{|\sin\theta|\leq r_n}2\cdot(R\cdot|\sin\theta|+r_{n+1})\mathrm{d}\theta\\
&\leq&\frac{1}{2\pi}\int_{|\sin\theta|\leq r_n}2\cdot(R\cdot r_n+r_n)\mathrm{d}\theta=\frac{R+1}{\pi}\cdot r_n\cdot\int_{|\sin\theta|\leq r_n}\mathrm{d}\theta.
\end{eqnarray*}
In particular, there exists a constant $C=C(R)$ such that
\[\mu\left\{\ell'\in\mathbb{L}:\text{$\psi(\ell')\geq1/r_n$ and $[\ell']_{r_{n+1}}$ meets $\ell_0\cap\overline{B}(0,R)$}\right\}\leq C\cdot r_n^2.\]
We deduce that there exists a constant $C'$ such that $\mathbb{P}(B_n)\leq C'\cdot r_n^2\cdot r_n^{-\alpha(\beta-1)}$, which yields \eqref{eq:bad}.
Finally, let us check \eqref{eq:psi}.
Fix $\ell\in\mathbb{L}$, and without loss of generality by the invariance of $\mu$ under translations, assume that $\ell_0\cap\ell=\{0\}$.
In particular, we have 
\[\ell_0\cap[\ell]_{r_n}=\{(x,0)\,;\,x\in[-\psi(\ell)\cdot r_n,\psi(\ell)\cdot r_n]\}.\]
It follows that
\begin{eqnarray*}
\lefteqn{\int_{\psi(\ell')\leq1/r_n}\psi(\ell')\cdot\mathbf{1}\left(\text{$[\ell']_{r_{n+1}}$ meets $\ell_0\cap[\ell]_{r_n}$}\right)\mathrm{d}\mu(\ell')}\\
&=&\frac{1}{2\pi}\int_{1/|\sin\theta|\leq1/r_n}\int_{-\infty}^\infty\frac{1}{|\sin\theta|}\cdot\mathbf{1}\left(\exists y\in[w-r_{n+1},w+r_{n+1}]:\left|\frac{y}{\sin\theta}\right|\leq\psi(\ell)\cdot r_n\right)\mathrm{d}w\mathrm{d}\theta\\
&\leq&\frac{1}{2\pi}\int_{|\sin\theta|\geq r_n}\int_{-\infty}^\infty\frac{1}{|\sin\theta|}\cdot\mathbf{1}\left(|w|\leq\psi(\ell)\cdot r_n\cdot|\sin\theta|+r_{n+1}\right)\mathrm{d}w\mathrm{d}\theta\\
&=&\frac{1}{2\pi}\int_{|\sin\theta|\geq r_n}\frac{1}{|\sin\theta|}\cdot2\cdot(\psi(\ell)\cdot r_n\cdot|\sin\theta|+r_{n+1})\mathrm{d}\theta\\
&\leq&\frac{1}{2\pi}\int_{|\sin\theta|\geq r_n}\frac{1}{|\sin\theta|}\cdot2\cdot(\psi(\ell)\cdot r_n+\psi(\ell)\cdot r_n)\mathrm{d}\theta=\frac{2}{\pi}\cdot\psi(\ell)\cdot r_n\cdot\int_{|\sin\theta|\geq r_n}\frac{\mathrm{d}\theta}{|\sin\theta|}.
\end{eqnarray*}
Equation \eqref{eq:psi} readily follows, which completes the proof.
\end{proof}

Now armed with Lemma \ref{lem:anormalroutes}, we can come to the proof of Lemma \ref{lem:offroad}.

\begin{proof}[Proof of Lemma \ref{lem:offroad}]
Without loss of generality, we assume that $\ell_0=\mathbb{R}\times\{0\}$ and $v_0=1$.
Set ${r_n=5^{-n}}$ for all $n\in\mathbb{N}$.
Moreover, fix a parameter $\zeta\in{]0,1[}$ to be adjusted, and set ${v_n=\zeta\cdot r_n^{1/(\beta-1)}}$ for all $n\in\mathbb{N}$.
Finally, fix parameters $\alpha_1<\alpha_2<\alpha_3\in{]0,1[}$ to be adjusted, and consider the following claim.

\smallskip

\emph{\underline{Claim.}} There exists $\varepsilon>0$ such that for all sufficiently large $n$, on the complement of a bad event $B_n$ that has probability at most $r_n^{1+\varepsilon}$, the following holds: if there exists a geodesic $V'$-path ${\gamma:[0,\tau]\rightarrow\mathbb{R}^2}$ from a point outside $\overline{B}\left(0,r_n^{\alpha_1}\right)$ to a point on $\ell_0\cap\overline{B}(0,r_n)$ such that 
\[\gamma[0,\tau]\cap\ell_0\subset\overline{B}(0,r_n),\]
then there must exist a road of $\Pi$ with speed at least $v_n^{\alpha_3}$ that passes through $\overline{B}\left(0,2r_n\right)$.

\smallskip

First, let us explain how this claim yields the result of the lemma.
Fix $\delta>0$, and let $\mathcal{X}$ be the set of points $x\in\ell_0$ for which there exists a geodesic $V'$-path $\gamma:[0,\tau]\rightarrow\mathbb{R}^2$ from $x$ to a point outside $\overline{B}(x,\delta)$ such that $\gamma(t)\notin\ell_0$ for all $t\in{]0,\tau]}$.
Let us show that for each $R>0$, almost surely, we have 
\begin{equation}\label{eq:goallemoffroad}
\mathcal{X}\cap\overline{B}(0,R)\subset\bigcup_{(\ell,v)\in\Pi}\ell_0\cap\ell.
\end{equation}
For every $n\in\mathbb{N}$, let $\left(B(z,r_n)\,;\,z\in Z_n\right)$ be a covering of $\ell_0\cap\overline{B}(0,R)$ by balls of radius $r_n$, with centres $z\in\ell_0\cap\overline{B}(0,R)$ at least $r_n$ apart from each other.
By a measure argument, we have $\#Z_n\leq C\cdot r_n^{-1}$ for some constant $C$.
Now, for all $n$ large enough so that $2r_n^{\alpha_1}\leq\delta$, by the triangle inequality, we have
\[\mathcal{X}\cap\overline{B}(0,R)\subset\bigcup_{z\in\mathcal{Z}_n}\ell_0\cap B(z,r_n),\]
where $\mathcal{Z}_n$ denotes the set of points $z\in Z_n$ for which there exists a geodesic $V'$-path ${\gamma:[0,\tau]\rightarrow\mathbb{R}^2}$ from a point outside $\overline{B}\left(z,r_n^{\alpha_1}\right)$ to a point on ${\ell_0\cap\overline{B}(z,r_n)}$ such that 
\[\gamma[0,\tau]\cap\ell_0\subset\overline{B}(z,r_n).\]
Then, for every $z\in Z_n$, let $B_n^z$ be the bad event $B_n$, but for the pushforward $\Pi\circ f_{-z,1}^{-1}$ in place of $\Pi$, where $f_{x,r}$ is the map of \eqref{eq:defscaling}.
By \eqref{eq:selfsimilarPi}, we have $\mathbb{P}\left(B_n^z\right)=\mathbb{P}(B_n)$.
Now, by the union bound, we have
\[\mathbb{P}\left(\exists z\in Z_n:\text{$B_n^z$ is realised}\right)\leq\sum_{z\in Z_n}\mathbb{P}(B_n^z)=\#Z_n\cdot\mathbb{P}(B_n)\leq C\cdot r_n^{-1}\cdot r_n^{1+\varepsilon}=C\cdot r_n^\varepsilon.\]
By the Borel--Cantelli lemma, we deduce that almost surely, for all sufficiently large $n$, there is no point $z\in Z_n$ for which the bad event $B_n^z$ is realised.
Thus, for all sufficiently large $n$, we have
\[\mathcal{Z}_n\subset\left\{z\in Z_n:V^z_{2r_n}\geq v_n^{\alpha_3}\right\}.\]
It follows that
\[\mathcal{X}\cap\overline{B}(0,R)\subset\left\{x\in\ell_0:\text{$V^x_{3r_n}\geq v_n^{\alpha_3}$ for all sufficiently large $n$}\right\},\]
which yields \eqref{eq:goallemoffroad} by Lemma \ref{lem:anormalroutes}.

\smallskip

Now, let us dive into the proof of the claim.
We introduce the ``vertical speed'' function ${\phi:\mathbb{L}\times\mathbb{R}_+^*\rightarrow\mathbb{R}_+}$, defined by $\phi(\ell,v)=\left|\left\langle\ell,\ell_0^\perp\right\rangle\right|\cdot v$ for all ${(\ell,v)\in\mathbb{L}\times\mathbb{R}_+^*}$.
Note that we have $\phi(\ell,v)\leq v$ for all $(\ell,v)\in\mathbb{L}\times\mathbb{R}_+^*$.
In particular, for every $\phi_0>0$, there are finitely many roads $(\ell,v)\in\Pi$ with $\phi(\ell,v)\geq\phi_0$ that pass through $\overline{B}(0,r_n^{\alpha_1})$.
Then, for each $k\in\mathbb{N}$, we define $R_k^1,R_k^2,R_k^3,R_k^4$ as the four rectangles that are represented in Figure \ref{fig:circuit}, and we denote by $\Pi_k$ the restriction of $\Pi$ to the set of roads $(\ell,v)$ with speed $v\geq v_k$ for which there exists $i\in\llbracket1,4\rrbracket$ such that $\ell$ crosses $R_k^i$ in the length direction (i.e, passes through both shorter sides of the rectangle).
Note that there exists a constant $c$ such that
\begin{equation}\label{eq:cstc}
\phi(\ell,v)\geq c\cdot v_k\quad\text{for all $(\ell,v)\in\Pi_k$.}
\end{equation}
Finally, let $A_k$ be the good event: ``for each $i\in\llbracket1,4\rrbracket$, there is a road of $\Pi$ with speed at least $v_k$ that crosses $R_k^i$ in the length direction''.
We call $k$ a ``good scale'' if the event $A_k$ is realised.

\begin{figure}[!ht]
\centering
\includegraphics[width=0.7\linewidth]{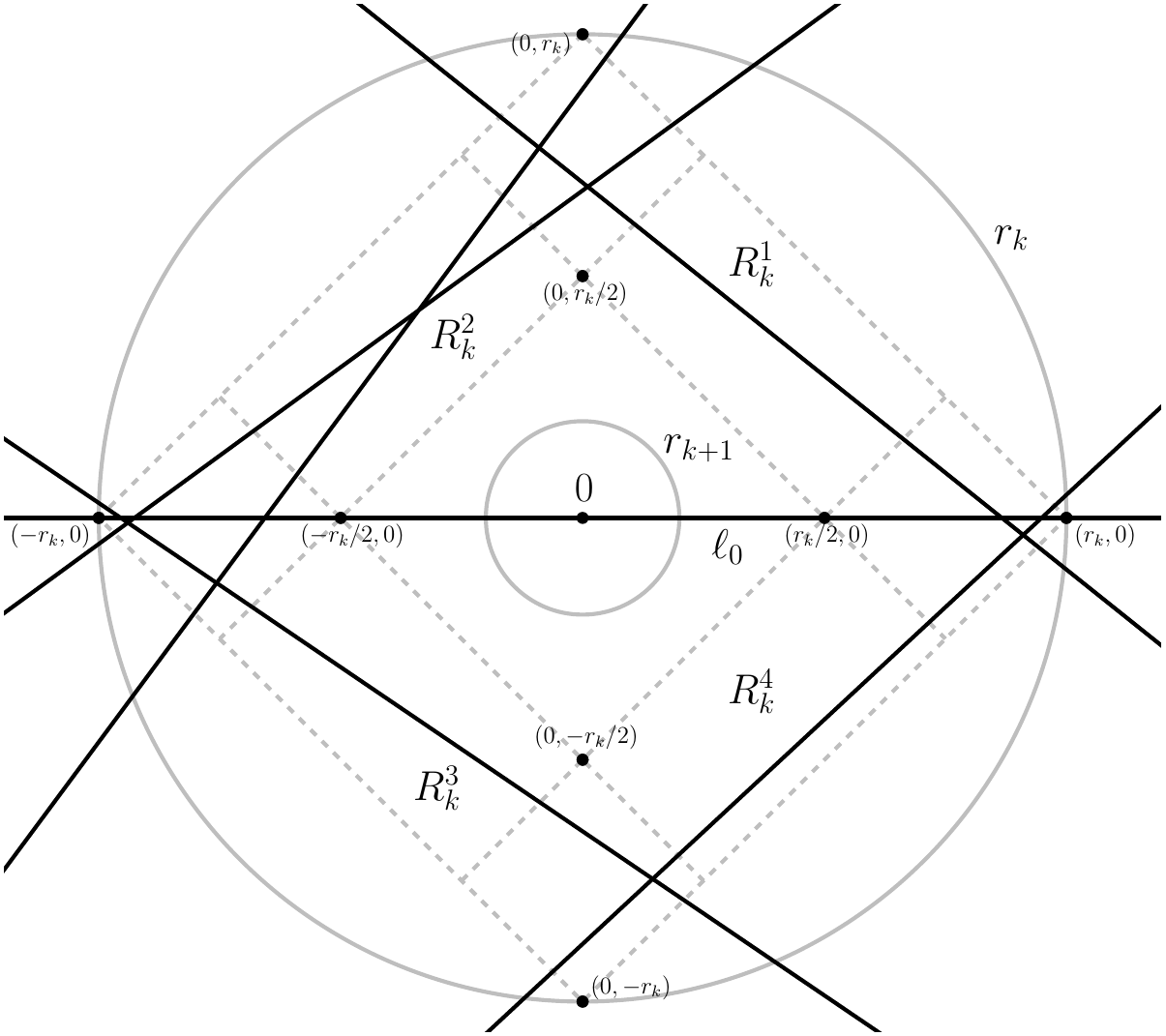}
\caption{Definition of the rectangles $R_k^1,R_k^2,R_k^3,R_k^4$, and illustration of the good event $A_k$. In this example, the event $A_k$ is realised, hence $k$ is a good scale. Note that there exists a constant $c$ such that for every line $\ell\in\mathbb{L}$ that crosses $R_k^i$ in the length direction for some $i\in\llbracket1,4\rrbracket$, we have $\left|\left\langle\ell,\ell_0^\perp\right\rangle\right|\geq c$. Moreover, note that every line $\ell\in\mathbb{L}$ that crosses $R_k^i$ in the length direction for some $i\in\llbracket1,4\rrbracket$ passes through $\overline{B}(0,r_k)$ but does not pass through $\overline{B}(0,r_{k+1})$.}\label{fig:circuit}
\end{figure}

At this point, let us deal with the technical matter of ajudsting our parameters.
First, we choose $\alpha_1<\alpha_2<\alpha_3$ close enough to $1$ so that
\begin{equation}\label{eq:tech1}
\alpha_1-\frac{\alpha_3}{\beta-1}\cdot(\beta-2)+(\alpha_1-1)>0
\end{equation}
and
\begin{equation}\label{eq:tech2}
\alpha_1>1-\frac{\alpha_3}{\beta-1}.
\end{equation}
Next, we set $\phi_n^0=c\cdot v_n^{\alpha_2}$, where $c$ is the constant of \eqref{eq:cstc}.
Then, using \eqref{eq:tech1}, we fix $j\in\mathbb{N}^*$ large enough so that
\begin{equation}\label{eq:expsmallgap}
\left(\alpha_1-\frac{\alpha_3}{\beta-1}\cdot(\beta-2)+(\alpha_1-1)\right)\cdot j>1.
\end{equation}
Now that $j$ has been fixed, note that by \eqref{eq:tech2}, the following holds for all sufficiently large $n$:
\begin{equation}\label{eq:remember}
\frac{r_n}{\phi_n^0}+j\cdot2r_n^{\alpha_1}\leq\frac{r_n}{v_n^{\alpha_3}}.
\end{equation}
Finally, we adjust $\zeta$ so that the following bad event, say $D_n$, has probability at most $r_n^2$ for all sufficiently large $n$: ``there is no good scale between $\alpha_1n$ and $\alpha_2n$''.
To this end, note that for each $k\in\mathbb{N}$, the event $A_k$ is measurable with respect to the restriction of $\Pi$ to the set of roads that pass through $\overline{B}(0,r_k)$ but do not pass through $\overline{B}(0,r_{k+1})$, by the definition of the $R_k^i$ (see Figure \ref{fig:circuit}).
Moreover, by the definition of $v_k$, we have ${\mathbb{P}(A_k)=\mathbb{P}(A_0)=:p(\zeta)}$.
It follows that 
\[\mathbb{P}(D_n)=p(\zeta)^{\lfloor\alpha_2n\rfloor-\lceil\alpha_1n\rceil+1}.\]
Since ${\lim_{\zeta\downarrow0}p(\zeta)=1}$, we can fix $\zeta$ small enough so that $p(\zeta)^{\alpha_2-\alpha_1}<1/10$.
We deduce that for all sufficiently large $n$, we have
\[\mathbb{P}(D_n)\leq10^{-n}=r_n^2.\]
From now on, we assume that $n$ is large enough so that \eqref{eq:remember} holds.
With this setup, consider the following implication.

\smallskip

\emph{\underline{Implication.}} On the complement of the bad event $D_n$, if there exists a geodesic $V'$-path ${\gamma:[0,\tau]\rightarrow\mathbb{R}^2}$ from a point outside $\overline{B}\left(0,r_n^{\alpha_1}\right)$ to a point on $\ell_0\cap\overline{B}(0,r_n)$ such that 
\[\gamma[0,\tau]\cap\ell_0\subset \overline{B}(0,r_n),\]
and if there is no road of $\Pi$ with speed at least $v_n^{\alpha_3}$ that passes through $\overline{B}(0,2r_n)$, then the following bad event, say $E_n$, must be realised: ``there exists $j$ roads $(\ell,v)\in\Pi$ that pass through $\overline{B}(0,r_n^{\alpha_1})$ such that $1/\phi(\ell,v)\in\left[\left.1\middle/\phi_n^0\right.,\left.1\middle/\phi_n^0\right.+j\cdot\left.2r_n^{\alpha_1}\middle/r_n\right.\right]$''.

\smallskip

Provided that this implication holds, suppose furthermore that there exists $\varepsilon>0$ such that
\begin{equation}\label{eq:smallgapevent}
\mathbb{P}(E_n)\leq r_n^{1+\varepsilon}\quad\text{for all sufficiently large $n$.}
\end{equation}
Then, the claim we are after readily follows, with $B_n=D_n\cup E_n$.
Therefore, to complete the proof of the lemma, it suffices to prove the implication, and check that \eqref{eq:smallgapevent} holds.

We start with the implication.
Suppose that there exists a good scale $k$ between $\alpha_1n$ and $\alpha_2n$, and that there exists a geodesic $V'$-path ${\gamma:[0,\tau]\rightarrow\mathbb{R}^2}$ from a point outside $\overline{B}\left(0,r_n^{\alpha_1}\right)$ to a point on $\ell_0\cap\overline{B}(0,r_n)$ such that ${\gamma[0,\tau]\cap\ell_0\subset\overline{B}(0,r_n)}$.
Moreover, assume that there is no road of $\Pi$ with speed at least $v_n^{\alpha_3}$ that passes through $\overline{B}(0,2r_n)$, and let us prove that then, the bad event $E_n$ must be realised.
Let $\phi_n^1>\ldots>\phi_n^j\in\left]0,\phi_n^0\right[$ be defined by
\[\frac{1}{\phi_n^i}=\frac{1}{\phi_n^0}+i\cdot\frac{2r_n^{\alpha_1}}{r_n}\quad\text{for all $i\in\llbracket1,j\rrbracket$.}\]
We will prove that for each $i\in\llbracket1,j\rrbracket$, there must exist a road $(\ell,v)\in\Pi$ that passes through $\overline{B}(0,r_n^{\alpha_1})$ such that
\[\frac{1}{\phi(\ell,v)}\in\left]\frac{1}{\phi_n^{i-1}},\frac{1}{\phi_n^i}\right].\]
Fix $i\in\llbracket1,j\rrbracket$.
We denote by $\mathcal{R}$ be the set of roads $(\ell,v)\in\Pi$ that pass through $\overline{B}(0,r_n^{\alpha_1})$ and have vertical speed $\phi(\ell,v)\geq\phi_n^{i-1}$.
Note that $\mathcal{R}$ is finite, and contains the roads of $\Pi_k$.
Indeed, since $k$ is between $\alpha_1n$ and $\alpha_2n$, each road $(\ell,v)\in\Pi_k$ passes through $\overline{B}(0,r_n^{\alpha_1})$, and has vertical speed $\phi(\ell,v)\geq c\cdot v_n^{\alpha_2}=\phi_n^0\geq\phi_n^{i-1}$, by \eqref{eq:cstc}.
Since $k$ is a good scale, it follows that the roads of $\mathcal{R}$ induce a convex polygon $P\subset\overline{B}(0,r_n^{\alpha_1})$, which formally is defined as the intersection of the half-spaces $\left\{x-\lambda\cdot\mathrm{proj}_\ell(0)\,;\,\text{$x\in\ell$ and $\lambda\geq0$}\right\}$, over the roads $(\ell,v)\in\mathcal{R}$.
This polygon $P$ has the following key property: for any point $x$ on its boundary $\partial P$, and for any point $y\in\ell_0\cap\overline{B}(0,r_n)$, one can use the roads of $\mathcal{R}$ to drive from $x$ to a point $z$ on $\ell_0\cap\overline{B}(0,r_n^{\alpha_1})$, in time at most $\left.d(x,\ell_0)\middle/\phi_n^{i-1}\right.$, and then use the road $(\ell_0,1)$ to drive from $z$ to $y$, in time at most $2r_n^{\alpha_1}$ by the triangle inequality.
See Figure \ref{fig:circuit2} for an illustration.

\begin{figure}[!ht]
\centering
\includegraphics[width=0.7\linewidth]{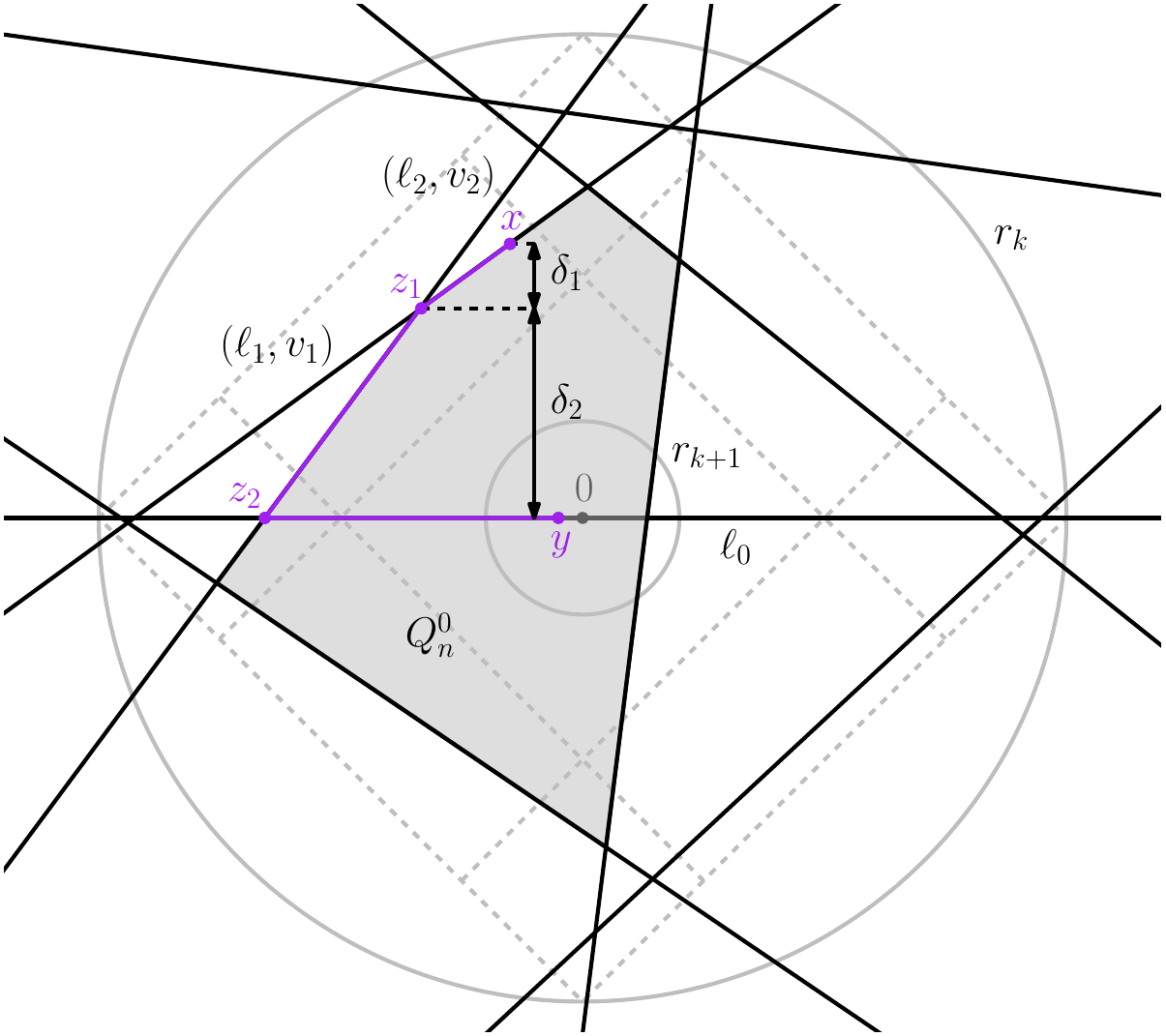}
\caption{We reproduce Figure \ref{fig:circuit}, assuming furthermore that $k$ is between $\alpha_1n$ and $\alpha_2n$. Then, we consider all the roads $(\ell,v)\in\Pi$ with $\phi(\ell,v)\geq\phi_n^{i-1}$ that pass through $\overline{B}(0,r_n^{\alpha_1})$, that is, the roads of $\mathcal{R}$. In this example, this adds exactly two more roads to the picture. The convex polygon $P$ containing $0$ induced by the roads of $\mathcal{R}$ is represented here in gray. Let us prove its key property on the example illustrated here. With the notation introduced on the picture, one can use $(\ell_1,v_1)$ to drive from $x$ to $z_1$, in time $\delta_1/\phi(\ell_1,v_1)\leq\left.\delta_1\middle/\phi_n^{i-1}\right.$, then use $(\ell_2,v_2)$ to drive from $z_1$ to $z_2$, in time ${\delta_2/\phi(\ell_2,v_2)\leq\left.\delta_2\middle/\phi_n^{i-1}\right.}$, and finally use $(\ell_0,1)$ to drive from $z_2$ to $y$, in time ${|z_2-y|\leq2r_n^{\alpha_1}}$. 
The total driving time of this $V'$-path is at most ${\left.(\delta_1+\delta_2)\middle/\phi_n^0\right.+2r_n^{\alpha_1}}=d(x,\ell_0)+2r_n^{\alpha_1}$.}\label{fig:circuit2}
\end{figure}

By \eqref{eq:remember}, we have 
\[\frac{r_n}{\phi_n^{i-1}}=\frac{r_n}{\phi_n^0}+(i-1)\cdot2r_n^{\alpha_1}\leq\frac{r_n}{v_n^{\alpha_3}},\] 
hence $\phi_n^{i-1}\geq v_n^{\alpha_3}$.
In particular, since $\partial P$ consists of roads of $\mathcal{R}$, that have speed at least $\phi_n^{i-1}$, the circuit $\partial P$ must lie outside $\overline{B}(0,2r_n)$.
Let us denote by $t\in{]0,\tau[}$ the last time at which $\gamma$ crosses $\partial P$.
As $\gamma$ is geodesic, the key property of $P$ (see Figure \ref{fig:circuit2}) implies that
\begin{equation}\label{eq:vertical1}
\tau-t\leq\frac{d(\gamma(t),\ell_0)}{\phi_n^{i-1}}+2r_n^{\alpha_1}.
\end{equation}
On the other hand, observe that over the time interval $[t,\tau]$, the $V'$-path $\gamma$ must travel a distance at least $r_n$ by using the roads of $\Pi$ within the ball $\overline{B}(0,2r_n)$.
This yields
\begin{equation}\label{eq:obvious}
\tau-t\geq\frac{r_n}{V^0_{2r_n}}>\frac{r_n}{v_n^{\alpha_3}}.
\end{equation}
By \eqref{eq:remember}, we have 
\[\frac{r_n}{\phi_n^{i-1}}+2r_n^{\alpha_1}=\frac{r_n}{\phi_n^0}+i\cdot2r_n^{\alpha_1}\leq\frac{r_n}{v_n^{\alpha_3}}.\]
In particular, combining this with \eqref{eq:vertical1} and \eqref{eq:obvious}, we must have ${d(\gamma(t),\ell_0)>r_n}$.
Then, we complement \eqref{eq:vertical1} with the bound
\begin{equation}\label{eq:vertical2}
\tau-t\geq\frac{d(\gamma(t),\ell_0)}{\phi(\ell,v)},
\end{equation}
where $(\ell,v)$ denotes the road of $\Pi$ with maximal vertical speed used by $\gamma$ over the time interval $[t,\tau]$.
By the definition of $t$, the road $(\ell,v)$ passes through $\overline{B}(0,r_n^{\alpha_1})$, and has vertical speed $\phi(\ell,v)<\phi_n^{i-1}$.
Moreover, combining \eqref{eq:vertical1} and \eqref{eq:vertical2}, we obtain that
\[\frac{d(\gamma(t),\ell_0)}{\phi(\ell,v)}\leq\frac{d(\gamma(t),\ell_0)}{\phi_n^{i-1}}+2r_n^{\alpha_1},\]
hence
\[\frac{1}{\phi(\ell,v)}\leq\frac{1}{\phi_n^{i-1}}+\frac{2r_n^{\alpha_1}}{d(\gamma(t),\ell_0)}\leq\frac{1}{\phi_n^{i-1}}+\frac{2r_n^{\alpha_1}}{r_n}=\frac{1}{\phi_n^i}.\]
This concludes the proof of the implication.

Finally, to complete the proof of the lemma, it remains to check \eqref{eq:smallgapevent}.
By definition, we have
\[E_n=\left(\Psi_n\left[\frac{1}{\phi_n^0},\frac{1}{\phi_n^0}+j\cdot\frac{2r_n^{\alpha_1}}{r_n}\right]\geq j\right),\]
where $\Psi_n$ denotes the pushforward by $1/\phi$ of the restriction of $\Pi$ to the set of roads that pass through $\overline{B}\left(0,r_n^{\alpha_1}\right)$.
By the mapping theorem, the random measure $\Psi_n$ is a Poisson process with intensity measure $c'\cdot r_n^{\alpha_1}\cdot w^{\beta-2}\mathrm{d}w$ on $\mathbb{R}_+^*$, for some constant $c'$.
Indeed, for every Borel function ${\varphi:\mathbb{R}_+^*\rightarrow\mathbb{R}_+}$, we have
\begin{eqnarray*}
\lefteqn{\int_{\left\langle\overline{B}\left(0,r_n^{\alpha_1}\right)\right\rangle}\int_{\mathbb{R}_+^*}\varphi\left(\frac{1}{\phi(\ell,v)}\right)\cdot v^{-\beta}\mathrm{d}v\mathrm{d}\mu(\ell)}\\
&=&\int_{\left\langle\overline{B}\left(0,r_n^{\alpha_1}\right)\right\rangle}\int_0^\infty\varphi\left(\frac{1}{\left|\left\langle\ell,\ell_0^\perp\right\rangle\right|\cdot v}\right)\cdot v^{-\beta}\mathrm{d}v\mathrm{d}\mu(\ell)\\
&=&\int_{\left\langle\overline{B}\left(0,r_n^{\alpha_1}\right)\right\rangle}\int_0^\infty\varphi(w)\cdot\left(\frac{1}{\left|\left\langle\ell,\ell_0^\perp\right\rangle\right|\cdot w}\right)^{-\beta}\cdot\frac{\mathrm{d}w}{\left|\left\langle\ell,\ell_0^\perp\right\rangle\right|\cdot w^2}\mathrm{d}\mu(\ell)\\
&=&\int_{\left\langle\overline{B}\left(0,r_n^{\alpha_1}\right)\right\rangle}\left|\left\langle\ell,\ell_0^\perp\right\rangle\right|^{\beta-1}\mathrm{d}\mu(\ell)\cdot\int_0^\infty\varphi(w)\cdot w^{\beta-2}\mathrm{d}w,
\end{eqnarray*}
with, by the self-similarity of $\mu$ under scaling:
\[\int_{\left\langle\overline{B}\left(0,r_n^{\alpha_1}\right)\right\rangle}\left|\left\langle\ell,\ell_0^\perp\right\rangle\right|^{\beta-1}\mathrm{d}\mu(\ell)=\int_{\left\langle\overline{B}(0,1)\right\rangle}\left|\left\langle\ell,\ell_0^\perp\right\rangle\right|^{\beta-1}\mathrm{d}\mu(\ell)\cdot r_n^{\alpha_1}.\]
Using the inequality $\mathbb{P}(\mathrm{Poisson}(\lambda)\geq j)\leq\lambda^j$, it follows that
\[\mathbb{P}(E_n)\leq\left(c'\cdot r_n^{\alpha_1}\cdot\int_{\left.1\middle/\phi_n^0\right.}^{\left.1\middle/\phi_n^0\right.+j\cdot\left.2r_n^{\alpha_1}\middle/r_n\right.}w^{\beta-2}\mathrm{d}w\right)^j.\]
Finally, since
\[\int_{\left.1\middle/\phi_n^0\right.}^{\left.1\middle/\phi_n^0\right.+j\cdot\left.2r_n^{\alpha_1}\middle/r_n\right.}w^{\beta-2}\mathrm{d}w\leq\left(\frac{1}{\phi_n^0}+j\cdot\frac{2r_n^{\alpha_1}}{r_n}\right)^{\beta-2}\cdot j\cdot\frac{2r_n^{\alpha_1}}{r_n}\leq\left(\frac{1}{v_n^{\alpha_3}}\right)^{\beta-2}\cdot j\cdot\frac{2r_n^{\alpha_1}}{r_n}\]
by \eqref{eq:remember}, we obtain
\[\mathbb{P}(E_n)\leq\left(c'\cdot r_n^{\alpha_1}\cdot\left(\frac{1}{v_n^{\alpha_3}}\right)^{\beta-2}\cdot j\cdot\frac{2r_n^{\alpha_1}}{r_n}\right)^j,\]
which yields \eqref{eq:smallgapevent}, by \eqref{eq:expsmallgap}.
The proof of the lemma is complete.
\end{proof}

\section{On the cut locus}

We conclude this paper with a short discussion about the cut locus. 
The notion of cut locus was introduced by Poincaré in the context of smooth manifolds.
In our setting of rough random metric, we follow the definition of Angel, Kolesnik and Miermont \cite{angelkolesnikmiermont}: for a fixed point $x\in\mathbb{R}^2$, the weak \textbf{cut locus} $ \mathcal{C}_x$ of $\left( \mathbb{R}^2,T\right)$ with respect to $x$ is the set of points $y\in\mathbb{R}^2$ from which there exists multiple geodesics to $x$. 
Actually, by Lemma \ref{lem:nobanana}, for each $x$, almost surely $\mathcal{C}_x$ agrees with the strong cut locus, that is the set of points $y$ from which there exists multiple geodesics to $x$ which are required to be disjoint in a neighborhood of $y$. 
In \cite{angelkolesnikmiermont}, the authors study the stability of $ \mathcal{C}_x$ when the base point $x$ is moved over $ \mathbb{R}^2$, and the two notions could differ for exceptional points $x$. 
We shall here focus on more basic properties, and will henceforth suppose that $x=0$ is fixed.
We use the term cut locus for the weak or strong (indifferently, since the two are equal almost surely) cut locus $\mathcal{C}_0$ with respect to $0$.
Let us make a couple of remarks on $ \mathcal{C}_0$:
\begin{itemize}
\item  Almost surely, we have $ \mathcal{I} \cap \mathcal{C}_0 =\emptyset$. 
Indeed, if there exist two disjoint geodesics $\gamma_{1}, \gamma_{2}$ from an intersection point $x \in \mathcal{I}$ to $0$, then $\gamma_1$ and $\gamma_2$ coalesce at a point $y$  by before hitting $0$, by item 1 Proposition \ref{prop:confluence}.
See Figure \ref{fig:cutlocus-inter}. 
Since geodesics do not pause en route (Theorem \ref{thm:pause}), the confluence point $y$ must also belong to $ \mathcal{I}$, and the geodesic segments of $\gamma_{1}$ and $\gamma_{2}$ between $x$ and $y$ are what we called simple $V$-paths, with the same driving time. 
But by the multivariate Mecke formula, almost surely, it is not possible to find two distinct simple $V$-paths between two points of $\mathcal{I}$ with the same driving time.

\begin{figure}[!h]
\begin{center}
\includegraphics[width=10cm]{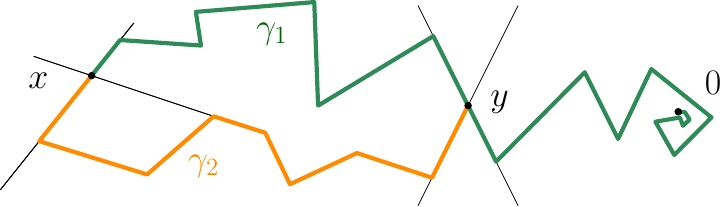}
\caption{Illustration of the proof of $\mathcal{I} \cap \mathcal{C}_0 =\emptyset$. \label{fig:cutlocus-inter}}
\end{center}
\end{figure}
 
\item Almost surely, for each road $(\ell_0,v_0) \in \Pi$, the set $ \ell_0\cap\mathcal{C}_0$ is countable. 
Indeed, let $x \in \ell$ be a point such that there exist two distinct geodesics $\gamma_1$ and $\gamma_2$ from $x$ to $0$.
First, by the previous point, we have $x\in\mathcal{L}\setminus\mathcal{I}$.
Then, by the results of the previous section, the point $x$ is not a $3^+$-star, and there must exist roads $(\ell_1,v_1)\neq(\ell_2,v_2)\in\Pi\setminus\{(\ell_0,v_0)\}$ such that $\gamma_i$ leaves $(\ell_0,v_0)$ by using the road $(\ell_i,v_i)$, for each $i$ (since $\gamma_1$ and $\gamma_2$ are disjoint locally around $x$, these two roads must be distinct).
In turn, the point $x \in \ell$ can be recovered from the two roads $(\ell_i,v_i)$: if $x_i$ denotes the intersection point of $\ell_i$ with $\ell_0$, then $x$ is the unique point $z$ in $\ell$ such that 
$$  \frac{|z-x_{1}|}{v_0} + T(x_{1},0) = \frac{|z-x_{2}|}{v_0} + T(x_{2},0).$$
Therefore, since there are countably many pairs of roads of $\Pi$, the set $\ell_0\cap\mathcal{C}_0$ must be countable.
\end{itemize}

Recall that in smooth contexts (complete, analytic Riemannian surface homeomorphic to the sphere), the cut locus is known to be a tree with finitely many branches \cite{myers1935connections}. 
In our setting of fractal random metric, we believe that the cut locus $\mathcal{C}_{0}$ is dense, path-connected and does not contain any non-trivial cycle. 
More precisely, the degree of  $ x \in \mathcal{C}_{0}$, i.e, the number of connected components of $ \mathcal{C}_{0} \setminus\{x\}$, should be equal to the maximal number of distinct geodesics from $x$ to $0$, and furthermore points with degree at least $3$ should form a countable and dense subset.
This last assertion, and the fact that the cut locus contains a non-trivial connected set, have been proved for the \textsc{Lqg} random metric by Gwynne \cite{gwynnenetworks}. {The kind of topological arguments he used should also work out in our setting (and in fact the above assertions should hold under relatively weak assumptions on the planar random metric), and to convince the reader of this fact we give the following example: the argument of Gwynne \cite[Lemma 3.2]{gwynnenetworks} shows that

\begin{prop}
Almost surely, the set of points $x\in\left.\mathbb{R}^2\middle\backslash\{0\}\right.$ such that there exists (at least) three distinct geodesics from $x$ to $0$ is countable.
\end{prop}
\begin{proof}
Fix $0<r<R$, and let $A=\left.\overline{B}(0,R)\middle\backslash B(0,r)\right.$ be the annulus between radii $r$ and $R$ around $0$.
Fix $\delta\in{]0,r[}$, and let us show that almost surely, the set $\mathcal{X}_\delta$ of points $x\in A$ for which there exists three geodesics $\gamma_i:[0,T(x,0)]\rightarrow\mathbb{R}^2$ from $x$ to $0$ such that 
\[\gamma_i{]0,T(x,0)]}\cap\gamma_j{]0,T(x,0)]}\cap\overline{B}(x,\delta)=\emptyset\quad\text{for all $i\neq j$}\]
is finite.
By Proposition \ref{prop:K+stars}, there exists $m\in\mathbb{N}^*$ such that almost surely, the following holds: for every $y\in A$, there exists $\eta_y\in{]0,\delta/2[}$ and a cut set ${\mathcal{Z}_y\subset\left.\overline{B}(y,\delta/2)\middle\backslash\overline{B}(y,\eta)\right.}$ with cardinality $\#\mathcal{Z}_y\leq m$, such that any geodesic $V$-path from a point inside $\overline{B}(y,\eta_y)$ to a point outside $\overline{B}(y,\delta/2)$ must pass through $z$ for some $z\in\mathcal{Z}_y$.
Then, since the open balls $(B(y,\eta_y)\,;\,y\in A)$ cover $A$, by compactness, there exists a finite subset $Y\subset A$ such that the open balls $(B(y,\eta_y)\,;\,y\in Y)$ cover $A$.
Now, we claim that almost surely, we have
\[\#\mathcal{X}_\delta\leq\#Y\cdot\binom{m}{3}<\infty.\]
This follows from the argument of Gwynne \cite[Lemma 3.2]{gwynnenetworks}, which shows that for each $y\in Y$, we have
\begin{equation}\label{eq:gwynne}
\#\mathcal{X}_\delta\cap\overline{B}(y,\eta_y)\leq\binom{m}{3}.
\end{equation}
Indeed, for each triple of distinct points $z_1,z_2,z_3\in\mathcal{Z}_y$, let $\mathcal{X}_\delta^{y;z_1,z_2,z_3}$ be the set of points ${x\in\overline{B}(y,\eta_y)}$ for which there exists three geodesics $\gamma_i:[0,T(x,0)]\rightarrow\mathbb{R}^2$ from $x$ to $0$ with
\[\gamma_i{]0,T(x,0)]}\cap\gamma_j{]0,T(x,0)]}\cap\overline{B}(x,\delta)=\emptyset\quad\text{for all $i\neq j$},\]
such that $\gamma_i$ passes through $z_i$ for each $i$.
First, since $\overline{B}(y,\delta/2)\subset\overline{B}(x,\delta)$ for all $x\in\overline{B}(y,\eta_y)$, we must have
\[\mathcal{X}_\delta\cap\overline{B}(y,\eta_y)\subset\bigcup_{\text{$z_1,z_2,z_3\in\mathcal{Z}_y$ distinct}}\mathcal{X}_\delta^{y;z_1,z_2,z_3}.\]
Then, for each triple of distinct points $z_1,z_2,z_3\in\mathcal{Z}_y$, the argument of Gwynne \cite[Lemma 3.2]{gwynnenetworks} shows that $\#\mathcal{X}_\delta^{y;z_1,z_2,z_3}\leq1$.
See Figure \ref{fig:gwynne} for a visual proof, and \cite[Lemma 3.2]{gwynnenetworks} for a detailed proof in the framework of \textsc{Lqg} random metrics.
Thus, we obtain \eqref{eq:gwynne}, which completes the proof.

\begin{figure}[ht]
\centering
\includegraphics[width=0.7\linewidth]{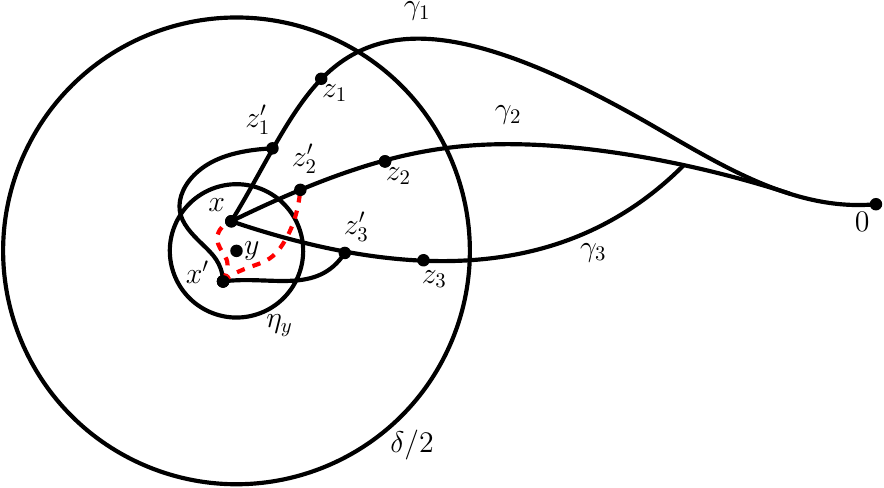}
\caption{Illustration of the argument of Gwynne \cite[Lemma 3.2]{gwynnenetworks}.
Fix $x\in\mathcal{X}_\delta^{y;z_1,z_2,z_3}$, and draw the corresponding three geodesics $\gamma_i$ from $x$ to $0$, where $\gamma_i$ passes through $z_i$ for each $i$.
By Lemma \ref{lem:nobanana}, their newtork must have a similar form to the one illustrated here.
Then, try to draw three geodesics $\gamma'_i$ from a point $x'\in\left.\overline{B}(y,\eta_y)\middle\backslash\{x\}\right.$ to $0$ such that $\gamma'_i$ passes through $z_i$ for each $i$.
Let $z_i'$ be the point at which $\gamma'_i$ meets $\gamma_i$, which is necessarily between $x$ and $z_i$, and observe that we cannove have $z_i'=x$, again by Lemma \ref{lem:nobanana}.
The existence of such a point $x'$ is simply not possible, and we conclude that $\#\mathcal{X}_\delta^{y;z_1,z_2,z_3}\leq1$.}\label{fig:gwynne}
\end{figure}
\end{proof}}

Back to our discussion, notice that any point of $\mathcal{C}_0$ with degree at least $3$ is in particular a $3^+$-star, and thus cannot belong to $ \mathcal{L}\setminus\mathcal{I}$ by Theorem \ref{thm:frame}.
Also, such a point cannot belong to $\mathcal{I}$ since $\mathcal{I}\cap\mathcal{C}_0=\emptyset$, and so the existence of such points furnishes examples of a $3^+$-star which is not  {on the geodesic frame (a canonical example of a $3^+$-star is given by a confluence point of two distinct geodesics to a typical point, hence lies on the geodesic frame)}.
Furthermore, provided that $\mathcal{C}_{0}$ is path-connected and non-trivial, it must be uncountable, and so most of $ \mathcal{C}_{0}$ lies outside $\mathcal{L}$, since we saw that almost surely $\mathcal{L}\cap\mathcal{C}_{0}$ is countable. 
We also believe that the ``branches'' of $ \mathcal{C}_{0}$ are of fractal nature: {almost surely, for every $x \neq y \in \mathcal{C}_{0}$, and for any path $\gamma:[0,1]\rightarrow\mathcal{C}_0$ from $x$ to $y$, the Hausdorff dimension of $\gamma[0,1]$ should be strictly greater than $1$ with respect to the Euclidean metric.}
See Open problem \ref{op:} at the end of the Introduction.

\section{Appendix}\label{sec:appendix}

In this section, we present the folklore results that we have deferred to the appendix.
The first one is the following mixing property of $\Pi$.

\begin{lem}\label{lem:mixingzoom}
Let $\mathbb{M}$ be the space of counting measures on $\mathbb{L}\times\mathbb{R}_+^*$, endowed as usual with the $\sigma$-algebra generated by the maps $\left(\pi\in\mathbb{M}\mapsto\pi(B),\,\text{$B$ Borel subset of $\mathbb{L}\times\mathbb{R}_+^*$}\right)$.
For any bounded measurable functions $\varphi$ and $\psi$ from $\mathbb{M}$ to $\mathbb{R}$, we have
\[\mathbb{E}\left[\varphi(\Pi)\cdot\psi\left(f_k^*\Pi\right)\right]\longrightarrow\mathbb{E}[\varphi(\Pi)]\cdot\mathbb{E}[\psi(\Pi)]\quad\text{as $k\to\infty$.}\]
\end{lem}

\begin{proof}[Proof of Lemma \ref{lem:mixingzoom}]
For each ${n\in\mathbb{N}^*}$, let $\rho_n:\mathbb{M}\rightarrow\mathbb{M}$ be the restriction map to the set of roads with speed at least $1/n$ that pass through $\overline{B}(0,n)$, and let ${\mathcal{A}_n=\sigma(\rho_n(\Pi))}$ be the $\sigma$-algebra generated by the restriction $\rho_n(\Pi)$ of $\Pi$.
Fix bounded measurable functions $\varphi,\psi:\mathbb{M}\rightarrow\mathbb{R}$, and fix $\varepsilon>0$.
First, we use the following classical measure-theoretic approximation argument \cite[Lemma 3.16]{kallenberg}: since the $\sigma$-algebra generated by $\Pi$ agrees with the $\sigma$-algebra generated by $\bigcup_{n\geq1}\mathcal{A}_n$, there exists $n\in\mathbb{N}^*$ and bounded measurable functions $\varphi_n,\psi_n:\mathbb{M}\rightarrow\mathbb{R}$ such that
\begin{equation}\label{eqproofmixinggeo1}
\mathbb{E}[|\varphi(\Pi)-\varphi_n(\rho_n(\Pi))|]\leq\varepsilon\quad\text{and}\quad\mathbb{E}[|\psi(\Pi)-\psi_n(\rho_n(\Pi))|]\leq\varepsilon.
\end{equation}
Now, let $C$ be a constant that dominates $\varphi$, $\psi$, $\varphi_n$ and $\psi_n$.
On the one hand, for every $k\in\mathbb{N}$, we have
\begin{equation}\label{eqproofmixinggeo2}
|\mathbb{E}[\varphi(\Pi)\cdot\psi(f_k^*\Pi)]-\mathbb{E}[\varphi_n(\rho_n(\Pi))\cdot\psi_n(\rho_n(f_k^*\Pi))]|\leq2C\varepsilon.
\end{equation}
On the other hand, the random measures $\rho_n(\Pi)$ and $\rho_n\left(f_k^*\Pi\right)$ are ``asymptotically independent'', as $k\to\infty$.
Indeed, let $G_k$ be the event: ``no road of $\Pi$ with speed at least $1/n$ passes through $\overline{B}(0,n\cdot r_k)$''.
The following holds:
\begin{itemize}
\item we have $\mathbb{P}(G_k)\rightarrow1$ as $k\to\infty$,
\item on $G_k$, the random measure $\rho_n(\Pi)$ agrees with the restriction $\Pi_k^1$ of $\Pi$ to the set of roads with speed at least $1/n$ that pass through $\overline{B}(0,n)$ but do not pass through $\overline{B}(0,n\cdot r_k)$, and $\rho_n\left(f_k^*\Pi\right)$ agrees with the restriction $\Pi_k^2$ of $\Pi$ to the set of roads with speed at least $1/n\cdot r_k^{(d-1)/(\beta-1)}$ that pass through $\overline{B}(0,n\cdot r_k)$,
\item the random measures $\Pi_k^1$ and $\Pi_k^2$ are independent.
\end{itemize}
With these three ingredients, standard manipulations show that
\[\mathbb{E}[\varphi_n(\rho_n(\Pi))\cdot\psi_n(\rho_n(f_k^*\Pi))]\underset{k\to\infty}{\longrightarrow}\mathbb{E}[\varphi_n(\rho_n(\Pi))]\cdot\mathbb{E}[\psi_n(\rho_n(\Pi))].\]
Then, we conclude with \eqref{eqproofmixinggeo1} and \eqref{eqproofmixinggeo2}: as $k\to\infty$, we have
\[\mathbb{E}[\varphi(\Pi)\cdot\psi(f_k^*\Pi)]\longrightarrow\mathbb{E}[\varphi(\Pi)]\cdot\mathbb{E}[\psi(\Pi)].\]
\end{proof}

The second one is the proof of Lemma \ref{lem:confluencetool1}, which is essentially a mix between \cite[Theorem 3.1]{kahn} of the third author and \cite[Proposition 1.3]{moa1fractal} of the first author.

\begin{proof}[Proof of Lemma \ref{lem:confluencetool1}]
Beofre diving into the proof, let us recall the following estimate on the invariant measure $\mu$ on the space of lines.
For two (or more) compact subsets $K_1,K_2\subset\mathbb{R}^d$, we use $\langle K_1\,;\,K_2 \rangle$ as shorthand for $\langle K_1 \rangle\cap\langle K_2\rangle$.
In complement to \eqref{eq:normmugeo}, we have the following estimate: there exists constants $c$ and $C$ such that for every $x\neq y\in\mathbb{R}^d$, and for every $0<r,s\leq|x-y|$,
\begin{equation}\label{eq:twopointmugeo}
c\cdot\frac{r^{d-1}\cdot s^{d-1}}{|x-y|^{d-1}}\leq\mu\left\langle\overline{B}(x,r)\,;\,\overline{B}(y,s)\right\rangle\leq C\cdot\frac{r^{d-1}\cdot s^{d-1}}{|x-y|^{d-1}}.
\end{equation}

Now, let us dive into the proof.
Fix a parameter $\rho\in{]0,r[}$ to be adjusted, let $\Pi_\rho$ be the restriction of $\Pi$ to the set of roads with speed less than $1$ that do not pass through $\overline{B}(z,\rho)$, and let $V_\rho$ be ``speed limits'' function induced by $\Pi_\rho$.
For $x,y\in\mathbb{R}^d$ and $s>0$, we denote by $V_\rho(x,y,s)$ the speed of the fastest road of $\Pi_\rho$ that passes through both balls $\overline{B}(x,s)$ and $\overline{B}(y,s)$.
Let us recall Kendall's construction of $V$-paths from \cite[Theorem 3.1]{kendall}, and use it to construct a $V_\rho$-path between every pair of points $(x,y)$ with $x,y\in A$.
Fix a parameter $\alpha\in{]0,1[}$ to be adjusted, and fix $x,y\in A$.
To construct a $V_\rho$-path from $x$ to $y$, we proceed as follows.
Assuming that $x\neq y$ (for otherwise there is nothing to do), consider the fastest road $(\ell,v)$ of $\Pi_\rho$ that passes through both balls $\overline{B}(x,\alpha|x-y|)$ and ${\overline{B}(y,\alpha|x-y|)}$\footnote{In particular, we will adjust $\alpha$ and $\rho$ to ensure that such a road exists.}, and let $x'$ (resp. $y'$) be the orthogonal projection of $x$ (resp. $y$) on $\ell$.
The road $(\ell,v)$ allows to drive from $x'$ to $y'$ in time $|x'-y'|/v$, and it remains to find routes from $x$ to $x'$ and from $y'$ to $y$.
For this purpose, we recursively apply the procedure to the pairs of points $(x,x')$ and $(y',y)$.
To formalise this, let us encode the construction in the nodes $\varnothing,1,2,\ldots$ of the infinite binary tree $\mathbb{B}$ as follows.
Initially, we set $(x_\varnothing,y_\varnothing)=(x,y)$.
Then, by induction, for every ${n\in\mathbb{N}}$ such that $((x_u,y_u)\,;\,u\in\mathbb{B}:|u|\leq n)$ have been constructed, we proceed as follows.
For each node $u\in\mathbb{B}$ at height $n$, there is the following alternative.
\begin{itemize}
\item If $x_u=y_u$, then ``there is nothing to do''.
We let $x_u'=x_u$ and $y_u'=y_u$, and we let $T_u=0$ be the driving time of the trivial $V$-path from $x'_u$ to $y'_u$.
Then, we define the pairs of points associated with the children $u1$ and $u2$ of $u$ in $\mathbb{B}$ as $(x_{u1},y_{u1})=(x_u,x_u')$ and $(x_{u2},y_{u2})=(y_u',y_u)$.
\item Otherwise, consider the fastest road $(\ell_u,v_u)$ of $\Pi_\rho$ that passes through both balls 
\[\overline{B}(x_u,\alpha|x_u-y_u|)\quad\text{and}\quad\overline{B}(y_u,\alpha|x_u-y_u|).\]
Note that by definition, we have $v_u=V_\rho(x_u,y_u,\alpha|x_u-y_u|)$.
We let $x_u'$ (resp. $y_u'$) be the orthogonal projection of $x_u$ (resp. $y_u$) on $\ell_u$, and we let $T_u=|x_u'-y_u'|/v_u$ be the driving time of the $V$-path that drives from $x_u'$ to $y_u'$ using the road $(\ell_u,v_u)$.
Then, we define the pairs of points associated with the children $u1$ and $u2$ of $u$ in $\mathbb{B}$ as $(x_{u1},y_{u1})=(x_u,x_u')$ and ${(x_{u2},y_{u2})=(y_u',y_u)}$.
\end{itemize}
Taking for granted that $\sum_{u\in\mathbb{B}}T_u<\infty$, it is a neither immediate nor difficult deterministic consequence of the construction that all the $V$-paths from $x_u'$ to $y_u'$ for $u\in\mathbb{B}$ yield a $V$-path $\gamma:[0,T]\rightarrow\mathbb{R}^d$ from $x$ to $y$, with driving time $T=\sum_{u\in\mathbb{B}}T_u$.
The rest of the proof consists in showing that $\alpha$ and $\rho$ can be adjusted in order to obtain the desired bound on the driving times $\sum_{u\in\mathbb{B}}T_u$ provided by this construction.
Note that for every $u\in\mathbb{B}$ such that $x_u\neq y_u$, we have
\[T_u=\frac{|x_u'-y_u'|}{v_u}\leq\frac{|x_u-y_u|}{V_\rho(x_u,y_u,\alpha|x_u-y_u|)}.\]
Moreover, note that we have $|x_u-y_u|\leq\alpha^{|u|}\cdot|x-y|$ for all $u\in\mathbb{B}$, and
\[x_u,y_u\in\overline{B}\left(x,\frac{\alpha}{1-\alpha}\right)\cup\overline{B}\left(y,\frac{\alpha}{1-\alpha}\right)\subset(A)_{\alpha/(1-\alpha)}\quad\text{for all $u\in\mathbb{B}$,}\]
where $(A)_\varepsilon=\left\{x\in\mathbb{R}^d:d(x,A)\leq\varepsilon\right\}$ denotes the $\varepsilon$-neighbourhood of $A$.
Now, consider the following lemma, which provides a uniform lower bound on the speeds 
\[\left(V_\rho(x,y,\alpha|x-y|)\,;\,x\neq y\in(A)_{\alpha/(1-\alpha)}\right).\]

\begin{lem}\label{lem:unifcontrolVgeo}
For all sufficiently small $\alpha$, there exists $\rho=\rho(\alpha)\in{]0,r[}$ and a positive random variable $\zeta$, measurable with respect to the restriction of $\Pi_\rho$ to the set of roads that pass through $\overline{B}(0,2)$, such that
\[V_\rho(x,y,\alpha|x-y|)\geq\zeta\cdot|x-y|^{(d-1)/(\beta-1)}\cdot\ln\left(\frac{4}{\alpha|x-y|}\right)^{-1/(\beta-1)}\quad\text{for all $x\neq y\in(A)_{\alpha/(1-\alpha)}$.}\]
Moreover, there exists constants $C=C(\alpha)$ and $c=(\alpha)$ such that
\[\mathbb{P}\left(\zeta<\frac{1}{t}\right)\leq C\cdot\exp\left[-c\cdot t^{\beta-1}\right]\quad\text{for all $t\in\mathbb{R}_+^*$.}\]
\end{lem}

Taking this lemma for granted, we deduce that 
\[\begin{split}
\sum_{u\in\mathbb{B}}T_u&\leq\sum_{u\in\mathbb{B}}\frac{|x_u-y_u|}{\zeta\cdot|x_u-y_u|^{(d-1)/(\beta-1)}\cdot\ln(4/(\alpha|x_u-y_u|))^{-1/(\beta-1)}}\\
&=\zeta^{-1}\cdot\sum_{u\in\mathbb{B}}|x_u-y_u|^{(\beta-d)/(\beta-1)}\cdot\ln\left(\frac{4}{\alpha|x_u-y_u|}\right)^{1/(\beta-1)}=\zeta^{-1}\cdot\sum_{u\in\mathbb{B}}\phi(|x_u-y_u|),
\end{split}\]
where $\phi:s\mapsto s^{(\beta-d)/(\beta-1)}\cdot\ln(4/(\alpha s))^{1/(\beta-1)}$.
Now, recall that ${|x_u-y_u|\leq\alpha^{|u|}\cdot|x-y|\leq2}$ for all $u\in\mathbb{B}$.
Since the function ${t\mapsto t^{-(\beta-d)}\cdot\ln(t)}$ is nonincreasing over the interval $\left[e^{1/(\beta-d)},\infty\right[$, we deduce that for $\alpha$ small enough so that $4/(\alpha\cdot2)\geq e^{1/(\beta-d)}$, the function $\phi$ is nondecreasing over the interval $]0,2]$.
Then, summing over the height $n$ of $u$ in $\mathbb{B}$, we get
\[\begin{split}
\sum_{u\in\mathbb{B}}\phi(|x_u-y_u|)&\leq\sum_{u\in\mathbb{B}}\phi\left(\alpha^{|u|}\cdot|x-y|\right)\\
&=\sum_{n\geq0}2^n\cdot\phi\left(\alpha^n\cdot|x-y|\right)\\
&=\sum_{n\geq0}2^n\cdot\left(\alpha^n\cdot|x-y|\right)^{(\beta-d)/(\beta-1)}\cdot\ln\left(\frac{4}{\alpha^{n+1}\cdot|x-y|}\right)^{1/(\beta-1)}.
\end{split}\]
Next, using that $|x-y|\geq2$ we bound, for all $n\in\mathbb{N}$:
\[\begin{split}
\ln\left(\frac{4}{\alpha^{n+1}\cdot|x-y|}\right)&=(n+1)\ln\left(\frac{1}{\alpha}\right)+\ln\left(\frac{4}{|x-y|}\right)\\
&\leq(n+1)\cdot\frac{\ln(1/\alpha)}{\ln2}\cdot\ln\left(\frac{4}{|x-y|}\right)+(n+1)\ln\left(\frac{4}{|x-y|}\right)\\
&=\left(\frac{\ln(1/\alpha)}{\ln2}+1\right)\cdot(n+1)\cdot\ln\left(\frac{4}{|x-y|}\right)\\
&=:C\cdot(n+1)\cdot\ln\left(\frac{4}{|x-y|}\right).
\end{split}\]
Plugging this bound in the sum above, we obtain
\begin{eqnarray*}
\lefteqn{\sum_{u\in\mathbb{B}}\phi(|x_u-y_u|)}\\
&\leq&\sum_{n\geq0}2^n\cdot\left(\alpha^n\cdot|x-y|\right)^{(\beta-d)/(\beta-1)}\cdot\left(C\cdot(n+1)\cdot\ln\left(\frac{4}{|x-y|}\right)\right)^{1/(\beta-1)}\\
&=&C^{1/(\beta-1)}\cdot\sum_{n\geq0}\left(2\cdot\alpha^{(\beta-d)/(\beta-1)}\right)^n\cdot(n+1)^{1/(\beta-1)}\cdot|x-y|^{(\beta-d)/(\beta-1)}\cdot\ln\left(\frac{4}{|x-y|}\right)^{1/(\beta-1)}.
\end{eqnarray*}
Assuming further that $\alpha$ is small enough so that $2\cdot\alpha^{(\beta-d)/(\beta-1)}<1$, we get
\[\sum_{n\geq0}\left(2\cdot\alpha^{(\beta-d)/(\beta-1)}\right)^n\cdot(n+1)^{1/(\beta-1)}<\infty,\]
and we conclude that 
\[\sum_{u\in\mathbb{B}}\phi(|x_u-y_u|)\leq C'\cdot|x-y|^{(\beta-d)/(\beta-1)}\cdot\ln\left(\frac{4}{|x-y|}\right)^{1/(\beta-1)}\]
for some constant $C'=C'(\alpha)$.
This yields the result of the lemma, with $\Gamma=C'\cdot\zeta^{-1}$ (the tail estimate for $\Gamma$ comes from the tail estimate for $\zeta$).
To actually complete the proof, it remains to prove Lemma \ref{lem:unifcontrolVgeo}.

\begin{proof}[Proof of Lemma \ref{lem:unifcontrolVgeo}]
The proof combines a discretisation argument together with tail estimates on the random variables $V_\rho(x,y,r)$.
Let $\Delta=2(1+\alpha)/(1-\alpha)$ be the diameter of $(A)_{\alpha/(1-\alpha)}$, and set ${r_n=\alpha^n\cdot\Delta}$ for all $n\in\mathbb{N}$.
For each $n\in\mathbb{N}$, let $\left(B(w,r_n/2)\,;\,w\in W_n\right)$ be a covering of $(A)_{\alpha/(1-\alpha)}$ by balls of radius $r_n/2$, with centres $w\in(A)_{\alpha/(1-\alpha)}$ at least $r_n/2$ apart from each other.
By a measure argument, we have $\#W_n\leq C\cdot\alpha^{-dn}$ for some constant $C$.
Now, we claim that for every $x\neq y\in(A)_{\alpha/(1-\alpha)}$, we have
\begin{equation}\label{eq:discretisationgeo}
V_\rho(x,y,\alpha|x-y|)\geq\min\left\{V_\rho\left(w,w',\frac{r_{n+1}}{2}\right)\,;\,w\neq w'\in W_{n+1}:|w-w'|\leq\left(1+\alpha^2\right)\cdot r_{n-1}\right\},
\end{equation}
where $n\in\mathbb{N}^*$ is such that $r_n\leq|x-y|\leq r_{n-1}$.
Indeed, let ${w,w'\in W_{n+1}}$ be such that ${x\in B(w,r_{n+1}/2)}$ and ${y\in B(w',r_{n+1}/2)}$.
By the $1$-Lipschitz continuity of $|\cdot|$, we have
\[||w-w'|-|x-y||\leq|w-x|+|w'-y|\leq\frac{r_{n+1}}{2}+\frac{r_{n+1}}{2}=\alpha^2\cdot r_{n-1}.\]
In particular, it follows that $w\neq w'$, and $|w-w'|\leq\left(1+\alpha^2\right)\cdot r_{n-1}$.
Next, by the triangle inequality, we have ${\overline{B}(w,r_{n+1}/2)\subset\overline{B}(x,r_{n+1})}$ and ${\overline{B}(w',r_{n+1}/2)\subset\overline{B}(y,r_{n+1})}$.
Since ${r_{n+1}=\alpha\cdot r_n\leq\alpha|x-y|}$, it follows that $V_\rho(x,y,\alpha|x-y|)\geq V_\rho(w,w',r_{n+1}/2)$, which completes the proof of the claim.
Now, let us introduce the random variables
\[\zeta_n=\frac{n^{1/(\beta-1)}}{r_{n-1}^{(d-1)/(\beta-1)}}\cdot\min\left\{V_\rho\left(w,w',\frac{r_{n+1}}{2}\right)\,;\,w\neq w'\in W_{n+1}:|w-w'|\leq\left(1+\alpha^2\right)\cdot r_{n-1}\right\},\]
and set $\zeta_*=\inf_{n\geq1}\zeta_n$.
Note that these random variables are measurable with respect to the restriction of $\Pi_\rho$ to the set of roads that pass through $(A)_{\alpha/(1-\alpha)+r_2/2}$.
Moreover, we claim that for every $x\neq y\in(A)_{\alpha/(1-\alpha)}$, we have
\begin{equation}\label{eq:unifcontrolVrho}
V_\rho(x,y,\alpha|x-y|)\geq c\cdot\zeta_*\cdot|x-y|^{(d-1)/(\beta-1)}\cdot\ln\left(\frac{\Delta}{\alpha|x-y|}\right)^{-1/(\beta-1)},
\end{equation}
where $c=\ln(1/\alpha)^{1/(\beta-1)}$.
Indeed, given $x\neq y\in(A)_{\alpha/(1-\alpha)}$, fix an integer $n\in\mathbb{N}^*$ such that ${r_n\leq|x-y|\leq r_{n-1}}$.
By \eqref{eq:discretisationgeo}, we have
\[V_\rho(x,y,\alpha|x-y|)\geq\frac{r_{n-1}^{(d-1)/(\beta-1)}}{n^{1/(\beta-1)}}\cdot\zeta_n\geq\frac{|x-y|^{(d-1)/(\beta-1)}}{\log_{1/\alpha}(\Delta/(\alpha|x-y|))^{1/(\beta-1)}}\cdot\zeta_*,\]
and \eqref{eq:unifcontrolVrho} readily follows.
From now on, we assume that $\alpha$ is small enough so that
\[1+\frac{\alpha}{1-\alpha}+\frac{r_2}{2}=\frac{1+\alpha^2\cdot(1+\alpha)}{1-\alpha}\leq2\quad\text{and}\quad\Delta=\frac{2(1+\alpha)}{1-\alpha}\leq4.\]
It follows that the random variables $\zeta_n$ and $\zeta$ are measurable with the restriction of $\Pi_\rho$ to the set of roads that pass through $\overline{B}(0,2)$, and that
\[V_\rho(x,y,\alpha|x-y|)\geq c\cdot\zeta_*\cdot|x-y|^{(d-1)/(\beta-1)}\cdot\ln\left(\frac{4}{\alpha|x-y|}\right)^{-1/(\beta-1)}\quad\text{for all $x\neq y\in(A)_{\alpha/(1-\alpha)}$.}\]
To complete the proof of the lemma, it remains to show that $\alpha$ and $\rho$ can be adjusted in order to obtain the desired tail estimate for $\zeta=c\cdot\zeta_*$.
It suffices to show that there exists constants $C'=C'(\alpha)$ and $c'=c'(\alpha)$ such that
\[\mathbb{P}\left(\zeta_*<\frac{1}{t}\right)\leq C'\cdot\exp\left[-c'\cdot t^{\beta-1}\right]\quad\text{for all sufficiently large $t$.}\]
Fix $t\in[2,\infty[$.
By the union bound, we have
\[\mathbb{P}\left(\zeta_*<\frac{1}{t}\right)\leq\sum_{n\geq1}\mathbb{P}\left(\zeta_n<\frac{1}{t}\right).\]
Now, recall the definition of $\zeta_n$: the quantity $\mathbb{P}(\zeta_n<1/t)$ is the probability that there exists $w\neq w'\in W_{n+1}$ with $|w-w'|\leq\left(1+\alpha^2\right)\cdot r_{n-1}$ such that 
\[V_\rho\left(w,w',\frac{r_{n+1}}{2}\right)<\frac{1}{t}\cdot\frac{r_{n-1}^{(d-1)/(\beta-1)}}{n^{1/(\beta-1)}}=:v_n.\]
For fixed $w,w'$ as above, the quantity $\mathbb{P}(V_\rho(w,w',r_{n+1}/2)<v_n)$ is the probability that there is no road of $\Pi_\rho$ with speed at least $v_n$ that passes through $\overline{B}(w,r_{n+1}/2)$ and $\overline{B}(w',r_{n+1}/2)$.
Therefore, we have $\mathbb{P}(V_\rho(w,w',r_{n+1}/2)<v_n)=e^{-\lambda}$, where
\begin{equation}\label{eq:parameterlambda}
\lambda=\upsilon_{d-1}^{-1}\cdot\mu\left(\left\langle\overline{B}\left(w,\frac{r_{n+1}}{2}\right)\,;\,\overline{B}\left(w',\frac{r_{n+1}}{2}\right)\right\rangle\middle\backslash\overline{B}(z,\rho)\right)\cdot(\beta-1)\cdot\int_{v_n}^1v^{-\beta}\mathrm{d}v.
\end{equation}
The constants $\upsilon_{d-1}^{-1}$ and $(\beta-1)$ come from the precise definition of the intensity measure of $\Pi$, see \eqref{eq:normgeo} for a reminder.
Since $v_n\leq1/t\leq1/2$, we have
\[(\beta-1)\cdot\int_{v_n}^1v^{-\beta}\mathrm{d}v=v_n^{-(\beta-1)}-1\geq v_n^{-(\beta-1)}-(2v_n)^{-(\beta-1)}=\left(1-2^{-(\beta-1)}\right)\cdot v_n^{-(\beta-1)}.\]
Now, recall the estimate \eqref{eq:twopointmugeo}.
On the one hand, we have
\[\begin{split}
\mu\left\langle\overline{B}\left(w,\frac{r_{n+1}}{2}\right)\,;\,\overline{B}\left(w',\frac{r_{n+1}}{2}\right)\right\rangle&\geq c\cdot\frac{(r_{n+1}/2)^{d-1}\cdot(r_{n+1}/2)^{d-1}}{|w-w'|^{d-1}}\\
&\geq c\cdot\frac{(r_{n+1}/2)^{d-1}\cdot(r_{n+1}/2)^{d-1}}{\left(\left(1+\alpha^2\right)\cdot r_{n-1}\right)^{d-1}}\\
&=c\cdot\frac{\left(\alpha^2\middle/2\right)^{d-1}\cdot\left(\alpha^2\middle/2\right)^{d-1}}{\left(1+\alpha^2\right)^{d-1}}\cdot r_{n-1}^{d-1}=:c'\cdot r_{n-1}^{d-1}.
\end{split}\]
On the other hand, we have
\[\begin{split}
\mu\left\langle\overline{B}\left(w,\frac{r_{n+1}}{2}\right)\,;\,\overline{B}\left(w',\frac{r_{n+1}}{2}\right)\,;\,\overline{B}(z,\rho)\right\rangle&\leq\mu\left\langle\overline{B}\left(w,\frac{r_{n+1}}{2}\right)\,;\,\overline{B}(z,\rho)\right\rangle\\
&\leq C\cdot\frac{(r_{n+1}/2)^{d-1}\cdot\rho^{d-1}}{|w-z|^{d-1}}\\
&=C\cdot\frac{\left(\alpha^2\middle/2\right)^{d-1}\cdot\rho^{d-1}}{|w-z|^{d-1}}\cdot r_{n-1}^{d-1}.
\end{split}\]
Now, we want to adjust $\alpha$ and $\rho$ so as to obtain
\[\mu\left(\left\langle\overline{B}\left(w,\frac{r_{n+1}}{2}\right)\,;\,\overline{B}\left(w',\frac{r_{n+1}}{2}\right)\right\rangle\middle\backslash\overline{B}(z,\rho)\right)\geq\frac{c'}{2}\cdot r_{n-1}^{d-1}\]
uniformly in $w,w'$ as above.
It suffices to ensure that
\[C\cdot\frac{\left(\alpha^2\middle/2\right)^{d-1}\cdot\rho^{d-1}}{|w-z|^{d-1}}\leq\frac{c'}{2}\]
uniformly in $w\in(A)_{\alpha/(1-\alpha)}$.
First, we assume that $\alpha$ is small enough so that
\[\frac{\alpha}{1-\alpha}\leq\frac{r}{2}.\]
This way, we have $|w-z|\geq r/2$ for all $w\in(A)_{\alpha/(1-\alpha)}$.
Now that $\alpha$ has been fixed, we choose $\rho$ small enough so that
\[C\cdot\frac{\left(\alpha^2\middle/2\right)^{d-1}\cdot\rho^{d-1}}{(r/2)^{d-1}}\leq\frac{c'}{2}.\]
We can finally come back to our tail estimate: for the parameter $\lambda$ of \eqref{eq:parameterlambda}, we obtain the bound
\[\lambda\geq\upsilon_{d-1}^{-1}\cdot\frac{c'}{2}\cdot r_{n-1}^{d-1}\cdot\left(1-2^{-(\beta-1)}\right)\cdot v_n^{-(\beta-1)}=\upsilon_{d-1}^{-1}\cdot\frac{c'}{2}\cdot\left(1-2^{-(\beta-1)}\right)\cdot t^{\beta-1}\cdot n=:c''\cdot t^{\beta-1}\cdot n,\]
uniformly in $w,w'$ as above.
By the union bound, we deduce that
\[\mathbb{P}\left(\zeta_n<\frac{1}{t}\right)\leq\#W_{n+1}^2\cdot\exp\left[-c''\cdot t^{\beta-1}\cdot n\right].\]
Finally, recall that $\#W_{n+1}\leq C\cdot\alpha^{-d(n+1)}$, hence
\[\mathbb{P}\left(\zeta_n<\frac{1}{t}\right)\leq C^2\cdot\alpha^{-2d(n+1)}\cdot\exp\left[-c''\cdot t^{\beta-1}\cdot n\right]=C^2\cdot\exp\left[2d\ln(1/\alpha)\cdot(n+1)-c''\cdot t^{\beta-1}\cdot n\right].\]
It follows that
\[\mathbb{P}\left(\zeta_n<\frac{1}{t}\right)\leq C^2\cdot\exp\left[-c''/2\cdot t^{\beta-1}\cdot n\right]\quad\text{for all sufficiently large $t$,}\]
uniformly in $n\in\mathbb{N}^*$.
Summing this over $n$, we obtain the desired tail estimate: we get
\[\mathbb{P}\left(\zeta_*<\frac{1}{t}\right)\leq\sum_{n\geq1}C^2\cdot\exp\left[-c''/2\cdot t^{\beta-1}\cdot n\right]=C^2\cdot\frac{\exp\left[-c''/2\cdot t^{\beta-1}\right]}{1-\exp\left[-c''/2\cdot t^{\beta-1}\right]}\]
for all sufficiently large $t$.
\end{proof}
\end{proof}

\bibliographystyle{siam}
\bibliography{biblio.bib}

\end{document}